\documentclass{amsart}

\usepackage{amsfonts,amssymb,amsmath,amscd}

\theoremstyle{plain}
\newtheorem{thm}{Theorem}[section]
\newtheorem{lemma}[thm]{Lemma}
\newtheorem{prop}[thm]{Proposition}

\newtheorem{cor}[thm]{Corollary}
\newtheorem*{introcor}{Corollary}

\theoremstyle{remark}

\theoremstyle{definition}
\newtheorem{defn}[thm]{Definition}
\newtheorem*{introdefn}{Definition}
\newtheorem{example}[thm]{Example}
\newtheorem{question}[thm]{Question}

\newcommand{\Hom}{\operatorname{Hom}}

\newcommand{\M}{\mathcal{M}}
\newcommand{\N}{\mathcal{N}}
\newcommand{\E}{\mathcal{E}}
\newcommand{\F}{\mathcal{F}}
\newcommand{\G}{\mathcal{G}}

\newcommand{\tenE}{-\otimes_{\mathcal{O}_{X}}\mathcal{E}}
\newcommand{\Funct}{{\sf Funct}_{k}({\sf Qcoh }X,{\sf Qcoh }Y)}

\begin{document}
\title{The Eilenberg-Watts Theorem over Schemes}
\author{A. Nyman}
\address{Department of Mathematics, 516 High St, Western Washington University, Bellingham, WA 98225-9063}
\email{adam.nyman@wwu.edu} 
\date{\today}
\thanks{The author was partially supported by the National Security Agency under grant NSA H98230-05-1-0021.}

\begin{abstract}
We study obstructions to a direct limit preserving right exact
functor $F$ between categories of quasi-coherent sheaves on schemes
being isomorphic to tensoring with a bimodule.  When the domain
scheme is affine, or if $F$ is exact, all obstructions vanish and we recover the Eilenberg-Watts
Theorem.  This result is crucial to
the proof that the noncommutative Hirzebruch surfaces constructed
in \cite{blowup} are noncommutative $\mathbb{P}^{1}$-bundles in
the sense of \cite{ncruled}.
 \end{abstract}

\maketitle \tableofcontents

\section{Introduction}

In this paper we describe a version of the Eilenberg-Watts Theorem over
schemes. In order to motivate our results we first recall the Eilenberg-Watts Theorem proved independently by Eilenberg \cite{E} and Watts \cite{W}:

\begin{thm} \label{thm.classical}
Let $k$ be a commutative ring, let $R$ and $S$ be $k$-algebras and let ${\sf Mod }R$ (resp. ${\sf Mod }S$) denote the category of right $R$-modules (resp. right $S$-modules).  If $F:{\sf Mod }R  \rightarrow {\sf Mod }S$ is a $k$-linear
right exact functor commuting with direct limits, then there
exists a $k$-central $R-S$-bimodule $M$ such that $F \cong
-\otimes_{R}M$.
\end{thm}
The bimodule $M$ in the previous theorem is easy to describe.
$M=F(R)$ as a right-module, and its left-module structure is
defined as follows:  for each $r \in R$, we let $\phi_{r} \in
\operatorname{Hom}_{R}(R,R)$ denote left multiplication by $r$.
For $m \in M$, we define $r \cdot m := F(\phi_{r})m$.

It is natural to ask if such a result holds when the categories
${\sf Mod }R$ and ${\sf Mod }S$ are replaced by categories of
quasi-coherent sheaves on schemes $X$ and $Y$, ${\sf Qcoh }X$ and
${\sf Qcoh }Y$.  In order to precisely pose the question in this
context, we need to introduce some notation.  To this end, if $Z$
is a scheme, $X$ and $Y$ are $Z$-schemes, $\E$ is a quasi-coherent
$\mathcal{O}_{X \times_{Z} Y}$-module, and the projections $X
\times_{Z} Y \rightarrow X,Y$ are denoted
${\operatorname{pr}}_{1}$ and ${\operatorname{pr}}_{2}$, we define
$$
\mathcal{M} \otimes_{\mathcal{O}_{X}} \E := \operatorname{pr}_{2*}
(\operatorname{pr}_{1}^{*}\M \otimes_{\mathcal{O}_{X \times_{Z}
Y}} \E).
$$
We make the further assumption that
$$
-\otimes_{\mathcal{O}_{X}}\E:{\sf Qcoh }X \rightarrow {\sf Qcoh }Y,
$$
which is automatic if $X \rightarrow Z$ is
quasi-compact, separated and $Z$ is affine.

Now let $k$ be a commutative ring and let $Z = \operatorname{Spec }k$.  Although the functor
$-\otimes_{\mathcal{O}_{X}}\E:{\sf Qcoh }X \rightarrow {\sf Qcoh }Y$ is not always right exact, it is
{\it locally} right exact in the sense that if $u:U \rightarrow X$ is an open immersion from an affine scheme to
$X$, then $u_{*}(-)\otimes_{\mathcal{O}_{X}}\E:{\sf Qcoh }U \rightarrow {\sf Qcoh }Y$ is right exact
(see the proof that (\ref{eqn.pullback}) is an isomorphism in Section 3).  This suggests that a natural generalization of Theorem \ref{thm.classical} to the case of functors
between quasi-coherent sheaves on schemes would involve a characterization of locally right exact $k$-linear functors
$F: {\sf Qcoh }X \rightarrow {\sf Qcoh }Y$ commuting with direct limits.  However, since (globally) right exact
functors $F: {\sf Qcoh }X \rightarrow {\sf Qcoh }Y$ appear naturally in the construction of certain non-commutative ruled surfaces
(see the remark following Theorem \ref{thm.almostend} for more details), and since our motivation for studying generalizations
of Theorem \ref{thm.classical} comes from attempts to better understand these constructions, we specialize our study to right exact functors.
It is thus natural for us to ask the following
\begin{question}
Let $F:{\sf Qcoh }X \rightarrow {\sf Qcoh }Y$ denote a $k$-linear,
right exact functor commuting with direct limits.  Is $F$
isomorphic to tensoring with a bimodule, i.e. does there exist an
object $\E$ of ${\sf Qcoh }\mathcal{O}_{X \times_{Z} Y}$ such that
$F \cong \tenE$?
\end{question}
When $X$ is affine, we recall in Proposition \ref{prop.wattone}
that the answer to this question is yes.  Proposition
\ref{prop.wattone} follows from a generalization of Theorem
\ref{thm.classical} proved in \cite{ns}.

In general, the answer to this question is {\it no}, as the
following example illustrates.
\begin{example} \cite[Example 3.1.3]{ncruled} \label{example.cohomology1}
Suppose $k$ is a field, $X=\mathbb{P}_{k}^{1}$ and
$Y=Z=\operatorname{Spec }k$. If $F={H}^{1}(X,-)$,
then $F$ is $k$-linear, right exact, and commutes with
direct limits.  However, as we will prove in Proposition
\ref{prop.sheafid}, $F$ is not isomorphic to tensoring with a
bimodule.
\end{example}
The purpose of this paper is to study the obstructions to a
$k$-linear right exact functor $F: {\sf Qcoh }X \rightarrow {\sf
Qcoh }Y$ which commutes with direct limits being isomorphic to tensoring with a bimodule.  In order to state our
main result, we introduce notation and conventions which will be
employed throughout the paper.

We let $k$ denote a commutative ring, $Z={\operatorname{Spec}}
k$ and we assume all schemes and products of
schemes are over $Z$. We assume $X$ is a quasi-compact and
separated scheme and $Y$ is a separated scheme.

We note that the category
$$
{\sf Funct}({\sf Qcoh }X,{\sf Qcoh }Y)
$$
of functors from ${\sf Qcoh }X$ to ${\sf Qcoh }Y$ is abelian, and
we denote the full subcategory of $k$-linear functors (see Section
2 for a precise definition of {\it $k$-linear functor}) in ${\sf
Funct}({\sf Qcoh }X,{\sf Qcoh }Y)$ by
$$
{\sf Funct}_{k}({\sf Qcoh }X,{\sf Qcoh }Y).
$$
The category ${\sf Funct}_{k}({\sf Qcoh }X,{\sf Qcoh }Y)$ is
abelian as well.  We denote the full subcategory of ${\sf
Funct}_{k}({\sf Qcoh }X,{\sf Qcoh }Y)$ consisting of right exact
functors commuting with direct limits by
$$
{\sf Bimod}_{k}(X-Y).
$$
We denote the full subcategory of ${\sf Bimod}_{k}(X-Y)$
consisting of functors which take coherent objects to coherent
objects by
$$
{\sf bimod}_{k}(X-Y).
$$

The following definition, studied in Section
\ref{section.totallyglobal}, plays a central role in our theory.
\begin{introdefn}
An object $F$ of $\Funct$ is {\it totally global} if for every
open immersion $u:U \rightarrow X$ with $U$ affine, $Fu_{*}=0$.
\end{introdefn}
The functor $F$ in Example \ref{example.cohomology1} is totally global.

In order to generalize the Eilenberg-Watts Theorem, we first study an
assignment, which we call the {\it Eilenberg-Watts functor},
$$
W:{\sf Bimod}_{k}(X-Y) \rightarrow {\sf Qcoh }X  \times  Y,
$$
whose construction was sketched in \cite[Lemma 3.1.1]{ncruled}.
We prove that it is functorial (Subsection \ref{subsection.defn}),
left-exact (Proposition \ref{prop.lex}), compatible with affine
localization (Proposition \ref{prop.loc}), and has the property
that if $F \cong -\otimes_{\mathcal{O}_{X}}\F$ then $W(F) \cong \F$
(Proposition \ref{prop.sheafid}).  It follows from Propostion
\ref{prop.wattone} that if $X$ is affine, then $F \cong
-\otimes_{\mathcal{O}_{X}}W(F)$.

We then work towards our main result, established in Section
\ref{section.wattstheory}:

\begin{thm} \label{thm.main}
If $F \in {\sf Bimod}_{k}(X-Y)$, then there exists a
natural transformation
$$
\Gamma_{F}:F \longrightarrow -\otimes_{\mathcal{O}_{X}}W(F)
$$
such that $\operatorname{ker }\Gamma_{F}$
and $\operatorname{cok }\Gamma_{F}$ are totally global (Corollary
\ref{cor.totallyglobal}).  Furthermore, $\Gamma_{F}$ is an isomorphism if
\begin{enumerate}
\item{} $X$ is affine or

\item{} $F$ is exact (Corollary \ref{thm.cokvanish}) or

\item{}  $F \cong -\otimes_{\mathcal{O}_{X}}\F$ for some object $\F$
in ${\sf Qcoh }X \times Y$ (Proposition \ref{prop.recover}).
\end{enumerate}
\end{thm}
As a consequence, if $F \in {\sf Bimod}_{k}(X-Y)$, then
$-\otimes_{\mathcal{O}_{X}}W(F)$ serves as the ``best"
approximation of $F$ by tensoring with a bimodule in the following sense (Corollary
\ref{cor.approx}):

\begin{introcor}
Let $\mathcal{F}'$ be an object of ${\sf Qcoh }X \times Y$ and
suppose $F' := -\otimes_{\mathcal{O}_{X}}\F'$ is an object in
${\sf Bimod}_{k}(X-Y)$.  If $\Phi:F \rightarrow F'$ is a morphism
in ${\sf Bimod}_{k}(X-Y)$, then $\Phi$ factors through
$\Gamma_{F}$.
\end{introcor}

In order to describe necessary and sufficient conditions for $\Gamma_{F}$ to be an isomorphism,
we introduce some notation:  Let $\{U_{i}\}$ be a finite affine open cover of $X$, let $U_{ij} : = U_{i} \cap U_{j}$,
and let $u_{i}:U_{i} \rightarrow X$ and $u_{ij}^{i}:U_{ij} \rightarrow U_{i}$ denote inclusions.  If $\mathcal{M}$
is an object of ${\sf Qcoh }X$, there is a canonical morphism (defined by (\ref{eqn.delta1}))
$$
\delta_{\M}:\oplus_{i}u_{i*}u_{i}^{*}\M \longrightarrow
\oplus_{i<j}u_{i*}u_{ij*}^{i}u_{ij}^{i*}u_{i}^{*}\M
$$
which is essentially the beginning of the sheafified $\check{\mbox{C}}$ech complex.  We prove the following (Corollary \ref{cor.gammaisom}):

\begin{thm} \label{thm.gammaisom}
If $F \in {\sf Bimod}_{k}(X-Y)$ then $\Gamma_{F}$ is an isomorphism if and only if
\begin{enumerate}

\item{} for all flat objects $\mathcal{L}$ in ${\sf Qcoh }X$, the canonical map $F\operatorname{ker }\delta_{\mathcal{L}} \rightarrow \operatorname{ker }F \delta_{\mathcal{L}}$ is an isomorphism, and

\item{} $-\otimes_{\mathcal{O}_{X}}W(F)$ is right exact.

\end{enumerate}
\end{thm}
The first item in Theorem \ref{thm.gammaisom} says that $F$ must be close to being flat-acyclic, hence
close to being a tensor product.  The second item in Theorem \ref{thm.gammaisom} implies that
$-\otimes_{\mathcal{O}_{X}}W(F)$ is in ${\sf Bimod}_{k}(X-Y)$.

Theorem \ref{thm.main} suggests that in order to obtain more precise information about objects in ${\sf Bimod}_{k}(X-Y)$
one must have a better understanding of the structure of totally global functors.
While a general structure theory of totally global functors seems far off,
we begin a very specialized investigation of this subject in Section \ref{section.structure}.  In particular,
we classify totally global functors in ${\sf
bimod}_{k}({\mathbb{P}^{1}}-{\mathbb{P}^{0}})$ when $k$ is
an algebraically closed field.  Our result in this direction is the following (Corollary \ref{cor.structureoftg}):

\begin{thm} \label{thm.almostend}
If $F \in {\sf bimod}_{k}(\mathbb{P}^{1}-\mathbb{P}^{0})$ is totally global, then $F$ is a direct sum of
cohomologies, i.e. there exist integers $m,
n_{i} \geq 0$ such that
$$
F \cong \oplus_{i=-m}^\infty {H}^{1}(\mathbb{P}^{1},(-)(i))^{\oplus n_{i}}.
$$
\end{thm}

We conclude the introduction by mentioning an application
of Theorem \ref{thm.main}(2).  In \cite{blowup}, Ingalls and
Patrick show that the blow-up of a noncommutative weighted
projective space is a noncommutative Hirzebruch surface in an
appropriate sense. More precisely, they show that the blow-up is a
projectivization of an exact functor $F:{\sf Qcoh }\mathbb{P}^{1}
\rightarrow {\sf Qcoh }\mathbb{P}^{1}$ which commutes with
direct limits.  It follows from Theorem \ref{thm.main}(2) that $F
\cong -\otimes_{\mathcal{O}_{\mathbb{P}^{1}}} \mathcal{F}$ where
$\mathcal{F}$ is a quasi-coherent
$\mathcal{O}_{\mathbb{P}^{1}\times \mathbb{P}^{1}}$-module.  This provides a crucial step in the proof that the noncommutative Hirzeburch surface Ingalls and Patrick construct is
a noncommutative ruled surface in the sense of \cite{ncruled}.
\vskip .2in
{\it An apology for including proofs that diagrams commute:}  This paper contains a number of ``technical" proofs that various
diagrams commute.  While some readers may frown upon the practice of including such proofs, we thought it wise to include them for the following reasons:

First, we are interested in proving a version of Theorem \ref{thm.main} in which $Y$ is a non-commutative space (see \cite[Section 1.2]{vblowup} for
the definition of quasi-scheme, which is what we mean by non-commutative space).  The proof of such a result will require the proof that diagrams similar to those
in this paper commute.  Since local arguments are often unavailable in the non-commutative setting, it will be important
to have a careful record of which proofs of commutativity can be reduced to arguments global on $Y$ (which should carry over without change to the non-commutative setting),
and which are local on $Y$ (which will have to be replaced by global arguments in the non-commutative setting).

Second, it is sometimes very difficult, even for extremely experienced mathematicians, to decide which diagrams
commute for elementary reasons and which commute for deeper reasons.  This fact is evidenced by the need for \cite{conrad} to fill gaps
in \cite{hartres}.  The gaps were not widely recognized as substantial until many years after the publication of \cite{hartres}.  Although the
diagrams appearing in this paper are far less complicated than those studied in \cite{conrad}, we felt it important to save the skeptical reader
from reconstructing the often tedious arguments on their own.
\vskip .2in
{\it Acknowledgements:}  I am grateful to S. Paul Smith for numerous helpful conversations,
for clarifying the proof of Proposition \ref{prop.wattone} and for allowing me to include
some of his results in Section \ref{section.structure}.  I am also grateful to Daniel Chan for showing me how to
generalize an earlier version of Theorem \ref{thm.main}(2).

Finally, I thank Quan Shui Wu for
hosting me at Fudan University during the 2006-2007 academic
year, during which parts of this paper were written.

\section{The Eilenberg-Watts Theorem} \label{section.wattstheorem}
The purpose of this section is to recall the naive generalization
of the Eilenberg-Watts Theorem that holds when the domain scheme is affine
(see Proposition \ref{prop.wattone} for a precise
formulation of this statement). The result is used implicitly in
\cite[Example 3.1.3]{ncruled}. We first recall the following
definition, which is invoked in the statement of Proposition
\ref{prop.wattone}.

\begin{defn}  Recall that $Z=\operatorname{Spec }k$, let $f:X \rightarrow Z$ denote the $k$-scheme structure
map for $X$ and let $g:Y \rightarrow Z$ denote the $k$-scheme structure map for $Y$.
An element $F \in {\sf Funct}({\sf Qcoh }X,{\sf Qcoh
}Y)$ is {\it $k$-linear} if the diagram
$$
\begin{CD}
k \times \Hom_{\mathcal{O}_{X}}(\M,\N) & \rightarrow & k \times \Hom_{\mathcal{O}_{Y}}(F\M,F\N) \\
@VVV @VVV \\
\Hom_{\mathcal{O}_{X}}(\M,\N) & \rightarrow &
\Hom_{\mathcal{O}_{Y}}(F\M,F\N)
\end{CD}
$$
whose horizontal arrows are induced by $F$, and whose vertical
arrows are induced by the $k$-module structure on $\Hom_{\mathcal{O}_{X}}(\M,\N)$ coming from global sections of the structure maps
$\mathcal{O}_{Z} \rightarrow f_{*}\mathcal{O}_{X}$ and
$\mathcal{O}_{Z} \rightarrow g_{*}\mathcal{O}_{Y}$ respectively,
commutes.
\end{defn}

\begin{prop}
\label{prop.wattone} \cite[Example 4.2]{ns} If $X$ is affine, then
the inclusion functor
$$
{\sf Qcoh}(X  \times  Y) \rightarrow {\sf Bimod}_{k}(X - Y)
$$
induced by the assignment $\F \mapsto
-\otimes_{\mathcal{O}_{X}}\F$ is an equivalence of categories.
\end{prop}

Proposition \ref{prop.wattone} follows from a general form of
the Eilenberg-Watts Theorem \cite[Theorem 3.1]{ns} characterizing right exact functors $F:{\sf Mod }R
\rightarrow {\sf A}$ commuting with direct limits, where $R$ is a
ring, ${\sf Mod }R$ denotes the category of right $R$-modules and
$\sf A$ is an abelian category.

We recall the proof that the inclusion functor in Proposition
\ref{prop.wattone} is essentially surjective since we will invoke
it in the sequel.  We first construct an object, $\F$, of ${\sf Qcoh }(X
\times Y)$ whose image is isomorphic to $F \in {\sf Bimod}_{k}(X -
Y)$ as follows: Let $X = \operatorname{Spec }R$, and let $U
\subset Y$ be affine open.  We first define an $R \otimes_k
\mathcal{O}_{Y}(U)$-module, $N$.  We let $N$ have underlying set
and right-module structure equal to $F(\mathcal{O}_{X})(U)$.
We let
$$
\mu_{r} \in
\operatorname{Hom}_{\mathcal{O}_{X}}(\mathcal{O}_{X},\mathcal{O}_{X})
$$
correspond to multiplication by $r \in R \cong
\Gamma(X,\mathcal{O}_{X})$, and we give $N$ an $R$-module
structure by defining $r \cdot n := F(\mu_{r})(U) n$ for $n \in
F(\mathcal{O}_{X})(U)$. It remains to show that $N$ is
$k$-central, but this follows directly from the fact that $F$ is
$k$-linear. We conclude that $N$ is an $R \otimes_{k} \mathcal{O}_{Y}(U)$-module,
hence corresponds to a quasi-coherent $\mathcal{O}_{X \times
U}$-module, $\F_U$.  It is straightforward to check that the
sheaves $\F_U$ glue to give a quasi-coherent $\mathcal{O}_{X
 \times  Y}$-module, which we call $\mathcal{F}$.

We next construct an isomorphism $\Theta: -\otimes_{\mathcal{O}_{X}}\F
\rightarrow F$ as follows:  Let $\mathcal{M}$ be an
$\mathcal{O}_{X}$-module and let $U \subset Y$ be an affine open
subset.  We define a morphism
$\Theta_{\mathcal{M}}(U):\mathcal{M}\otimes_{\mathcal{O}_{X}} \F
(U) \rightarrow F\M(U)$.  To this end, we note that
\begin{eqnarray*}
\mathcal{M}\otimes_{\mathcal{O}_{X}} \F (U) & = & \operatorname{pr}_{2*}(\operatorname{pr}_{1}^{*} \M \otimes_{\mathcal{O}_{X  \times  Y}} \F)(U) \\
& = & (\operatorname{pr}_{1}^{*} \M \otimes_{\mathcal{O}_{X  \times  Y}} \F)(X  \times  U) \\
& \cong & \M(X) \otimes_{R} F(\mathcal{O}_{X})(U).
\end{eqnarray*}
Hence, in order to define $\Theta_{\mathcal{M}}(U)$, it suffices
to construct an $\mathcal{O}_{Y}(U)$-module map $w: \M(X)
\otimes_{R} F(\mathcal{O}_{X})(U) \rightarrow F\mathcal{M}(U)$.
This is constructed as in the proof of the Eilenberg-Watts Theorem, as
follows. Suppose $m \in \M(X)$, $n \in F(\mathcal{O}_{X})(U)$, $r
\in R$, and
$$
\mu_{m} \in
\operatorname{Hom}_{\mathcal{O}_{X}}(\mathcal{O}_{X},\M)
$$
corresponds to the homomorphism in
$\operatorname{Hom}_{R}(R,\mathcal{M}(X))$ sending $1$ to $m$.  Then
$$
F(\mu_{m}) \in \Hom_{\mathcal{O}_{Y}}(F(\mathcal{O}_{X}),F\M)
$$
and
\begin{eqnarray*}
F(\mu_{mr})(U)(n) & = & F(\mu_{m} \mu_{r})(U)(n) \\
& = & F(\mu_{m})(U) F(\mu_{r})(U) (n) \\
& = & F(\mu_{m})(U) (rn).
\end{eqnarray*}
Hence, the function $w(m \otimes n):=F(\mu_{m})(U)(n)$ extends to
a well defined homomorphism of $\mathcal{O}_{Y}(U)$-modules $w:
\M(X) \otimes_{R} F(\mathcal{O}_{X})(U) \rightarrow
F\mathcal{M}(U)$, which in turn corresponds to a map of
$\mathcal{O}_{Y}(U)$-modules
$\Theta_{\mathcal{M}}(U):\mathcal{M}\otimes_{\mathcal{O}_{X}} \F
(U) \rightarrow F\M(U)$.  It is straightforward to show that the
maps $\Theta_{\M}(U)$ glue to give a map of
$\mathcal{O}_{Y}$-modules
$$
\Theta_{\M}:\M \otimes_{\mathcal{O}_{X}}\F \rightarrow F\M
$$
and that $\Theta_{\M}$ is an isomorphism which is natural in $\M$.

In the sequel, we will often refer to $\Theta$ as the canonical
isomorphism in the proof of Proposition \ref{prop.wattone}.

\section{Basechange, the Projection Formula, and Compatibilities}
\label{section.basechange} Our construction of the Eilenberg-Watts functor
and our proof of Theorem \ref{thm.main} depends, in a fundamental way,
on the existence and properties of two canonical isomorphisms
which are constructed using basechange and the projection formula.
The purpose of this section is to describe these isomorphisms as
well as several fundamental compatibilities involving them.

Throughout this section, we let $U$ denote an affine scheme, we let $u:U \rightarrow X$ denote an open
immersion, we let $v=u \times \operatorname{id}_{Y}$, and
we let $p,q:U \times Y \rightarrow U,Y$ denote projections.

We begin with some preliminary observations. We note that the
diagram
$$
\begin{CD}
U \times Y & \overset{p}{\longrightarrow} & U \\
@V{v}VV @VV{u}V \\
X \times Y & \underset{\operatorname{pr}_{1}}{\longrightarrow} & X
\end{CD}
$$
is a fiber square.  We claim the basechange and projection formula
morphisms
$$
\operatorname{pr}_{1}^{*}u_{*} \longrightarrow
{v}_{*}{v}^{*}\operatorname{pr}_{1}^{*}u_{*}
\overset{\cong}{\longrightarrow} {v}_{*}p^{*}u^{*}u_{*}
\longrightarrow  {v}_{*}p^{*}
$$
and
$$
{v}_{*}- \otimes_{\mathcal{O}_{X \times Y}}- \rightarrow
{v}_{*}{v}^{*}({v}_{*}- \otimes_{\mathcal{O}_{X \times Y}}-)
\overset{\cong}{\rightarrow}  {v}_{*}({v}^{*}{v}_{*}-
\otimes_{\mathcal{O}_{U \times Y}}{v}^{*}-)  \rightarrow {v}_{*}(-
\otimes_{\mathcal{O}_{U \times Y}}{v}^{*}-)
$$
induced by unit and counit morphisms of $(u^{*},u_{*})$ and
$(v^{*},v_{*})$, and by the distributivity of pullbacks over tensor products, are isomorphisms.  To this end, we note that it
suffices to prove that they are isomorphisms over subsets of the
form $V \times W$ where $V \subset X$ and $W \subset Y$ are open
affine subsets. This reduces the claim to a straightforward affine
computation, which we omit.

Let $\mathcal{E} \in {\sf Qcoh}U \times Y$ and $\F \in {\sf Qcoh
}X \times Y$.  We define canonical isomorphisms
\begin{equation} \label{eqn.pullback2}
u^{*}(-) \otimes_{\mathcal{O}_{U}} \E  \longrightarrow -
\otimes_{\mathcal{O}_{X}} v_{*}\E
\end{equation}
and
\begin{equation} \label{eqn.pullback}
u_{*}(-) \otimes_{\mathcal{O}_{X}}\F \longrightarrow
-\otimes_{\mathcal{O}_{U}}v^{*}\F,
\end{equation}
natural in $\E$ and $\F$, as follows:  The map
(\ref{eqn.pullback2}) is defined to be the composition
\begin{eqnarray*}
u^{*}(-) \otimes_{\mathcal{O}_{U}}\E & \overset{=}{\longrightarrow} & q_{*}(p^{*}u^{*}- \otimes_{\mathcal{O}_{U \times Y}} \E) \\
& \overset{\cong}{\longrightarrow} & q_{*}(v^{*}\operatorname{pr}_{1}^{*}- \otimes_{\mathcal{O}_{U \times Y}} \E) \\
& \overset{=}{\longrightarrow} &
\operatorname{pr}_{2*}v_{*}(v^{*}\operatorname{pr}_{1}^{*}-
\otimes_{\mathcal{O}_{U \times Y}} \E) \\
& \overset{\cong}{\longrightarrow} &
\operatorname{pr}_{2*}(\operatorname{pr}_{1}^{*}-
\otimes_{\mathcal{O}_{X \times Y}}
v_{*}\E) \\
& \overset{=}{\longrightarrow} & -\otimes_{\mathcal{O}_{X}}
v_{*}\E
\end{eqnarray*}
where the second morphism comes from the equality
$\operatorname{pr}_{1} v = up$ and the fourth morphism is the
projection formula.

We define the map (\ref{eqn.pullback}) as the composition
\begin{eqnarray*}
u_{*}(-)\otimes_{\mathcal{O}_{X}}\F & \overset{=}{\longrightarrow} & \operatorname{pr}_{2*}(\operatorname{pr}_{1}^{*}u_{*}-\otimes_{\mathcal{O}_{X \times Y}} \mathcal{F}) \\
& \overset{\cong}{\longrightarrow} & \operatorname{pr}_{2*}({v}_{*}p^{*}- \otimes_{\mathcal{O}_{X \times Y}} \mathcal{F}) \\
& \overset{\cong}{\longrightarrow} & \operatorname{pr}_{2*}{v}_{*}(p^{*}-\otimes_{\mathcal{O}_{U \times Y}} {v}^{*}\mathcal{F}) \\
& \overset{=}{\longrightarrow} &
q_{*}(p^{*}-\otimes_{\mathcal{O}_{U \times Y}} {v}^{*}\mathcal{F})
\end{eqnarray*}
where the second morphism is basechange and the third morphism is
the projection formula.

Naturality of (\ref{eqn.pullback2}) and (\ref{eqn.pullback})
follows from naturality of basechange and the projection formula.

The remainder of this section is devoted to the proof that
(\ref{eqn.pullback2}) and (\ref{eqn.pullback}) satisfy three
compatibilities. The first says that (\ref{eqn.pullback2}) and
(\ref{eqn.pullback}) are compatible with the units and counits of
the adjoint pairs $(u^{*},u_{*})$ and $(v^{*},v_{*})$ (Lemma
\ref{lemma.tensorcomp}).  The second says that if $\widetilde{U}
\overset{\tilde{u}}{\longrightarrow} U$ is an open affine
immersion and $\tilde{v}=\tilde{u} \times \operatorname{id}_{Y}$,
then (\ref{eqn.pullback2}) and (\ref{eqn.pullback}) are compatible
with the canonical isomorphisms $(u \tilde{u})^{*} \cong
\tilde{u}^{*}u^{*}$ and $(v \tilde{v})^{*} \cong
\tilde{v}^{*}v^{*}$ (Lemma \ref{lemma.peel}). The third says that
(\ref{eqn.pullback2}) and (\ref{eqn.pullback}) are compatible with
affine basechange (Lemma
\ref{lemma.canonicalcompatwithbasechange}).

\begin{lemma} \label{lemma.tensorcomp}
Consider the following diagram
\begin{equation} \label{eqn.bigcompat0}
\begin{CD}
-\otimes_{\mathcal{O}_{X}}\F & \longrightarrow &
-\otimes_{\mathcal{O}_{X}} v_{*}v^{*}\F \\
@VVV @VVV \\
u_{*}u^{*}(-) \otimes_{\mathcal{O}_{X}} \F & \longrightarrow &
u^{*}(-)\otimes_{\mathcal{O}_{U}} v^{*}\F
\end{CD}
\end{equation}
whose top horizontal and left vertical are induced by unit
morphisms, whose right vertical is the inverse of
(\ref{eqn.pullback2}) and whose bottom horizontal is induced by
(\ref{eqn.pullback}).  Then this diagram commutes.

Similarly, consider the following diagram
\begin{equation} \label{eqn.bigcompat}
\begin{CD}
-\otimes_{\mathcal{O}_{U}}\E & \longleftarrow & -
\otimes_{\mathcal{O}_{U}}v^{*}v_{*}\E \\
@AAA @AAA \\
u^{*}u_{*}(-)\otimes_{\mathcal{O}_{U}}\E & \longleftarrow &
u_{*}(-) \otimes_{\mathcal{O}_{X}} v_{*}\E
\end{CD}
\end{equation}
whose top horizontal and left vertical are induced by counit
morphisms, whose bottom horizontal is induced by the inverse of
(\ref{eqn.pullback2}), and whose right vertical is
(\ref{eqn.pullback}).  Then this diagram commutes.
\end{lemma}

\begin{proof}
We first show that (\ref{eqn.bigcompat0}) commutes. Consider the following diagram

\begin{equation} \label{eqn.compat00}
\begin{CD}
\operatorname{pr}_{1}^{*}(-) \otimes_{\mathcal{O}_{X \times Y}}\F
& \longrightarrow & \operatorname{pr}_{1}^{*}(-)
\otimes_{\mathcal{O}_{X \times Y}} v_{*}v^{*}\F  \\
@VVV @VVV \\
\operatorname{pr}_{1}^{*}u_{*}u^{*}(-) \otimes_{\mathcal{O}_{X
\times Y}}\F & & v_{*}(v^{*}\operatorname{pr}_{1}^{*}(-)
\otimes_{\mathcal{O}_{U \times Y}} v^{*}\F) \\
@VVV @VVV \\
v_{*}p^{*}u^{*}(-)\otimes_{\mathcal{O}_{X \times Y}}\F & \longrightarrow &
v_{*}(p^{*}u^{*}(-)\otimes_{\mathcal{O}_{U \times Y}}v^{*}\F)
\end{CD}
\end{equation}
whose top horizontal and upper-left vertical are induced by unit
morphisms, whose upper-right vertical and bottom horizontal are induced by the
projection formula, whose bottom-left vertical is basechange and
whose bottom-right vertical is canonical.  It suffices to show that this diagram commutes.

To this end, we consider the following diagram
\begin{equation} \label{eqn.smalldiagram}
\begin{CD}
\operatorname{pr}_{1}^{*} & \longrightarrow &
v_{*}v^{*}\operatorname{pr}_{1}^{*} \\
@VVV @VVV \\
\operatorname{pr}_{1}^{*}u_{*}u^{*} & & v_{*}p^{*}u^{*} \\
@VVV @VVV \\
v_{*}v^{*}\operatorname{pr}_{1}^{*}u_{*}u^{*} & \longrightarrow &
v_{*}p^{*}u^{*}u_{*}u^{*}
\end{CD}
\end{equation}
whose top horizontal, left verticals and bottom-right vertical
are induced by unit morphisms, and whose bottom horizontal and
upper-right vertical are canonical.  We claim that this diagram commutes.  The claim follows
by splitting (\ref{eqn.smalldiagram}) into two subdiagrams via the
morphism
$$
v_{*}v^{*}\operatorname{pr}_{1}^{*} \longrightarrow
v_{*}v^{*}\operatorname{pr}_{1}^{*}u_{*}u^{*}
$$
induced by the unit of $(u^{*},u_{*})$, and noticing that each
commutes by the naturality of the unit.

Next consider the following diagram
\begin{equation} \label{eqn.compat01}
\begin{CD}
\operatorname{pr}_{1}^{*}(-) \otimes_{\mathcal{O}_{X \times Y}}\F
& \longrightarrow & \operatorname{pr}_{1}^{*}(-)
\otimes_{\mathcal{O}_{X \times Y}}v_{*}v^{*}\F \\
@VVV @VVV \\
v_{*}v^{*}\operatorname{pr}_{1}^{*}(-) \otimes_{\mathcal{O}_{X
\times Y}} \F & \longrightarrow &
v_{*}(v^{*}\operatorname{pr}_{1}^{*}(-)\otimes_{\mathcal{O}_{U
\times Y}} v^{*}\F) \\
@VVV @VVV \\
v_{*}p^{*}u^{*}(-)\otimes_{\mathcal{O}_{X \times Y}} \F &
\longrightarrow & v_{*}(p^{*}u^{*}(-) \otimes_{\mathcal{O}_{U
\times Y}} v^{*}\F)
\end{CD}
\end{equation}
whose top horizontal and top-left vertical are induced by unit
morphisms, whose top-right vertical and middle and bottom
horizontals are induced by the projection formula, and whose bottom verticals
are canonical.  The claim implies that, in order to show (\ref{eqn.compat00})
commutes, it suffices to show that both squares of (\ref{eqn.compat01}) commute.  The bottom square of (\ref{eqn.compat01}) commutes by the
naturality of the projection formula.

We next prove that the top
square of (\ref{eqn.compat01}) commutes.  To this end, consider the following diagram
\begin{equation} \label{eqn.symmetrictensor1}
\begin{CD}
-\otimes_{\mathcal{O}_{X \times Y}}- & \longrightarrow &
-\otimes_{\mathcal{O}_{X \times Y}}v_{*}v^{*}(-) \\
@VVV @VVV \\
v_{*}v^{*}(-)\otimes_{\mathcal{O}_{X \times Y}}- & \longrightarrow
& v_{*}(v^{*}(-)\otimes_{\mathcal{O}_{U \times Y}} v^{*}(-))
\end{CD}
\end{equation}
whose right vertical and bottom horizontal are projection
formulas, and whose left vertical and top horizontal are induced
by units.  In order to prove that the top
square of (\ref{eqn.compat01}) commutes, it suffices to show that (\ref{eqn.symmetrictensor1}) commutes.  To prove this,
we note that the bottom route of (\ref{eqn.symmetrictensor1}) equals the bottom route in the diagram
\begin{equation} \label{eqn.symmetrictensor2}
\begin{CD}
-\otimes_{\mathcal{O}_{X \times Y}}- & \longrightarrow &
v_{*}v^{*}(-\otimes_{\mathcal{O}_{X \times Y}}-) & \longrightarrow
& v_{*}(v^{*}(-) \otimes_{\mathcal{O}_{U \times Y}} v^{*}(-)) \\
@VVV @VVV @AAA \\
v_{*}v^{*}-\otimes_{\mathcal{O}_{X \times Y}}- & \longrightarrow &
v_{*}v^{*}(v_{*}v^{*}-\otimes_{\mathcal{O}_{X \times Y}}-) &
\longrightarrow & v_{*}(v^{*}v_{*}v^{*}(-) \otimes_{\mathcal{O}_{U
\times Y}} v^{*}(-))
\end{CD}
\end{equation}
whose left arrows are unit morphisms, whose right horizontals are
induced by distributivity of pullbacks over tensor products
\begin{equation} \label{eqn.canonicalmorph}
v^{*}(-\otimes_{\mathcal{O}_{X \times Y}} -)
\overset{\cong}{\longrightarrow} v^{*}(-) \otimes_{\mathcal{O}_{U
\times Y}} v^{*}(-)
\end{equation}
and whose right horizontal is the counit morphism.  We claim that
(\ref{eqn.symmetrictensor2}) commutes.  It will follow from the claim
that the bottom route of (\ref{eqn.symmetrictensor1}) equals the composite
of the top horizontals in (\ref{eqn.symmetrictensor2}).
Similarly, the top route of (\ref{eqn.symmetrictensor1}) equals the composite
of the top horizontals in (\ref{eqn.symmetrictensor2}).  Therefore, the
commutativity of (\ref{eqn.symmetrictensor1}), and hence the commutativity of
(\ref{eqn.bigcompat0}) will follow from the commutativity of (\ref{eqn.symmetrictensor2}).
We establish this now.

The left square of (\ref{eqn.symmetrictensor2}) commutes by naturality of the unit
morphism.  To prove the right square of (\ref{eqn.symmetrictensor2}) commutes, we consider the following diagram
$$
\begin{CD}
v_{*}v^{*}(-\otimes_{\mathcal{O}_{X \times Y}}-) & \longrightarrow
& v_{*}(v^{*}(-) \otimes_{\mathcal{O}_{U \times Y}} v^{*}(-)) \\
@VVV @VVV  \\
v_{*}v^{*}(v_{*}v^{*}-\otimes_{\mathcal{O}_{X \times Y}}-) &
\longrightarrow & v_{*}(v^{*}v_{*}v^{*}(-) \otimes_{\mathcal{O}_{U
\times Y}} v^{*}(-))
\end{CD}
$$
whose verticals are induced by units and whose horizontals are
(\ref{eqn.canonicalmorph}).  This diagram commutes by naturality of
(\ref{eqn.canonicalmorph}).  It follows from this that the right square of (\ref{eqn.symmetrictensor2}) commutes as well.

The proof that (\ref{eqn.bigcompat}) commutes is similar to the
proof that (\ref{eqn.bigcompat0}) commutes, and we omit it.
\end{proof}

\begin{lemma} \label{lemma.peel}
Suppose $\widetilde{U} \subset U$ are open affine subschemes of
$X$, with inclusion morphisms $\tilde{u}:\widetilde{U} \rightarrow
U$ and $u:U \rightarrow X$. Let $\tilde{v}=\tilde{u} \times
\operatorname{id}_{Y}$, let $v=u \times \operatorname{id}_{Y}$,
and let $\F$ be an object of ${\sf Qcoh }X \times Y$. Consider the following diagram
\begin{equation} \label{eqn.peel1}
\begin{CD}
-\otimes_{\mathcal{O}_{\widetilde{U}}}\tilde{v}^{*}v^{*}\F &
\longrightarrow & \tilde{u}_{*}(-)\otimes_{\mathcal{O}_{U}}v^{*}\F
\\
@VVV @VVV \\
-\otimes_{\mathcal{O}_{\widetilde{U}}}(v\tilde{v})^{*}\F &
\longrightarrow & (u \tilde{u})_{*}(-) \otimes_{\mathcal{O}_{X}}\F
\end{CD}
\end{equation}
whose horizontals and right vertical are induced by the inverse of
(\ref{eqn.pullback}) and whose left vertical is induced by the
canonical isomorphism $\tilde{v}^{*}v^{*}
\overset{\cong}{\rightarrow} (v \tilde{v})^{*}$.  Then this diagram commutes.

Similarly, consider the following diagram
\begin{equation} \label{eqn.peel2}
\begin{CD}
-\otimes_{\mathcal{O}_{X}}v_{*}\tilde{v}_{*}\F & \longrightarrow &
u^{*}(-)\otimes_{\mathcal{O}_{U}}\tilde{v}_{*}\F \\
@VVV @VVV \\
(u\tilde{u})^{*}(-)\otimes_{\mathcal{O}_{\widetilde{U}}}\F &
\overset{\cong}{\longrightarrow} &
\tilde{u}^{*}u^{*}(-)\otimes_{\mathcal{O}_{\widetilde{U}}}\F
\end{CD}
\end{equation}
whose top horizontal and verticals are induced by the inverse of
(\ref{eqn.pullback2}), and whose bottom horizontal is induced by
the canonical isomorphism $(u \tilde{u})^{*} \cong
\tilde{u}^{*}u^{*}$.  Then this diagram commutes.
\end{lemma}

\begin{proof}
Let $p,q:U \times Y \rightarrow U,Y$ and $\tilde{p}, \tilde{q}:\widetilde{U}
\times Y \rightarrow \widetilde{U},Y$ denote projections.  Consider the diagram
\begin{equation} \label{eqn.peel}
\begin{CD}
\operatorname{pr}_{1}^{*}(u
\tilde{u})_{*}(-)\otimes_{\mathcal{O}_{X \times Y}}\F &
\longrightarrow &
v_{*}p^{*}\tilde{u}_{*}(-)\otimes_{\mathcal{O}_{X \times Y}}\F &
\longrightarrow &
v_{*}(p^{*}\tilde{u}_{*}(-)\otimes_{\mathcal{O}_{U
\times Y}}v^{*}\F) \\
@VVV @VVV @VVV \\
(v \tilde{v})_{*}\tilde{p}^{*}(-)\otimes_{\mathcal{O}_{X \times
Y}}\F & \overset{=}{\longrightarrow} & v_{*}
\tilde{v}_{*}\tilde{p}^{*}(-)\otimes_{\mathcal{O}_{X \times Y}}\F
& \longrightarrow & v_{*}(
\tilde{v}_{*}\tilde{p}^{*}(-)\otimes_{\mathcal{O}_{U \times
Y}}v^{*}\F) \\
@VVV & & @VVV \\
(v
\tilde{v})_{*}(\tilde{p}^{*}(-)\otimes_{\mathcal{O}_{\widetilde{U}
\times Y}}(v \tilde{v})^{*}\F) & @>>> &
v_{*}\tilde{v}_{*}(\tilde{p}^{*}(-)\otimes_{\mathcal{O}_{\widetilde{U}
\times Y}}\tilde{v}^{*}v^{*}\F)
\end{CD}
\end{equation}
whose top-left horizontal and top verticals are induced by
basechange, whose top-right horizontal, middle-right horizontal
and bottom verticals are induced by the projection formula, and
whose bottom isomorphism is induced by the canonical isomorphism
$(v\tilde{v})^{*} \overset{\cong}{\longrightarrow}
\tilde{v}^{*}v^{*}$.  Since $q=\operatorname{pr}_{2} v$
and $\tilde{q}=\operatorname{pr}_{2} v \tilde{v}$, in order
to prove (\ref{eqn.peel1}) commutes, it suffices to show that (\ref{eqn.peel}) commutes.

The upper-right square of (\ref{eqn.peel}) commutes by the
naturality of the projection formula.  The fact that the
upper-left square of (\ref{eqn.peel}) commutes follows from the
commutativity of the diagram
\begin{equation} \label{eqn.newbasechange}
\begin{CD}
\operatorname{pr}_{1}^{*}(u\tilde{u})_{*} &
\overset{=}{\longrightarrow} &
\operatorname{pr}_{1}^{*}u_{*}\tilde{u}_{*} \\
@V{\cong}VV @VV{\cong}V \\
(v\tilde{v})_{*}\tilde{p}^{*} & \underset{\cong}{\longrightarrow}
& v_{*}p^{*}\tilde{u}_{*}
\end{CD}
\end{equation}
whose non-trivial isomorphisms are induced by basechange.  The
commutativity of (\ref{eqn.newbasechange}) can be checked affine
locally and we omit the routine verification.

The commutativity of the bottom rectangle of (\ref{eqn.peel})
follows from the commutativity of the diagram
\begin{equation} \label{eqn.newproj}
\begin{CD}
(v\tilde{v})_{*}(-)\otimes_{\mathcal{O}_{X \times Y}}\F &
\overset{\cong}{\longrightarrow} &
v_{*}(\tilde{v}_{*}(-)\otimes_{\mathcal{O}_{U \times Y}}v^{*}\F)
\\
@V{\cong}VV @VV{\cong}V \\
(v \tilde{v})_{*}((-) \otimes_{\mathcal{O}_{\widetilde{U} \times
Y}}(v\tilde{v})^{*}\F) & \underset{\cong}{\longrightarrow} &
v_{*}\tilde{v}_{*}((-)\otimes_{\mathcal{O}_{\widetilde{U} \times
Y}}\tilde{v}^{*}v^{*}\F)
\end{CD}
\end{equation}
whose bottom horizontal is induced by the canonical isomorphism
$(v\tilde{v})^{*} \overset{\cong}{\longrightarrow}
\tilde{v}^{*}v^{*}$ and whose other arrows are induced by the
projection formula.  The commutativity of (\ref{eqn.newproj}) again
follows from a routine affine computation, which we omit.

The proof that (\ref{eqn.peel2}) commutes is similar and may be
reduced to the commutativity of a diagram of the form
(\ref{eqn.newproj}) as well. We leave the details to the reader.
\end{proof}

\begin{lemma} \label{lemma.canonicalcompatwithbasechange}
Let $U_{1}, U_{2} \subset X$ be affine open subschemes, let
$U_{12}:=U_{1} \cap U_{2}$ with inclusions
$$
\begin{CD}
U_{12} & \overset{u_{12}^{1}}{\longrightarrow} & U_{1}
\\
@V{u_{12}^{2}}VV @VV{u_{1}}V \\
U_{2} & \underset{u_{2}}{\longrightarrow} & X.
\end{CD}
$$
For $i=1,2$, let $v_{i}=u_{i} \times \operatorname{id}_{Y}$ and
let $v_{12}^{i}=u_{12}^{i} \times \operatorname{id}_{Y}$.  Let $\E$
be an object of ${\sf Qcoh }U_{1} \times Y$, and consider the diagram
$$
\begin{CD}
u_{1}^{*}u_{2*}(-)\otimes_{\mathcal{O}_{U_{1}}}\E &
\longrightarrow &
u_{12*}^{1}u_{12}^{2*}(-)\otimes_{\mathcal{O}_{U_{1}}}\E \\
@AAA @VVV \\
u_{2*}(-)\otimes_{\mathcal{O}_{X}}v_{1*}\E & &
u_{12}^{2*}(-)\otimes_{\mathcal{O}_{U_{12}}} v_{12}^{1*}\E \\
@AAA @VVV \\
-\otimes_{\mathcal{O}_{U_{2}}}v_{2}^{*}v_{1*}\E & \longrightarrow
& -\otimes_{\mathcal{O}_{U_{2}}}v_{12*}^{2}v_{12}^{1*}\E
\end{CD}
$$
whose horizontals are induced by basechange, and whose verticals
are induced by (\ref{eqn.pullback2}) and (\ref{eqn.pullback}).  Then this diagram commutes.
\end{lemma}

\begin{proof}
Consider the following diagram
\begin{equation} \label{eqn.dig1}
\begin{CD}
u_{2*}(-)\otimes_{\mathcal{O}_{X}}v_{1*}v_{12*}^{1}v_{12}^{1*}\E &
\longrightarrow &
u_{1}^{*}u_{2*}(-)\otimes_{\mathcal{O}_{U_{1}}}v_{12*}^{1}v_{12}^{1*}\E
& \longrightarrow &
u_{12}^{1*}u_{1}^{*}u_{2*}(-)\otimes_{\mathcal{O}_{U_{12}}}v_{12}^{1*}\E
\\
@AAA & & @VVV \\
-\otimes_{\mathcal{O}_{U_{2}}}v_{2}^{*}v_{1*}v_{12*}^{1}v_{12}^{1*}\E
& & & &
u_{12}^{2*}u_{2}^{*}u_{2*}(-)\otimes_{\mathcal{O}_{U_{12}}}
v_{12}^{1*}\E \\
@AAA & & @VVV \\
-\otimes_{\mathcal{O}_{U_{2}}}v_{2}^{*}v_{1*}\E & \longrightarrow
& -\otimes_{\mathcal{O}_{U_{2}}} v_{12*}^{2}v_{12}^{1*}\E &
\longleftarrow & u_{2}^{*}u_{2*}(-)
\otimes_{\mathcal{O}_{U_{2}}}v_{12*}^{2}v_{12}^{1*}\E
\end{CD}
\end{equation}
whose top horizontals and bottom-right vertical are induced by
(\ref{eqn.pullback2}), whose bottom-left horizontal is
induced by basechange, whose bottom-right horizontal is induced by
a counit, whose bottom-left vertical is induced by a unit, whose
top-left vertical is induced by the inverse of (\ref{eqn.pullback}), and whose
top-right vertical is canonical.  By the naturality of units, counits and the morphisms
(\ref{eqn.pullback2}) and (\ref{eqn.pullback}), and by the
commutativity of (\ref{eqn.bigcompat0}), it suffices to prove that (\ref{eqn.dig1}) commutes.  To this end,
we consider the diagram
$$
\begin{CD}
-\otimes_{\mathcal{O}_{U_{2}}}v_{2}^{*}v_{1*}v_{12*}^{1}v_{12}^{1*}\E
& \longleftarrow &
-\otimes_{\mathcal{O}_{U_{2}}}v_{12*}^{2}v_{12}^{1*}v_{12*}^{1}v_{12}^{1*}\E
\\
@AAA @AAA \\
-\otimes_{\mathcal{O}_{U_{2}}}v_{2}^{*}v_{1*}\E & \longleftarrow &
-\otimes_{\mathcal{O}_{U_{2}}}v_{12*}^{2}v_{12}^{1*}\E
\end{CD}
$$
whose verticals are induced by units and whose horizontals are
induced by basechange.  By the naturality of
basechange, this diagram commutes.  Hence, to prove that
(\ref{eqn.dig1}) commutes, it suffices to prove that if $\F : =
v_{12}^{1*}\E$, then the following diagram
\begin{equation} \label{eqn.dig2}
\begin{CD}
u_{2*}(-)\otimes_{\mathcal{O}_{X}}v_{1*}v_{12*}^{1}\F &
\longrightarrow &
u_{1}^{*}u_{2*}(-)\otimes_{\mathcal{O}_{U_{1}}}v_{12*}^{1}\F &
\longrightarrow &
u_{12}^{1*}u_{1}^{*}u_{2*}(-)\otimes_{\mathcal{O}_{U_{12}}}\F
\\
@AAA & & @VVV \\
-\otimes_{\mathcal{O}_{U_{2}}}v_{2}^{*}v_{1*}v_{12*}^{1}\F & & & &
u_{12}^{2*}u_{2}^{*}u_{2*}(-)\otimes_{\mathcal{O}_{U_{12}}}
\F \\
@AAA & & @VVV \\
-\otimes_{\mathcal{O}_{U_{2}}}v_{12*}^{2}v_{12}^{1*}v_{12*}^{1}\F
& \longrightarrow & -\otimes_{\mathcal{O}_{U_{2}}} v_{12*}^{2}\F &
\longleftarrow & u_{2}^{*}u_{2*}(-)
\otimes_{\mathcal{O}_{U_{2}}}v_{12*}^{2} \F
\end{CD}
\end{equation}
whose bottom-left vertical is induced by basechange, whose
bottom-left horizontal is induced by a counit, and whose other
maps are identical to the maps in (\ref{eqn.dig1}), commutes.

We complete the proof by showing that the diagram (\ref{eqn.dig2}) commutes.
To this end, we note that (\ref{eqn.dig2}) can be broken into the following four
subdiagrams: the diagram
\begin{equation} \label{eqn.sub1}
\begin{CD}
-\otimes_{\mathcal{O}_{U_{2}}}v_{2}^{*}v_{1*}v_{12*}^{1}\F &
\overset{=}{\longrightarrow} &
-\otimes_{\mathcal{O}_{U_{2}}}v_{2}^{*}v_{2*}v_{12*}^{2}\F \\
@VVV @VVV \\
-\otimes_{\mathcal{O}_{U_{2}}}v_{12*}^{2}v_{12}^{1*}v_{12*}^{1}\F
& \longrightarrow & -\otimes_{\mathcal{O}_{U_{2}}}v_{12*}^{2}\F
\end{CD}
\end{equation}
whose left vertical is induced by basechange and whose right
vertical and bottom horizontal are counits, the diagram
\begin{equation} \label{eqn.sub2}
\begin{CD}
u_{2*}(-)\otimes_{\mathcal{O}_{X}}v_{1*}v_{12*}^{1}\F &
\longrightarrow &
u_{1}^{*}u_{2*}(-)\otimes_{\mathcal{O}_{U_{1}}}v_{12*}^{1}\F &
\longrightarrow &
u_{12}^{1*}u_{1}^{*}u_{2*}(-)\otimes_{\mathcal{O}_{U_{12}}}\F \\
@V{=}VV & & @VV{\cong}V \\
u_{2*}(-)\otimes_{\mathcal{O}_{X}}v_{2*}v_{12*}^{2}\F &
\longrightarrow & u_{2}^{*}u_{2*}(-)\otimes_{\mathcal{O}_{U_{2}}}
v_{12*}^{2}\F & \longrightarrow &
u_{12}^{2*}u_{2}^{*}u_{2*}(-)\otimes_{\mathcal{O}_{U_{12}}} \F
\end{CD}
\end{equation}
whose horizontals are induced by the inverse of (\ref{eqn.pullback2}) and whose
right vertical is canonical, the diagram
\begin{equation} \label{eqn.sub3}
\begin{CD}
-\otimes_{\mathcal{O}_{U_{2}}}v_{2}^{*}v_{2*}v_{12*}^{2}\F &
\longleftarrow & u_{2*}(-)\otimes_{\mathcal{O}_{X}}
v_{2*}v_{12*}^{2}\F \\
@VVV @VVV \\
-\otimes_{\mathcal{O}_{U_{2}}}v_{12*}^{2}\F & \longleftarrow &
u_{2}^{*}u_{2*}(-)\otimes_{\mathcal{O}_{U_{2}}}v_{12*}^{2}\F
\end{CD}
\end{equation}
whose top horizontal is (\ref{eqn.pullback}), whose right vertical
is (\ref{eqn.pullback2}) and whose other arrows are counits, and the diagram
\begin{equation} \label{eqn.sub4}
\begin{CD}
u_{2*}(-)\otimes_{\mathcal{O}_{X}}v_{1*}v_{12*}^{1}\F &
\overset{=}{\longrightarrow} &
u_{2*}(-)\otimes_{\mathcal{O}_{X}}v_{2*}v_{12*}^{2}\F \\
@AAA @AAA \\
-\otimes_{\mathcal{O}_{U_{2}}}v_{2}^{*}v_{1*}v_{12*}^{1}\F &
\overset{=}{\longrightarrow} &
-\otimes_{\mathcal{O}_{U_{2}}}v_{2}^{*}v_{2*}v_{12*}^{2}\F
\end{CD}
\end{equation}
whose verticals are induced by the inverse of (\ref{eqn.pullback}).  It suffices
to show that these subdiagrams commute.  The fact that diagram
(\ref{eqn.sub1}) commutes is left as an exercise to the reader.
The fact that diagram (\ref{eqn.sub2}) commutes follows from Lemma
\ref{lemma.peel}. The fact that diagram (\ref{eqn.sub3}) commutes
follows from Lemma \ref{lemma.tensorcomp}, and the commutativity of
(\ref{eqn.sub4}) is trivial.
\end{proof}

\section{Totally Global Functors} \label{section.totallyglobal}
Our goal in this section is to define and study elementary properties of {\it totally
global functors}.

\begin{defn} \label{defn.totallyglobal}
We say $F \in {\sf Funct}_{k}({\sf Qcoh }X,{\sf Qcoh }Y)$ is {\it
totally global} if for any open immersion $u:U
\longrightarrow X$ with $U$ affine, $Fu_{*}=0$.
\end{defn}

We note that this definition makes sense.  For, $u: U \rightarrow
X$ is an affine morphism since $X$ is separated \cite[II, ex. 4.3]{hart}, so
that $u$ is quasi-compact and separated \cite[II, ex.
5.17b]{hart}.  Hence, $u_{*}$ takes quasi-coherent $\mathcal{O}_{X}$-modules to
quasi-coherent $\mathcal{O}_{X}$-modules \cite[II, Prop.
5.8c]{hart}.

The following lemma explains the motivation behind the use of the
term {\it totally global}.

\begin{lemma} \label{lemma.totglob1}
Suppose $X$ is noetherian.  If $F$ is totally global and $\M$ is a
quasi-coherent $\mathcal{O}_{X}$-module whose support lies in an
affine open subset $U$ of $X$ (included via $u$), then $F\M=0$.
\end{lemma}

\begin{proof}
Since $F$ commutes with direct limits and $X$ is noetherian, it
suffices to prove that $F\M = 0$ for $\M$ coherent.  Let
$i:\operatorname{Supp }\M \rightarrow X$ and
$i':\operatorname{Supp}\M \rightarrow U$ denote inclusions, so
that $i=ui'$.  Since $i$ is a closed immersion, the unit map $\M
\rightarrow i_{*}i^{*}\M$ is an isomorphism. Thus,
\begin{eqnarray*}
F\M & \cong & Fi_{*}i^{*}\M \\
& = & F(ui')_{*}i^{*}\M \\
& = & Fu_{*}{i'}_{*}i^{*}\M \\
& = & 0.
\end{eqnarray*}

\end{proof}

\begin{example}
Let $W$ be a noetherian scheme.  Then, for $i>0$ the functor $H^{i}(W,-)$ is
totally global by \cite[III, ex. 8.2]{hart}.
\end{example}

\begin{prop} \label{prop.vanish}
If $\F$ is an object of ${\sf Qcoh }X \times Y$ and
$F=-\otimes_{\mathcal{O}_{X}}\F$ is totally global then $\mathcal{F}=0$.
\end{prop}

\begin{proof}
Suppose $U$ is an affine scheme, $u:U \rightarrow X$ is an
open immersion, $v=u \times \operatorname{id }_{Y}$ and $p,q:U \times Y
\rightarrow U, Y$ are projections.  The map (\ref{eqn.pullback}) induces an
isomorphism,
\begin{eqnarray*}
u_{*}(-) \otimes_{\mathcal{O}_{X}} \F &
\overset{\cong}{\longrightarrow} &
-\otimes_{\mathcal{O}_{U}} v^{*} \F \\
& \overset{=}{\longrightarrow} & q_{*}(p^{*}(-)
\otimes_{\mathcal{O}_{U \times Y}} v^{*}\F)
\end{eqnarray*}
Since $F$ is totally global,
\begin{eqnarray*}
0 & = & Fu_{*}\mathcal{O}_{U} \\
& \cong & q_{*}v^{*}\F \\
& = & \operatorname{pr}_{2*}(v_{*}v^{*}\F).
\end{eqnarray*}
Thus, if
$W$ is an affine open subset of $Y$, $v^{*}\F(U \times W)=0$.
Therefore, $v^{*}\F=0$ since its sections on an affine open cover
are $0$.  We conclude that if $p \in U \times Y$, then $\F_{p}=0$.
Since $U$ is arbitrary, $\F=0$ as desired.
\end{proof}

For the remainder of this section, we take {\it affine open cover of
$X$} to mean a set of pairs $\{(U_{i},u_{i})\}$ where $u_{i}:U_{i}
\rightarrow X$ is inclusion of an affine open subset $U_{i}$ of
$X$ such that every point of $X$ is contained in some $U_{i}$.

\begin{prop} \label{prop.single}
If $F \in {\sf Bimod}_{k}(X-Y)$ and $\{(U_{i},u_{i})\}$ is an
affine open cover of $X$ such that $Fu_{i*}=0$ for all $i$, then
$F$ is totally global.
\end{prop}

\begin{proof}
We first prove that if $X$ is affine, $\{(W_{i},w_{i})\}$ is an
affine open cover of $X$, and $E \in {\sf Funct}_{k}(X-Y)$ is such
that $Ew_{i*}=0$ for all $i$, then $E=0$. Since $X$ is affine, $E
\cong -\otimes_{\mathcal{O}_{X}} \E$ for some object $\E$ of ${\sf
Qcoh} X \times Y$ by Proposition \ref{prop.wattone}. Thus, if
$p,q:W_{i} \times Y \rightarrow W_{i},Y$ are projections and
$v_{i} =w_{i} \times \operatorname{id}_{Y}$, then by
(\ref{eqn.pullback}),
$$
Ew_{i*} \cong q_{*}(p^{*}-\otimes_{\mathcal{O}_{W_{i} \times Y}}
{v_{i}}^{*}\E).
$$
The vanishing of $Ew_{i*}$ for all $i$ implies that
$q_{*}{v_{i}}^{*}\E = 0$ for all $i$. Therefore, for all $i$ and
all $W \subset Y$ open affine,
$$
{v_{i}}^{*}\E(W_{i} \times W)=0.
$$
This implies that ${v_{i}}^{*}\E=0$ for all $i$ which implies that
$\E$, and hence $E$, is $0$.

Now we prove that for any $X$ and any affine open cover $\{(V_j,v_j)\}$ of $X$, $Fv_{j*}=0$ for all $j$.  The proposition will follow.  Let $F$ and $\{(U_{i},u_{i})\}$ be as in
the statement of the Proposition.  Let $w_{ij}:U_{i} \cap V_{j}
\rightarrow V_{j}$ and $w_{ij}':U_{i} \cap V_{j} \rightarrow
U_{i}$ denote inclusions. Then $Fv_{j*}w_{ij*}=Fu_{i*}w_{ij*}'=0$
for all $i$ by hypothesis. But $V_{j}$ is affine, $Fv_{j*} \in
{\sf Bimod}_{k}(V_{j} \times Y)$ since $v_{j*}$ is right exact by
the affineness of $v_{j}$ \cite[III, Prop. 8.1 and
Remark 3.5.1]{hart}, and $\mathfrak{W} := \{(U_{i} \cap V_{j},
w_{ij})\}_{i}$ is an affine open cover of $V_{j}$.  Hence the
argument of the first paragraph applies to the functor $E=Fv_{j*}$
and the open cover $\mathfrak{W}$ of $V_{j}$, so that $Fv_{j*}=0$.
\end{proof}

%The following is elementary:

%\begin{lemma} \label{lemma.localizing}
%Let ${\sf TG}$ denote the full subcategory of ${\sf{Funct }}_{k}({\sf Qcoh }X, {\sf Qcoh }Y)$ consisting
%of totally global functors.  Then ${\sf TG}$ is a Serre
%subcategory of the category ${\sf{Funct }}_{k}({\sf Qcoh }X, {\sf Qcoh }Y)$.
%\end{lemma}

\section{The Eilenberg-Watts Functor} \label{section.wattsfunctor}
In this section we review the construction of an assignment
$$
W: {\sf Bimod}_{k}(X - Y) \rightarrow {\sf Qcoh }X \times Y
$$
sketched in \cite[Lemma 3.1.1]{ncruled}, and prove it is
functorial (Subsection \ref{subsection.defn}), left-exact
(Proposition \ref{prop.lex}), and compatible with affine
localization (Proposition \ref{prop.loc}).
 We will show in Corollary \ref{cor.approx} that if $F \in {\sf Bimod}_{k}(X-Y)$ then $-\otimes_{\mathcal{O}_{X}}W(F)$ serves as a
``best" approximation to $F$ by tensoring with a bimodule.  In
order to prove this, we will need the fact, proven in Proposition
\ref{prop.sheafid}, that if $F \cong -\otimes_{\mathcal{O}_{X}}\F$ then
$W(F) \cong \F$.  We end the section by showing that if $F$ is exact, then $\operatorname{pr}_{1}^{*}(-)\otimes_{\mathcal{O}_{X \times Y}}W(F)$ is exact (Corollary \ref{cor.fexact}).  This
result is used in Section \ref{section.wattstheory} to prove that if $F$ is exact then $F \cong -\otimes_{\mathcal{O}_{X}}W(F)$ (Corollary \ref{thm.cokvanish}).

\subsection{Preliminaries} \label{subsection.prelims}
Before defining the functor $W$, we describe conventions we will
employ throughout the rest of this paper.

Let $\{U_{i}\}_{i \in I}$ be a collection of open subschemes of
$X$ (we identify the underlying set of $U_{i}$ with a subset of
the underlying set of $X$).  For any finite subset $\{i_{1},
\ldots, i_{n}\}$ of $I$, we let
$$
U_{i_{1} \cdots i_{n}} = U_{i_{1}} \cap \cdots \cap U_{i_{n}}
$$
and we let
$$
u_{i_{1}\cdots i_{n}}: U_{i_{1} \cdots i_{n}} \rightarrow X
$$
denote the inclusion morphism.  For any inclusion $\{j_{1},\ldots,
j_{m}\} \subset \{i_{1},\ldots,i_{n}\}$ of finite subsets of $I$,
we let
$$
u_{i_{1} \cdots i_{n}}^{j_{1} \cdots j_{m}}:U_{i_{1} \cdots i_{n}}
\rightarrow U_{j_{1} \cdots j_{m}}
$$
denote the inclusion morphism.  Similar conventions apply when the
open subschemes are labelled $\{V_{j}\}$ or $\{W_{k}\}$, etc.  We
denote the open cover $\{U_{i}\}_{i \in I}$ by $\mathfrak{U}$.

For $i,j$ with $j \neq i$, we let
$$
\eta_{ij}^{i}:\operatorname{id}_{{\sf Qcoh }U_{i}} \rightarrow
u_{ij*}^{i}u_{ij}^{i*}
$$
denote the canonical unit of the adjoint pair
$(u_{ij}^{i*},u_{ij*}^{i})$.

\subsection{Definition of the Eilenberg-Watts Functor}
\label{subsection.defn}

Let $F$ be an object in the category ${\sf Bimod}_{k}(X-Y)$.  Our goal in this
subsection is to associate to $F$ an object $W(F) \in {\sf Qcoh }X
\times Y$, and show the assignment $F \mapsto W(F)$ is functorial.
To this end, we first choose a finite affine open cover of $X$,
$\mathfrak{U} = \{U_{i}\}_{i \in I}$ with $I=\{1, \ldots, n\}$.
Recall that $X$ is quasi-compact, so such an open cover exists.

For each $i \in I$, the proof of Proposition \ref{prop.wattone}
gives us an $\F_{i} \in {\sf Qcoh }U_{i} \times Y$ and a canonical
isomorphism $Fu_{i*} \overset{\cong}{\longrightarrow} -
\otimes_{\mathcal{O}_{U_{i}}} \F_{i}$.

Now let $V_{i} = U_{i} \times Y$.  Recalling our notational conventions about
open covers of $X \times Y$ in Section \ref{subsection.prelims}, we claim that there exists a
canonical isomorphism
\begin{equation} \label{eqn.psidef}
\psi_{ij}:v_{ij}^{i*}\F_{i} \overset{\cong}{\longrightarrow}
v_{ij}^{j*}\F_{j}.
\end{equation}
To prove the claim, we note that there are isomorphisms
\begin{eqnarray*}
-\otimes_{\mathcal{O}_{U_{ij}}}v_{ij}^{i*}\F_{i} &
\overset{\cong}{\longrightarrow} & u_{ij*}^{i}(-)
\otimes_{\mathcal{O}_{U_{i}}} \F_{i} \\
& \overset{\cong}{\longrightarrow} & Fu_{i*}u_{ij*}^{i} \\
& \overset{=}{\longrightarrow} & Fu_{j*}u_{ij*}^{j} \\
& \overset{\cong}{\longrightarrow} & u_{ij*}^{j}(-)
\otimes_{\mathcal{O}_{U_{j}}} \F_{j} \\
& \overset{\cong}{\longrightarrow} &
-\otimes_{\mathcal{O}_{U_{ij}}}v_{ij}^{j*}\F_{j},
\end{eqnarray*}
the first is the inverse of (\ref{eqn.pullback}), the second is
from the definition of $\F_{i}$ and the fourth and fifth are
defined similarly.  The map $\psi_{ij}$ corresponds to the
composition above under the equivalence from Proposition
\ref{prop.wattone}.

Next, for each pair $i,j \in I$ with $j > i$, we let
$$
\phi_{i}^{ij}:v_{i*}\F_{i} \longrightarrow
v_{i*}v_{ij*}^{i}v_{ij}^{i*}\F_{i}
$$
denote the morphism induced by $\eta_{ij}^{i}$ and we define
$$
\phi_{j}^{ij}:v_{j*}\F_{j} \longrightarrow
v_{i*}v_{ij*}^{i}v_{ij}^{i*}\F_{i}
$$
as the composition of the morphism $v_{j*}\F_{j} \longrightarrow
v_{j*}v_{ij*}^{j}v_{ij}^{j*}\F_{j} =
v_{i*}v_{ij*}^{i}v_{ij}^{j*}\F_{j}$ induced by $\eta_{ij}^{j}$ and
the morphism $v_{i*}v_{ij*}^{i}v_{ij}^{j*}\F_{j} \longrightarrow
v_{i*}v_{ij*}^{i}v_{ij}^{i*}\F_{i}$ induced by $\psi_{ij}^{-1}$.

Finally, since $I$ is finite, in order to specify a morphism
$$
\oplus_{i}v_{i*}\F_{i} \longrightarrow
\oplus_{i<j}v_{i*}v_{ij*}^{i}v_{ij}^{i*}\F_{i},
$$
it suffices to define a morphism
$$
\theta_{i}^{jk}:v_{i*}\F_{i}
\rightarrow v_{j*}v_{jk*}^{j}v_{jk}^{j*}\F_{j}
$$
for all $i,j,k \in I$ with $j<k$.  We define such a morphism as

\begin{equation} \label{eqn.babytheta}
\theta_{i}^{jk} =
\begin{cases}
\phi_{i}^{ik} & \text{if $i=j$},\\
-\phi_{i}^{ji} & \text{if $i=k$, and} \\
0 & \text{otherwise}.
\end{cases}
\end{equation}
The morphisms $\{\theta_{i}^{jk}\}$ induce a morphism
\begin{equation} \label{eqn.theta}
\theta_{F}:\oplus_{i}v_{i*}\F_{i} \longrightarrow
\oplus_{i<j}v_{i*}v_{ij*}^{i}v_{ij}^{i*}\F_{i}.
\end{equation}
We define
$$
W_{\mathfrak{U}}(F) := \operatorname{ker }\theta_{F}.
$$
We next note that $W_{\mathfrak{U}}(F)$ is an object of ${\sf Qcoh
}X \times Y$.  For, since $v_{i*}v_{ij*}^{i}=v_{ij*}$ is an affine
morphism, it is quasi-compact and separated by \cite[II, ex.
5.17b]{hart}.  Hence if $\mathcal{M}$ is an object of ${\sf Qcoh
}U_{ij} \times Y$ then $v_{ij*}\M$ is an object of ${\sf Qcoh }X
\times Y$ by \cite[II, Prop. 5.8c]{hart}.

We define $W_{\mathfrak{U}}$ on morphisms as follows.  Let
$\Delta:E \longrightarrow F$ be a morphism in ${\sf Bimod }_{k}(X
- Y)$ and let $\Delta * u_{i*}:E u_{i*}
\longrightarrow F u_{i*}$ denote the horizontal
composition of the natural transformations $\Delta$ and
$\operatorname{id}_{u_{i*}}$. By the proof of Proposition
\ref{prop.wattone}, there are canonical isomorphisms $Eu_{i*}
\longrightarrow -\otimes_{\mathcal{O}_{U_{i}}}\E_{i}$ and $Fu_{i*}
\longrightarrow -\otimes_{\mathcal{O}_{U_{i}}}\F_{i}$. Hence,
$\Delta * u_{i*}$ induces, via these isomorphisms, a morphism
$$
-\otimes_{\mathcal{O}_{U_{i}}}\E_{i} \longrightarrow
-\otimes_{\mathcal{O}_{U_{i}}}\F_{i}.
$$
Therefore, by Proposition \ref{prop.wattone}, there is an induced
morphism
$$
\delta_{i}:\E_{i} \longrightarrow \F_{i}.
$$
The fact that the maps $\{\delta_{i}\}_{i \in I}$ induce a
morphism $\delta:W_{\mathfrak{U}}(E) \longrightarrow W_{\mathfrak{U}}(F)$ now follows from the
naturality of $\eta_{ij}^{i}$ and of $\psi_{ij}$. We leave it as a
straightforward exercise for the reader to check that the
naturality of $\psi_{ij}$ follows from the naturality of
(\ref{eqn.pullback2}) and (\ref{eqn.pullback}).

We define
$$
W_{\mathfrak{U}}(\Delta) := \delta.
$$
It is straightforward to complete the verification that
$W_{\mathfrak{U}}$ is a functor and we omit it.

\subsection{Properties of the Eilenberg-Watts Functor}
The following result will not be used in the sequel.
\begin{prop} \label{prop.lex}
The functor $W_{\mathfrak{U}}:{\sf Bimod}_{k}(X-Y) \rightarrow
{\sf Qcoh }X \times Y$ is left-exact in the sense that if $F', F,
F'' \in {\sf Bimod}_{k}(X-Y)$ are such that
\begin{equation} \label{eqn.exactif}
0 \rightarrow F' \overset{\Lambda}{\rightarrow} F
\overset{\Xi}{\rightarrow} F'' \rightarrow 0
\end{equation}
is exact in ${\sf Funct}_{k}({\sf Qcoh }X,{\sf Qcoh }Y)$, then
$$
0 \rightarrow W_{\mathfrak{U}}(F')
\overset{W_{\mathfrak{U}}(\Lambda)}{\rightarrow}
W_{\mathfrak{U}}(F) \overset{W_{\mathfrak{U}}(\Xi)}{\rightarrow}
W_{\mathfrak{U}}(F'')
$$
is exact in ${\sf Qcoh }X \times Y$.
\end{prop}

\begin{proof}
Exactness of (\ref{eqn.exactif}) implies that, for all $u_{i}$,
$$
0 \rightarrow F'u_{i*} \overset{\Lambda}{\rightarrow} Fu_{i*}
\overset{\Xi}{\rightarrow} F''u_{i*} \rightarrow 0
$$
is exact in ${\sf Funct}_{k}(U_{i}-Y)$.  Thus, this sequence is
exact in ${\sf Bimod}_{k}(U_{i}-Y)$.  By Proposition
\ref{prop.wattone}, the induced sequence
$$
0 \rightarrow \mathcal{F}_{i}' \rightarrow \mathcal{F}_{i}
\rightarrow \mathcal{F}_{i}'' \rightarrow 0
$$
is exact in ${\sf Qcoh }U_{i}  \times  Y$.  Therefore the induced
sequences
$$
0 \rightarrow \oplus_{i} v_{i*}\mathcal{F}_{i}' \rightarrow
\oplus_{i} v_{i*}\mathcal{F}_{i} \rightarrow \oplus_{i}
v_{i*}\mathcal{F}_{i}'' \rightarrow 0
$$
and
$$
0 \rightarrow \oplus_{i<j}
v_{i*}v_{ij*}^{i}v_{ij}^{i*}\mathcal{F}_{i}' \rightarrow \oplus_{i
< j} v_{i*}v^{i}_{ij*}v^{i*}_{ij}\mathcal{F}_{i} \rightarrow
\oplus_{i < j} v_{i*}v^{i}_{ij*}v_{ij}^{i*}\mathcal{F}_{i}''
\rightarrow 0
$$
are exact since $v_{i}$ and $v_{ij}^{i}$ are affine and
$v^{i}_{ij}$ is an open immersion.  There is thus a commutative
diagram with exact rows
$$
\begin{CD}
0 & \rightarrow & \oplus_{i} v_{i*}\mathcal{F}_{i}' & \rightarrow & \oplus_{i} v_{i*}\mathcal{F}_{i} &
\rightarrow & \oplus_{i} v_{i*}\mathcal{F}_{i}'' & \rightarrow & 0 \\
& & @V{\theta_{F'}}VV @V{\theta_{F}}VV @VV{\theta_{F''}}V & & \\
0 & \rightarrow & \oplus_{i<j}
v_{i*}v_{ij*}^{i}v_{ij}^{i*}\mathcal{F}_{i}' & \rightarrow &
\oplus_{i < j} v_{i*}v^{i}_{ij*}v^{i*}_{ij}\mathcal{F}_{i} &
\rightarrow & \oplus_{i < j}
v_{i*}v^{i}_{ij*}v_{ij}^{i*}\mathcal{F}_{i}'' & \rightarrow & 0
\end{CD}
$$
Left-exactness of $W_{\mathfrak{U}}$ follows from the Snake Lemma.
\end{proof}

\begin{prop} \label{prop.loc}
The functor $W_{\mathfrak{U}}$ is compatible with affine
localization in the sense that if $\mathfrak{U} \cap U_{k}$
denotes the affine open cover $\{U_{ik}\}_{i \in I}$ of $U_{k}$,
then
$$
W_{\mathfrak{U} \cap U_{k}}(Fu_{k*}) \cong
v_{k}^{*}W_{\mathfrak{U}}(F)
$$
naturally in $F$.
\end{prop}

\begin{proof}
We prove the result in several steps.

{\it Step 1:  We note that the canonical basechange morphisms
$$
v_{k}^{*}v_{i*} \longrightarrow v_{ik*}^{k}v_{ik}^{i*}
$$
and
$$
v_{ik}^{i*}v_{ij*}^{i} \longrightarrow v_{ijk*}^{ik}v_{ijk}^{ij*}
$$
associated to the diagrams
$$
\begin{CD}
U_{ik} \times Y & \overset{v_{ik}^{i}}{\longrightarrow} & U_{i} \times Y \\
@V{v_{ik}^{k}}VV @VV{v_{i}}V \\
U_{k} \times Y & \underset{v_{k}}{\longrightarrow} & X \times Y
\end{CD}
$$
and
$$
\begin{CD}
U_{ijk} \times Y & \overset{v_{ijk}^{ij}}{\longrightarrow} & U_{ij} \times Y \\
@V{v_{ijk}^{ik}}VV @VV{v_{ij}^{i}}V \\
U_{ik} \times Y & \underset{v_{ik}^{i}}{\longrightarrow} & U_{i} \times Y
\end{CD}
$$
are isomorphisms.}  This follows from a routine affine
computation, which we omit.
\newline
{\it Step 2:  Consider the composition
$$
v_{ik*}^{k}v_{ik}^{i*} \rightarrow
v_{ik*}^{k}v_{ik}^{i*}v_{ij*}^{i}v_{ij}^{i*} \rightarrow
v_{ik*}^{k}v_{ijk*}^{ik}v_{ijk}^{ij*}v_{ij}^{i*} \rightarrow
v_{ik*}^{k}v_{ijk*}^{ik}v_{ijk}^{ik*}v_{ik}^{i*}
$$
whose left arrow is induced by the unit of the adjoint pair
$(v_{ij}^{i*},v_{ij*}^{i})$, whose middle arrow is induced by the second
basechange isomorphism from Step 1, and whose right arrow is
induced from the canonical isomorphism
\begin{equation} \label{eqn.anothercanonical}
v_{ijk}^{ij*}v_{ij}^{i*} \overset{\cong}{\rightarrow}
(v_{ij}^{i}v_{ijk}^{ij})^{*}=(v_{ik}^{i}v_{ijk}^{ik})^{*}
\overset{\cong}{\rightarrow} v_{ijk}^{ik*}v_{ik}^{i*}.
\end{equation}
We note that this composition is equal to the morphism induced by the unit of the adjoint pair
$(v_{ijk}^{ik*},v_{ijk*}^{ik})$.}  In order to prove this fact, consider the following diagram
$$
\begin{CD}
v_{ik}^{i*} & \rightarrow & v_{ik}^{i*}v_{ij*}^{i}v_{ij}^{i*} \\
@VVV @VVV \\
v_{ijk*}^{ik}v_{ijk}^{ik*}v_{ik}^{i*} & \rightarrow &
v_{ijk*}^{ik}v_{ijk}^{ij*}v_{ij}^{i*}
\end{CD}
$$
whose top horizontal and left vertical are induced by canonical
units, whose right vertical is induced by basechange
isomorphisms from Step 1, and whose bottom horizontal is induced by the
inverse of (\ref{eqn.anothercanonical}).  It suffices to prove that this diagram commutes.  The
verification of this fact follows from a routine affine
computation, which we omit.
\newline
{\it Step 3:  Let $F$ be an object of ${\sf Bimod}_{k}(X-Y)$ and
consider the morphism $\delta$
$$
\delta:\oplus_{i}v_{ik*}^{k}v_{ik}^{i*}\F_{i} \longrightarrow \oplus_{i<j}v_{ik*}^{k}v_{ijk*}^{ik}v_{ijk}^{ik*}v_{ik}^{i*}\F_{i}
$$
defined by the composition
\begin{eqnarray*}
\oplus_{i}v_{ik*}^{k}v_{ik}^{i*}\F_{i} &
\overset{\cong}{\longrightarrow} & \oplus_{i}v_{k}^{*}v_{i*}\F_{i}
\\
& \overset{v_{k}^{*}\theta_{F}}{\longrightarrow} & \oplus_{i
<j}v_{k}^{*}v_{i*}v_{ij*}^{i}v_{ij}^{i*}\F_{i} \\
& \overset{\cong}{\longrightarrow} &
\oplus_{i<j}v_{ik*}^{k}v_{ik}^{i*}v_{ij*}^{i}v_{ij}^{i*}\F_{i} \\
& \overset{\cong}{\longrightarrow} &
\oplus_{i<j}v_{ik*}^{k}v_{ijk*}^{ik}v_{ijk}^{ij*}v_{ij}^{i*}\F_{i}
\\
& \overset{\cong}{\rightarrow} &
\oplus_{i<j}v_{ik*}^{k}v_{ijk*}^{ik}v_{ijk}^{ik*}v_{ik}^{i*}\F_{i}
\end{eqnarray*}
whose first and third and fourth arrows are basechange
isomorphisms from Step 1, and whose fifth arrow is induced by the
canonical isomorphism (\ref{eqn.anothercanonical}).  Let
$\delta_{i}^{jl}$ denote the component of $\delta$ from the $i$th
summand to the $jl$th summand, i.e.
$$
\delta_{i}^{jl}: v_{ik*}^{k}v_{ik}^{i*}\F_{i} \longrightarrow v_{jk*}^{k}v_{jlk*}^{jk}v_{jlk}^{jk*}v_{jk}^{j*}\F_{j}.
$$
We show that
\begin{itemize}
\item{} $\delta_{i}^{jl}=0$ if $i$ is not equal to $j$ or $l$,

\item{} $\delta_{i}^{ij}$ is induced by the unit of the adjoint pair
$(v_{ijk}^{ik*},v_{ijk*}^{ik})$, and

\item{} $\delta_{j}^{ij}$ is equal to
$-1$ times the composition
$$
v_{jk*}^{k}v_{jk}^{j*}\F_{j} \longrightarrow
v_{jk*}^{k}v_{ijk*}^{jk}v_{ijk}^{jk*}v_{jk}^{j*}\F_{j}
\longrightarrow
v_{ik*}^{k}v_{ijk*}^{ik}v_{ijk}^{ik*}v_{ik}^{i*}\F_{i}
$$
whose left arrow is induced by the unit of the adjoint pair
$(v_{ijk}^{jk*},v_{ijk*}^{jk})$ and whose right arrow corresponds,
under the equivalence of Proposition \ref{prop.wattone}, to the
composition of functors
\begin{equation} \label{eqn.wattcanonical}
\begin{split}
-\otimes_{\mathcal{O}_{U_{ijk}}}v_{ijk}^{jk*}v_{jk}^{j*}\F_{j}
&\overset{\cong}{\longrightarrow} u_{ijk*}^{jk}(-)\otimes_{\mathcal{O}_{U_{jk}}}v_{jk}^{j*}\F_{j} \\
&\overset{\cong}{\longrightarrow}  u_{jk*}^{j}u_{ijk*}^{jk}(-) \otimes_{\mathcal{O}_{U_{j}}}\F_{j} \\
&\overset{\cong}{\longrightarrow} Fu_{j*}u_{jk*}^{j}u_{ijk*}^{jk} \\
&\overset{=}{\longrightarrow} Fu_{i*}u_{ik*}^{i}u_{ijk*}^{ik} \\
&\overset{\cong}{\longrightarrow} -\otimes_{\mathcal{O}_{U_{ijk}}}v_{ijk}^{ik*}v_{ik}^{i*}\F_{i} \\
\end{split}
\end{equation}
whose first two arrows are induced by the inverse of (\ref{eqn.pullback}), whose third arrow
is the canonical isomorphism from the proof of Proposition
\ref{prop.wattone}, and whose last arrow is defined analogously to the composition of the first three arrows.
\end{itemize}
}
The fact that $\delta_{i}^{jl}=0$ if $i$ is not equal to $j$ or
$l$ follows from the definition of $\theta_{F}$.  The assertion
regarding $\delta_{i}^{ij}$ follows from Step 2.

It remains to verify the description of $\delta_{j}^{ij}$.  Consider the following diagram
\begin{equation} \label{eqn.bigdig000}
\begin{CD}
v_{k}^{*}v_{j*}v_{ij*}^{j}v_{ij}^{j*}\F_{j} & \longrightarrow &
v_{k}^{*}v_{i*}v_{ij*}^{i}v_{ij}^{i*}\F_{i} \\
@VVV @VVV \\
v_{jk*}^{k}v_{jk}^{j*}v_{ij*}^{j}v_{ij}^{j*}\F_{j} & &
v_{ik*}^{k}v_{ik}^{i*}v_{ij*}^{i}v_{ij}^{i*}\F_{i} \\
@VVV @VVV \\
v_{jk*}^{k}v_{ijk*}^{jk}v_{ijk}^{ij*}v_{ij}^{j*}\F_{j} & &
v_{ik*}^{k}v_{ijk*}^{ik}v_{ijk}^{ij*}v_{ij}^{i*}\F_{i} \\
@VVV @VVV \\
v_{jk*}^{k}v_{ijk*}^{jk}v_{ijk}^{jk*}v_{jk}^{j*}\F_{j} &
\longrightarrow &
v_{ik*}^{k}v_{ijk*}^{ik}v_{ijk}^{ik*}v_{ik}^{i*}\F_{i}
\end{CD}
\end{equation}
whose top horizontal is induced by the map $\psi_{ji}$ defined in
(\ref{eqn.psidef}), whose bottom horizontal is induced by the
morphism corresponding to (\ref{eqn.wattcanonical}) and whose
verticals are induced by basechange isomorphisms from Step 1 and
by canonical morphisms of the form (\ref{eqn.anothercanonical}).   By
Step 2, it suffices to prove that this diagram commutes.  To this end, consider the diagrams
\begin{equation} \label{eqn.bigdig1}
\begin{CD}
v_{k}^{*}v_{j*}v_{ij*}^{j} & \overset{=}{\longrightarrow} &
v_{k}^{*}v_{i*}v_{ij*}^{i} \\
@VVV @VVV \\
v_{jk*}^{k}v_{jk}^{j*}v_{ij*}^{j} & &
v_{ik*}^{k}v_{ik}^{i*}v_{ij*}^{i} \\
@VVV @VVV \\
v_{jk*}^{k}v_{ijk*}^{jk}v_{ijk}^{ij*} &
\underset{=}{\longrightarrow} &
v_{ik*}^{k}v_{ijk*}^{ik}v_{ijk}^{ij*}
\end{CD}
\end{equation}
whose verticals are induced by basechange morphisms, and the diagram
\begin{equation} \label{eqn.bigdig2}
\begin{CD}
v_{ijk}^{ij*}v_{ij}^{j*}\F_{j} &
\overset{v_{ijk}^{ij*}\psi_{ji}}{\longrightarrow} &
v_{ijk}^{ij*}v_{ij}^{i*}\F_{i} \\
@VVV @VVV \\
v_{ijk}^{jk*}v_{jk}^{j*}\F_{j} & \longrightarrow &
v_{ijk}^{ik*}v_{ik}^{i*}\F_{i}
\end{CD}
\end{equation}
whose verticals are induced by canonical morphisms of the form
(\ref{eqn.anothercanonical}) and whose bottom horizontal is the
morphism corresponding to (\ref{eqn.wattcanonical}).  In order to prove that (\ref{eqn.bigdig000})
commutes, it suffices to prove that (\ref{eqn.bigdig1}) and (\ref{eqn.bigdig2}) commute.

The commutativity of (\ref{eqn.bigdig1}) follows from a
straightforward affine computation, which we omit.  To prove that
(\ref{eqn.bigdig2}) commutes, we first note that
$v_{ijk}^{ij*}\psi_{ji}$ corresponds to the composition
(\ref{eqn.wattcanonical}) by the naturality of
(\ref{eqn.pullback}).  Hence, a straightforward computation shows
that the commutativity of (\ref{eqn.bigdig2}) follows from the
commutativity of four "corner" subdiagrams.  The upper-left such
diagram, for example, is the diagram
$$
\begin{CD}
-\otimes_{\mathcal{O}_{U_{ijk}}}v_{ijk}^{ij*}v_{ij}^{j*}\F_{j} &
\longrightarrow &
u_{ijk*}^{ij}(-)\otimes_{\mathcal{O}_{U_{ij}}}v_{ij}^{j*}\F_{j} &
\longrightarrow &
u_{ij*}^{j}u_{ijk*}^{ij}(-)\otimes_{\mathcal{O}_{U_{j}}}\F_{j} \\
@VVV & & @VV{=}V \\
-\otimes_{\mathcal{O}_{U_{ijk}}}(v_{ij}^{j}v_{ijk}^{ij})^{*}\F_{j}
& @>>> &
(u_{ij}^{j}u_{ijk}^{ij})_{*}(-)\otimes_{\mathcal{O}_{U_{j}}}\F_{j}
\end{CD}
$$
whose horizontals are induced by the isomorphisms (\ref{eqn.pullback}) and
whose left vertical is induced by the canonical isomorphism
$$
v_{ijk}^{ij*}v_{ij}^{j*} \overset{\cong}{\longrightarrow}
(v_{ij}^{j}v_{ijk}^{ij})^{*}.
$$
These corner subdiagrams commute by Lemma \ref{lemma.peel}.
\newline
{\it Step 4:  Let $\E_{i} \in {\sf Qcoh }U_{ik} \times Y$ denote
the object corresponding to the functor $Fu_{ik*} \in {\sf
Bimod}_{k}(U_{ik}-Y)$ in the proof of Proposition
\ref{prop.wattone}.  Consider the following diagram
\begin{equation} \label{eqn.focus0}
\begin{CD}
v_{ijk}^{jk*}v_{jk}^{j*}\F_{j} & \longrightarrow &
v_{ijk}^{ik*}v_{ik}^{i*}\F_{i} \\
@VVV @VVV \\
v_{ijk}^{jk*}\E_{j} & \longrightarrow & v_{ijk}^{ik*}\E_{i}
\end{CD}
\end{equation}
whose top horizontal is the map (\ref{eqn.wattcanonical}), whose
bottom horizontal is the map $\psi_{ji}$ defined by
(\ref{eqn.psidef}) but corresponding to the functor $Fu_{k*}$,
whose left vertical is induced by the composition
\begin{eqnarray*}
-\otimes_{\mathcal{O}_{U_{ijk}}}v_{ijk}^{jk*}v_{jk}^{j*}\F_{j} &
\overset{\cong}{\longrightarrow} & u_{ijk*}^{jk}(-)
\otimes_{\mathcal{O}_{U_{jk}}}v_{jk}^{j*}\F_{j}
\\
& \overset{\cong}{\longrightarrow} &
u_{jk*}^{j}u_{ijk*}^{jk}(-)\otimes_{\mathcal{O}_{U_{j}}}\F_{j} \\
& \overset{\cong}{\longrightarrow} &
Fu_{j*}u_{jk*}^{j}u_{ijk*}^{jk} \\
& \overset{=}{\longrightarrow} & Fu_{jk*}u_{ijk*}^{jk} \\
& \overset{\cong}{\longrightarrow} & u_{ijk*}^{jk}(-)
\otimes_{\mathcal{O}_{U_{jk}}}\E_{j} \\
& \overset{\cong}{\longrightarrow} &
-\otimes_{\mathcal{O}_{U_{ijk}}}v_{ijk}^{jk*}\E_{j}
\end{eqnarray*}
whose first, second, and sixth morphisms are induced by (\ref{eqn.pullback}),
and whose third and fifth morphisms are the canonical ones
constructed in Proposition \ref{prop.wattone}, and whose right
vertical is defined similarly.  Then this diagram commutes.}  Upon expanding the rows
and columns of the diagram, the proof is seen to follow from the
trivial commutativity of the diagram
$$
\begin{CD}
Fu_{j*}u_{jk*}^{j}u_{ijk*}^{jk} & \overset{=}{\longrightarrow} &
Fu_{i*}u_{ik*}^{i}u_{ijk*}^{ik} \\
@V{=}VV @VV{=}V \\
Fu_{jk*}u_{ijk*}^{jk} & \underset{=}{\longrightarrow} &
Fu_{ik*}u_{ijk*}^{ik}.
\end{CD}
$$

{\it Step 5:  We show that $\operatorname{ker }\delta \cong \operatorname{ker }\theta_{Fu_{k*}}$.}  We retain the notation from Step 4.  It suffices to show that, for all $i,j,l$, the diagram
\begin{equation} \label{eqn.focus}
\begin{CD}
v_{ik*}^{k}v_{ik}^{i*}\F_{i} & \overset{\delta_{i}^{jl}}{\longrightarrow} & v_{jk*}^{k}v_{jkl*}^{jk}v_{jkl}^{jk*}v_{jk}^{j*}\F_{j} \\
@VVV @VVV \\
v_{ik*}^{k}\E_{i} & \underset{(\theta_{Fu_{k*}})_{i}^{jl}}{\longrightarrow} & v_{jk*}^{k}v_{jkl*}^{jk}v_{jkl}^{jk*}\E_{j}
\end{CD}
\end{equation}
whose verticals correspond, under the equivalence of Proposition \ref{prop.wattone}, to the composition of functors
\begin{eqnarray*}
-\otimes_{\mathcal{O}_{U_{ik}}}v_{ik}^{i*}\F_{i} & \overset{\cong}{\longrightarrow} & u_{ik*}^{i}(-) \otimes_{\mathcal{O}_{U_{i}}} \F_{i} \\
& \overset{\cong}{\longrightarrow} & Fu_{i*}u_{ik*}^{i} \\
& \overset{=}{\longrightarrow} & Fu_{ik*} \\
& \overset{\cong}{\longrightarrow} & -\otimes_{\mathcal{O}_{U_{ik}}} \E_{i},
\end{eqnarray*}
commutes.

If $i \neq j$ and $i \neq l$, both routes of (\ref{eqn.focus}) are $0$ by definition of $\delta$ and $\theta$.  If $i=j$, both the top and bottom of (\ref{eqn.focus}) are induced by the unit of $(v_{jkl}^{jk*},v_{jkl*}^{jk})$ so that (\ref{eqn.focus}) commutes in this case as well.  It remains to prove that the diagram
\begin{equation} \label{eqn.focus1}
\begin{CD}
v_{jk*}^{k}v_{jk}^{j*}\F_{j} & \overset{\delta_{j}^{ij}}{\longrightarrow} & v_{ik*}^{k}v_{ijk*}^{ik}v_{ijk}^{ik*}v_{ik}^{i*}\F_{i} \\
@VVV @VVV \\
v_{jk*}^{k}\E_{j} & \underset{(\theta_{Fu_{k*}})_{j}^{ij}}{\longrightarrow} & v_{ik*}^{k}v_{ijk*}^{ik}v_{ijk}^{ik*}\E_{i}
\end{CD}
\end{equation}
whose verticals equal those of (\ref{eqn.focus}), commutes.

By Step 3, (\ref{eqn.focus1}) may be broken up into the diagram
\begin{equation} \label{eqn.fo1}
\begin{CD}
v_{jk*}^{k}v_{jk}^{j*}\F_{j} & \longrightarrow & v_{jk*}^{k}v_{ijk*}^{jk}v_{ijk}^{jk*}v_{jk}^{j*}\F_{j} \\
@VVV @VVV \\
v_{jk*}^{k}\E_{j} & \longrightarrow & v_{jk*}^{k}v_{ijk*}^{jk}v_{ijk}^{jk*}\E_{j}
\end{CD}
\end{equation}
whose horizontals are induced by units, to the left of the diagram
\begin{equation} \label{eqn.fo2}
\begin{CD}
v_{jk*}^{k}v_{ijk*}^{jk}v_{ijk}^{jk*}v_{jk}^{j*}\F_{j} & \longrightarrow & v_{ik*}^{k}v_{ijk*}^{ik}v_{ijk}^{ik*}v_{ik}^{i*}\F_{i} \\
@VVV @VVV \\
v_{jk*}^{k}v_{ijk*}^{jk}v_{ijk}^{jk*}\E_{j} & \longrightarrow & v_{ik*}^{k}v_{ijk*}^{ik}v_{ijk}^{ik*}\E_{i}
\end{CD}
\end{equation}
which is $v_{ijk*}^{k}$ applied to (\ref{eqn.focus0}).  The commutativity of (\ref{eqn.fo1}) is elementary, while the commutativity of (\ref{eqn.fo2}) follows from Step 4.

{\it Step 6:  Retain the notation from Step 5.  We prove that there is an isomorphism
$\rho:W_{\mathfrak{U} \cap U_{k}}(Fu_{k*}) \longrightarrow v_{k}^{*}W_{\mathfrak{U}}(F)$ making the diagram
\begin{equation} \label{eqn.rhocommute}
\begin{CD}
W_{\mathfrak{U} \cap U_{k}}(Fu_{k*}) & \longrightarrow &
\oplus_{i}v_{ik*}^{k}\E_{i} \\
& & @VVV  \\
@V{\rho}VV \oplus_{i}v_{ik*}^{k}v_{ik}^{i*}\F_{i} \\
& & @VVV \\
v_{k}^{*}W_{\mathfrak{U}}(F) & \longrightarrow &
\oplus_{i}v_{k}^{*}v_{i*}\F_{i}
\end{CD}
\end{equation}
whose top vertical is induced by the inverse of the left vertical in (\ref{eqn.focus}) and whose bottom vertical is induced by basechange, commute.  The proof will follow}.

By Step 5 there is an isomorphism
$\rho_1:W_{\mathfrak{U} \cap U_{k}}(Fu_{k*}) \rightarrow \operatorname{ker }\delta$ making the diagram
$$
\begin{CD}
W_{\mathfrak{U} \cap U_{k}}(Fu_{k*}) & \longrightarrow & \oplus_{i}v_{ik*}^{k}\E_{i} \\
@V{\rho_1}VV @VVV \\
\operatorname{ker }\delta & \longrightarrow & \oplus_{i}v_{ik*}^{k}v_{ik}^{i*}\F_{i} \\
\end{CD}
$$
whose right vertical is the upper right vertical in (\ref{eqn.rhocommute}), commute.  By
Step 3, there is an isomorphism $\rho_2:\operatorname{ker }\delta \longrightarrow v_{k}^{*}W_{\mathfrak{U}}(F)$
making the diagram
$$
\begin{CD}
\operatorname{ker }\delta & \longrightarrow & \oplus_{i}v_{ik*}^{k}v_{ik}^{i*}\F_{i} \\
@V{\rho_2}VV @VVV \\
v_{k}^{*}W_{\mathfrak{U}}(F) & \longrightarrow & \oplus_{i}v_{k}^{*}v_{i*}\F_{i}
\end{CD}
$$
whose right vertical is induced by basechange, commute.  We let $\rho=\rho_2 \rho_1$.  Naturality of $\rho$ in $F$
is a straightforward but tedious exercise, which we omit.
\end{proof}
We now work towards a proof of Proposition \ref{prop.sheafid}.  We
begin by introducing some notation and proving a preliminary lemma.

Let $S$ be a scheme with finite open cover $\{W_{i}\}_{i \in I}$
where $I=\{1,\ldots,n\}$ and let $\F$ be an object of ${\sf Qcoh
}S$.  Let
$$
\psi_{ij}:w_{ij}^{i*}w_{i}^{*}\F \overset{\cong}{\longrightarrow}
(w_{i}w_{ij}^{i})^{*}\F \overset{=}{\longrightarrow}
(w_{j}w_{ij}^{j})^{*} \overset{\cong}{\longrightarrow}
w_{ij}^{j*}w_{j}^{*}\F
$$
denote the canonical isomorphism, let
$$
\phi_{i}^{ij}:w_{i*}w_{i}^{*}\F \longrightarrow
w_{i*}w_{ij*}^{i}w_{ij}^{i*}w_{i}^{*}\F
$$
be induced by the unit of $(w_{ij}^{i*},w_{ij*}^{i})$, and let
$\phi_{j}^{ij}=w_{ij*}\psi_{ji} \circ \phi_{j}^{ji}$.  We define
\begin{equation} \label{eqn.delta1}
\delta_{\F}:\oplus_{i}w_{i*}w_{i}^{*}\F \longrightarrow
\oplus_{i<j}w_{i*}w_{ij*}^{i}w_{ij}^{i*}w_{i}^{*}\F
\end{equation}
via its components
$$
\delta_{i}^{jk}:w_{i*}w_{i}^{*}\F \longrightarrow
w_{j*}w_{jk*}^{j}w_{jk}^{j*}w_{j}^{*}\F
$$
as follows:
\begin{equation}
\delta_{i}^{jk} =
\begin{cases}
\phi_{i}^{ik} & \text{if $i=j$},\\
-\phi_{i}^{ji} & \text{if $i=k$, and} \\
0 & \text{otherwise}.
\end{cases}
\end{equation}

\begin{lemma} \label{lemma.descent}
The map $\F \longrightarrow
\oplus_{i} w_{i*}w_{i}^{*}\F$ induced by the units of
$\{(w_{i}^{*},w_{i*})\}_{i}$ is a kernel of $\delta_{\F}$.
\end{lemma}

\begin{proof}
Let $\eta_{\F}:\F \rightarrow \oplus_{i}w_{i*}w_{i}^{*}\F$ be
induced by the units of $\{(w_{i}^{*},w_{i*})\}_{i}$ and let
$$
\phi:\N \rightarrow \oplus_{i}w_{i*}w_{i}^{*}\F
$$
be a morphism such that $\delta_{\F} \phi=0$.  We must show that
there exists a unique $\psi:\N \rightarrow \F$ such that
$\eta_{\F} \psi = \phi$.

For each $i$, $\phi$ has a component
$$
\phi_{i}:\N \rightarrow w_{i*}w_{i}^{*}\F.
$$
By adjointness of $(w_{i}^{*},w_{i*})$, there exists a morphism
$$
\psi_{i}:w_{i}^{*}\N \rightarrow w_{i}^{*}\F
$$
such that $\phi_{i}$ is the composition
$$
\N \longrightarrow w_{i*}w_{i}^{*}\N
\overset{w_{i*}\psi_{i}}{\longrightarrow} w_{i*}w_{i}^{*}\F.
$$
whose left map is the unit.

{\it Step 1:  We show that there exists a unique morphism $\psi:\mathcal{N} \rightarrow \F$ such that $w_{i}^{*}\psi=\psi_{i}$ for all $i$.}  It suffices, by \cite[Section 6.1]{descent}, to show that, for
all pairs $i,j$, the diagram
\begin{equation} \label{eqn.descent0}
\begin{CD}
(w_{i}w_{ij}^{i})^{*}\N & \overset{\cong}{\longrightarrow} &
w_{ij}^{i*}w_{i}^{*}\N &
\overset{w_{ij}^{i*}\psi_{i}}{\longrightarrow} &
w_{ij}^{i*}w_{i}^{*}\F & \overset{\cong}{\longrightarrow} &  (w_{i}w_{ij}^{i})^{*}\F \\
@V{=}VV & & & & @VV{=}V \\
(w_{j}w_{ij}^{j})^{*}\N & \underset{\cong}{\longrightarrow} &
w_{ij}^{j*}w_{j}^{*}\N &
\underset{w_{ij}^{j*}\psi_{j}}{\longrightarrow} &
w_{ij}^{j*}w_{j}^{*}\F & \underset{\cong}{\longrightarrow} &
(w_{j}w_{ij}^{j})^{*}\F
\end{CD}
\end{equation}
whose unlabelled arrows are canonical, commutes.  To this end, we
note that since $\delta_{F} \phi=0$, the diagram
\begin{equation} \label{eqn.descent1}
\begin{CD}
\N & \longrightarrow & w_{i*}w_{i}^{*}\N &
\overset{w_{i*}\psi_{i}}{\longrightarrow} &
w_{i*}w_{i}^{*}\F & \longrightarrow & w_{i*}w_{ij*}^{i}w_{ij}^{i*}w_{i*}\F \\
@V{=}VV & & & & @VV{\cong}V \\
\N & \longrightarrow & w_{j*}w_{j}^{*}\N &
\underset{w_{j*}\psi_{j}}{\longrightarrow} & w_{j*}w_{j}^{*}\F &
\longrightarrow & w_{j*}w_{ij*}^{j}w_{ij}^{j*}w_{j}^{*}\F
\end{CD}
\end{equation}
whose right vertical is canonical and whose other unlabelled
morphisms are units, commutes for all pairs $i,j$.

Applying $w_{ij}^{*}$ to (\ref{eqn.descent1}) yields the
commutative diagram
\begin{equation} \label{eqn.descent2}
\begin{CD}
w_{ij}^{*}\N & \longrightarrow & w_{ij}^{*}w_{i*}w_{i}^{*}\N &
\overset{w_{ij}^{*}w_{i*}\psi_{i}}{\longrightarrow} &
w_{ij}^{*}w_{i*}w_{i}^{*}\F & \longrightarrow & w_{ij}^{*}w_{i*}w_{ij*}^{i}w_{ij}^{i*}w_{i*}\F \\
@V{=}VV & & & & @VV{\cong}V \\
w_{ij}^{*}\N & \longrightarrow & w_{ij}^{*}w_{j*}w_{j}^{*}\N &
\underset{w_{ij}^{*}w_{j*}\psi_{j}}{\longrightarrow} &
w_{ij}^{*}w_{j*}w_{j}^{*}\F & \longrightarrow &
w_{ij}^{*}w_{j*}w_{ij*}^{j}w_{ij}^{j*}w_{j}^{*}\F.
\end{CD}
\end{equation}
Consider the following diagram
\begin{equation} \label{eqn.descent3}
\begin{CD}
w_{ij}^{i*}w_{i}^{*}w_{i*}w_{i}^{*}\N
@>{w_{ij}^{i*}w_{i}^{*}w_{i*}\psi_{i}}>>
w_{ij}^{i*}w_{i}^{*}w_{i*}w_{i}^{*}\F & \longrightarrow &
w_{ij}^{i*}w_{i}^{*}\F \\
@AAA @VVV @VV{=}V \\
w_{ij}^{i*}w_{i}^{*}\N & &
w_{ij}^{i*}w_{i}^{*}w_{i*}w_{ij*}^{i}w_{ij}^{i*}w_{i}^{*}\F &
\longrightarrow & w_{ij}^{i*}w_{i}^{*}\F \\
@V{\cong}VV @V{\cong}VV @VV{\cong}V \\
w_{ij}^{j*}w_{j}^{*}\N & &
w_{ij}^{j*}w_{j}^{*}w_{j*}w_{ij*}^{j}w_{ij}^{j*}w_{j}^{*}\F &
\longrightarrow & w_{ij}^{j*}w_{j}^{*}\F \\
@VVV @AAA @VV{=}V \\
w_{ij}^{j*}w_{j}^{*}w_{j*}w_{j}^{*}\N @>>{w_{ij}^{j*}w_{j}^{*}w_{j*}\psi_{j}}>
w_{ij}^{j*}w_{j}^{*}w_{j*}w_{j}^{*}\F & \longrightarrow &
w_{ij}^{j*}w_{j}^{*}\F
\end{CD}
\end{equation}
whose unadorned arrows are induced by units and counits, and whose
unlabelled isomorphisms are canonical.  It follows from a straightforward computation that the commutativity
of (\ref{eqn.descent2}) implies the commutativity of (\ref{eqn.descent3}).  As one can check, the
outside circuit of this diagram starting at $w_{ij}^{i*}w_{i}^{*}\mathcal{N}$ equals (\ref{eqn.descent0}).

{\it Step 2:  We show that the map $\psi:\mathcal{N} \rightarrow \F$ from Step 1 is unique such that the diagram
\begin{equation} \label{eqn.maindescent}
\begin{CD}
\N & \longrightarrow &
w_{i*}w_{i}^{*}\N \\
@V{\psi}VV @VV{w_{i*}\psi_{i}}V \\
\F & \longrightarrow &
w_{i*}w_{i}^{*}\F
\end{CD}
\end{equation}
whose horizontals are units, commutes for all $i$.}  We first note that $\psi$ makes (\ref{eqn.maindescent}) commute by naturality of the unit of $(w_{i}^{*},w_{i*})$, since $\psi=w_{i}^{*}\psi_{i}$.

We next note that if $\gamma:\mathcal{N}
\rightarrow \F$ replacing $\psi$ in (\ref{eqn.maindescent}) makes (\ref{eqn.maindescent}) commute for all $i$, the
commutativity of the diagram constructed by applying
$w_{i}^{*}$ to (\ref{eqn.maindescent}) and composing on the right
with the counit $w_{i}^{*}w_{i*} \rightarrow
\operatorname{id}_{{\sf Qcoh}S}$ implies that $w_{i}^{*}\gamma =
\psi_{i}$.  Step 1 tells us that $\psi$ is unique with this
property.  Therefore $\gamma=\psi$.

{\it Step 3:  We complete the proof.}  By Step 2, $\psi:\mathcal{N} \rightarrow \F$ is unique making the diagram
$$
\begin{CD}
\N & \longrightarrow &
\oplus_{i}w_{i*}w_{i}^{*}\N \\
@V{\psi}VV @VV{w_{i*}\psi_{i}}V \\
\F & \underset{\eta_{\F}}{\longrightarrow} &
\oplus_{i}w_{i*}w_{i}^{*}\F.
\end{CD}
$$
whose top horizontal is induced by units, commute.  By the construction of $\psi_{i}$, the top route of this diagram is $\phi$.  The result follows.
\end{proof}

\begin{prop} \label{prop.sheafid}
If $\F$ is an object of the category ${\sf Qcoh }X \times Y$ and $F$ is an
object of the category ${\sf Bimod}_{k}(X-Y)$ such that
$F \cong -\otimes_{\mathcal{O}_{X}}\F$, then $W_{\mathfrak{U}}(F) \cong
\F$.
\end{prop}

\begin{proof}
Since $W_{\mathfrak{U}}$ is a functor, we may assume without loss of generality that $F=-\otimes_{\mathcal{O}_{X}}\F$.  Let $\psi_{i}:\F_{i} \longrightarrow
v_{i}^{*}\F$ correspond, via Proposition \ref{prop.wattone},
to the composition
\begin{eqnarray*}
-\otimes_{\mathcal{O}_{U_{i}}}\F_{i} &
\overset{\cong}{\longrightarrow} & Fu_{i*} \\
& \overset{=}{\longrightarrow} & u_{i*}(-)
\otimes_{\mathcal{O}_{X}}\F \\
& \overset{\cong}{\longrightarrow} &
-\otimes_{\mathcal{O}_{U_{i}}}v_{i}^{*}\F
\end{eqnarray*}
whose first arrow is the canonical isomorphism from the proof of
Proposition \ref{prop.wattone}, and whose third arrow is
(\ref{eqn.pullback}).

By Lemma \ref{lemma.descent}, it suffices to prove that the
diagram
$$
\begin{CD}
\oplus_{i}v_{i*}v_{i}^{*}\F &
\overset{\delta_{F}}{\longrightarrow} & \oplus_{i
<j}v_{i*}v_{ij*}^{i}v_{ij}^{i*}v_{i}^{*}\F \\
@V{\oplus_{i} v_{i*}\psi_{i}^{-1}}VV @VV{\oplus_{i
<j}v_{i*}v_{ij*}^{i}v_{ij}^{i*}\psi_{i}^{-1}}V \\
\oplus_{i}v_{i*}\F_{i} & \underset{\theta_{F}}{\longrightarrow} &
\oplus_{i<j}v_{i*}v_{ij*}^{i}v_{ij}^{i*}\F_{i}
\end{CD}
$$
commutes, where we specialize the notation for the definition of
$\delta_{F}$ preceding Lemma \ref{lemma.descent} to our situation by setting $S=X \times Y$,
$W_{i}=U_{i} \times Y$, and $w_{i}=v_{i}$.

We recall that $\delta_{i}^{ij}$ denotes the component of
$\delta_{\mathcal{F}}$ from the $i$th summand to the $i,j$th
summand, and $\theta_{i}^{ij}$ is defined similarly.  The
verification that
$$
v_{i*}v_{ij*}^{i}v_{ij}^{i*}\psi_{i}^{-1} \circ \delta_{i}^{ij} =
\theta_{i}^{ij} \circ v_{i*} \psi_{i}^{-1}
$$
is trivial, so that it remains to check that the diagram
$$
\begin{CD}
v_{j*}\F_{j} & \overset{}{\longrightarrow} &
v_{j*}v_{ij*}^{j}v_{ij}^{j*}\F_{j} &
\overset{v_{ij*}\psi_{ji}}{\longrightarrow} &
v_{i*}v_{ij*}^{i}v_{ij}^{i*}\F_{i} \\
@V{v_{j*}\psi_{j}}VV @V{v_{j*}v_{ij*}^{j}v_{ij}^{j*}\psi_{j}}VV @VV{v_{i*}v_{ij*}^{i}v_{ij}^{i*}\psi_{i}}V \\
v_{j*}v_{j}^{*}\F & \underset{}{\longrightarrow} &
v_{j*}v_{ij*}^{j}v_{ij}^{j*}v_{j}^{*}\F &
\underset{\cong}{\longrightarrow} &
v_{i*}v_{ij*}^{i}v_{ij}^{i*}v_{i}^{*}\F
\end{CD}
$$
whose unadorned arrows are induced by units, whose unlabelled
isomorphism is canonical, and whose upper right horizontal is defined by (\ref{eqn.psidef}), commutes.  The left square commutes by
naturality of units, while to prove the right square commutes, it
suffices to prove that the square
\begin{equation} \label{eqn.gluecanonical}
\begin{CD}
v_{ij}^{j*}\F_{j} & \overset{\psi_{ji}}{\longrightarrow} &
v_{ij}^{i*}\F_{i} \\
@V{v_{ij}^{j*}\psi_{j}}VV @VV{v_{ij}^{i*}\psi_{i}}V \\
v_{ij}^{j*}v_{j}^{*}\F & \underset{\cong}{\longrightarrow} &
v_{ij}^{i*}v_{i}^{*}\F
\end{CD}
\end{equation}
whose unlabeled isomorphism is canonical, commutes.  To prove that (\ref{eqn.gluecanonical}) commutes, it suffices, by Proposition \ref{prop.wattone},
to prove that the diagram resulting in applying the functor
$-\otimes_{\mathcal{O}_{U_{ij}}}(-)$ to (\ref{eqn.gluecanonical})
commutes.  Upon expanding the
resulting diagram, it is straightforward to check that the commutativity of
(\ref{eqn.gluecanonical}) follows from the commutativity of the
diagram
\begin{equation} \label{eqn.diagramlastlast}
\begin{CD}
-\otimes_{\mathcal{O}_{U_{ij}}}v_{ij}^{j*}v_{j}^{*}\F &
\overset{\cong}{\longrightarrow} & u_{ij*}^{j}(-)
\otimes_{\mathcal{O}_{U_{j}}}v_{j}^{*}\F &
\overset{\cong}{\longrightarrow} & u_{j*}u_{ij*}^{j}(-)
\otimes_{\mathcal{O}_{X}}\F \\
@V{\cong}VV & & @VV{=}V \\
-\otimes_{\mathcal{O}_{U_{ij}}}v_{ij}^{i*}v_{i}^{*}\F &
\underset{\cong}{\longrightarrow} & u_{ij*}^{i}(-)
\otimes_{\mathcal{O}_{U_{i}}} v_{i}^{*}\F &
\underset{\cong}{\longrightarrow} & u_{i*}u_{ij*}^{i}(-)
\otimes_{\mathcal{O}_{X}}\F
\end{CD}
\end{equation}
whose left vertical is canonical and whose horizontal isomorphisms
are induced by the inverse of (\ref{eqn.pullback}).  The commutativity of (\ref{eqn.diagramlastlast}) follows from Lemma
\ref{lemma.peel}.
\end{proof}

\begin{cor} \label{cor.fexact}
If $F \in {\sf Bimod}_{k}(X-Y)$ is exact, then $\operatorname{pr}_{1}^{*}(-)\otimes_{\mathcal{O}_{X \times Y}}W_{\mathfrak{U}}(F)$ is
exact.
\end{cor}

\begin{proof}
We first claim that
$-\otimes_{\mathcal{O}_{U_{i}}} v_{i}^{*}W_{\mathfrak{U}}(F)$ is exact.  To prove the claim, we note that by Proposition \ref{prop.wattone},
$Fu_{i*} \cong - \otimes_{\mathcal{O}_{U_{i}}} \mathcal{F}$ for some quasi-coherent
$\mathcal{O}_{U_{i} \times Y}$-module $\mathcal{F}$.  Thus, by Proposition \ref{prop.sheafid} and Proposition \ref{prop.loc},
 $Fu_{i*} \cong -\otimes_{\mathcal{O}_{U_{i}}} v_{i}^{*}W_{\mathfrak{U}}(F)$.  The claim follows.

We now proceed to prove the corollary.  Let $p, q:U_{i} \times Y \rightarrow U_{i},Y$ denote projections.  It suffices to show that, for all $i$,
$v_{i}^{*}(\operatorname{pr}_{1}^{*}(-)\otimes_{\mathcal{O}_{X \times Y}}W_{\mathfrak{U}}(F))$ is exact.  We note that
\begin{eqnarray*}
v_{i}^{*}(\operatorname{pr}_{1}^{*}(-)\otimes_{\mathcal{O}_{X \times Y}}W_{\mathfrak{U}}(F)) & \cong & v_{i}^{*}\operatorname{pr}_{1}^{*}(-)\otimes_{\mathcal{O}_{U_{i} \times Y}}v_{i}^{*}W_{\mathfrak{U}}(F) \\
& \cong & p^{*}u_{i}^{*}(-)\otimes_{\mathcal{O}_{U_{i} \times Y}}v_{i}^{*}W_{\mathfrak{U}}(F).
\end{eqnarray*}
Thus, to complete the proof, it suffices to show that if $\phi:\M \rightarrow \N$
is monic and $V \subset Y$ is affine open, then $q_{*}(p^{*}u_{i}^{*}(\phi)\otimes_{\mathcal{O}_{U_{i} \times Y}} v_{i}^{*}W_{\mathfrak{U}}(F))(V)$ is monic.  But
$q_{*}(p^{*}u_{i}^{*}(-)\otimes_{\mathcal{O}_{U_{i} \times Y}} v_{i}^{*}W_{\mathfrak{U}}(F))$ is exact by the
claim and the corollary follows.
\end{proof}

\section{The Eilenberg-Watts Transformation} \label{section.wattstheory}
Our goal in this section is to prove the generalization of the Eilenberg-Watts Theorem mentioned in Section 1 (Theorem \ref{thm.main}).
Throughout this section, we use the fact that since $X$ is separated, every object of
${\sf Qcoh }X$ is a quotient of a flat object \cite[Lemma 1.1.4]{polish}. We
begin by constructing, for each $F$ in ${\sf Bimod}_{k}(X-Y)$, a
natural transformation
$$
\Gamma_{F}:F \longrightarrow
-\otimes_{\mathcal{O}_{X}}W_{\mathfrak{U}}(F)
$$
which we show is natural in $F$.

The construction of $\Gamma_{F}$ will allow us to describe obstructions to its being an isomorphism
(Corollary \ref{cor.gammaisom}).  It will also follow readily from the construction of $\Gamma_{F}$ that if $F \cong -\otimes_{\mathcal{O}_{X}}\F$ for some
object $\F$ in ${\sf Qcoh }X \times Y$ then $\Gamma_{F}$ is an
isomorphism (Proposition \ref{prop.recover}), and $\Gamma$ is
compatible with affine localization (Proposition
\ref{prop.compatwithaffine}).  As a consequence of this last property, we show that the kernel and cokernel of $\Gamma_{F}$
are totally global (Corollary \ref{cor.totallyglobal}). It follows
immediately that $\Gamma_{F}$ is an isomorphism if $X$ is affine or if $F$ is exact.

\subsection{Construction of the Eilenberg-Watts Transformation}
Let $F$ be an object of ${\sf Bimod}_{k}(X-Y)$.  We construct a
natural transformation
$$
\Gamma_{F}:F \longrightarrow -\otimes_{\mathcal{O}_{X}}
W_{\mathfrak{U}}(F)
$$
and show it is natural in $F$.
\newline
{\it Step 1:  We note that for any morphism $\lambda: \M
\longrightarrow \N$ in ${\sf Qcoh }X$, the canonical
morphism coming from the universal property of the kernel
$$
\pi:F(\operatorname{ker }\lambda) \longrightarrow
\operatorname{ker }F\lambda
$$
is natural in the sense that if
$$
\begin{CD}
\M & \overset{\lambda}{\longrightarrow} & \N \\
@VVV @VVV \\
\M' & \underset{\lambda'}{\longrightarrow} & \N' \\
\end{CD}
$$
commutes, then the induced maps $\iota:F(\operatorname{ker
}\lambda) \longrightarrow F(\operatorname{ker }\lambda')$ and
$\iota':\operatorname{ker }F\lambda \longrightarrow
\operatorname{ker }F\lambda'$ make the diagram
$$
\begin{CD}
F(\operatorname{ker }\lambda) & \longrightarrow &
\operatorname{ker }F\lambda \\
@V{\iota}VV @VV{\iota'}V \\
F(\operatorname{ker }\lambda') & \longrightarrow &
\operatorname{ker }F\lambda'
\end{CD}
$$
whose horizontals are the canonical morphisms, commute.}
\newline
{\it Step 2:  Let $\mathcal{L}$ be a flat object in ${\sf Qcoh
}X$.  We construct a morphism
$$
\Gamma_{F\mathcal{L}}:F(\mathcal{L}) \longrightarrow \mathcal{L}
\otimes_{\mathcal{O}_{X}}W_{\mathfrak{U}}(F)
$$
in the category ${\sf Qcoh}Y$.}  Specialize the notation preceding
Lemma \ref{lemma.descent} to the case that $S=X$ and
$W_{i}=U_{i}$.  By Lemma \ref{lemma.descent}, the morphism
$$
\mathcal{L} \longrightarrow \oplus_{i}u_{i*}u_{i}^{*}\mathcal{L}
$$
induced by unit morphisms is a kernel of
$$
\delta_{\mathcal{L}}:\oplus_{i}u_{i*}u_{i}^{*}\mathcal{L}
\longrightarrow
\oplus_{i<j}u_{i*}u_{ij*}^{i}u_{ij}^{i*}u_{i}^{*}\mathcal{L}.
$$
Let
$$
\pi_{1}:F(\mathcal{L}) \longrightarrow \operatorname{ker
}F(\delta_{\mathcal{L}})
$$
denote the morphism from Step 1.  Let $\gamma_i$ denote the composition
$$
Fu_{i*}u_{i}^{*}\mathcal{L} \longrightarrow
u_{i}^{*}\mathcal{L} \otimes_{\mathcal{O}_{U_{i}}} \F_{i}
\longrightarrow \mathcal{L}\otimes_{\mathcal{O}_{X}}v_{i*}\F_{i}
$$
whose left arrow is the canonical isomorphism from Proposition
\ref{prop.wattone} and whose right arrow is induced by
(\ref{eqn.pullback2}).  Let $\gamma_{ij}$ denote the composition
\begin{eqnarray*}
Fu_{i*}u_{ij*}^{i}u_{ij}^{i*}u_{i}^{*}\mathcal{L} &
\overset{\cong}{\longrightarrow} &
u_{ij*}^{i}u_{ij}^{i*}u_{i}^{*}\mathcal{L}
\otimes_{\mathcal{O}_{U_{i}}} \F_{i} \\
& \overset{\cong}{\longrightarrow} &
u_{ij}^{i*}u_{i}^{*}\mathcal{L}
\otimes_{\mathcal{O}_{U_{ij}}}v_{ij}^{i*}\F_{i} \\
& \overset{\cong}{\longrightarrow} & u_{i}^{*}\mathcal{L}
\otimes_{\mathcal{O}_{U_{i}}}v_{ij*}^{i}v_{ij}^{i*}\F_{i} \\
& \overset{\cong}{\longrightarrow} & \mathcal{L}
\otimes_{\mathcal{O}_{X}} v_{i*}v_{ij*}^{i}v_{ij}^{i*}\F_{i}
\end{eqnarray*}
whose first arrow is from Proposition \ref{prop.wattone}, whose
second arrow is induced by (\ref{eqn.pullback}) and whose third and fourth
arrows are induced by (\ref{eqn.pullback2}).

We first claim
\begin{equation} \label{eqn.firsteq}
\gamma_{ij} \circ F\delta_{i}^{ij} = \mathcal{L}
\otimes_{\mathcal{O}_{X}} \theta_{i}^{ij} \circ \gamma_{i},
\end{equation}
for all $i<j$, where $\theta_{i}^{ij}$ is defined in (\ref{eqn.babytheta}).  To prove the claim, consider
the following diagram
$$
\begin{CD}
Fu_{i*}u_{i}^{*}\mathcal{L} & \longrightarrow &
Fu_{i*}u_{ij*}^{i}u_{ij}^{i*}u_{i}^{*}\mathcal{L} \\
@VVV @VVV \\
u_{i}^{*}\mathcal{L} \otimes_{\mathcal{O}_{U_{i}}} \F_{i} &
\longrightarrow & u_{ij*}^{i}u_{ij}^{i*}u_{i}^{*}\mathcal{L}
\otimes_{\mathcal{O}_{U_{i}}} \F_{i} \\
@VVV @VVV \\
\mathcal{L} \otimes_{\mathcal{O}_{X}} v_{i*}\F_{i} & &
u_{ij}^{i*}u_{i}^{*}\mathcal{L}
\otimes_{\mathcal{O}_{U_{ij}}}v_{ij}^{i*}\F_{i} \\
@VVV @VVV \\
\mathcal{L} \otimes_{\mathcal{O}_{X}}
v_{i*}v_{ij*}^{i}v_{ij}^{i*}\F_{i} & \longrightarrow &
u_{i}^{*}\mathcal{L} \otimes_{\mathcal{O}_{U_{i}}}
v_{ij*}^{i}v_{ij}^{i*} \F_{i}
\end{CD}
$$
whose two top horizontals and bottom-left vertical are induced by
the units, whose top verticals are from Proposition
\ref{prop.wattone}, whose left-middle vertical is induced by
(\ref{eqn.pullback2}), whose right-middle vertical is induced by
(\ref{eqn.pullback}), whose right-bottom vertical is induced by (\ref{eqn.pullback2}) and whose bottom
horizontal is induced by the inverse of (\ref{eqn.pullback2}).  The claim will follow from the commutativity of this
diagram.  The top square
commutes by the naturality of the canonical isomorphism from
Proposition \ref{prop.wattone}.  To show that the bottom rectangle commutes,
we first split it down the diagonal via the morphism
$$
u_{i}^{*}\mathcal{L} \otimes_{\mathcal{O}_{U_{i}}} \F_{i}
\longrightarrow u_{i}^{*}\mathcal{L} \otimes_{\mathcal{O}_{U_{i}}}
v_{ij*}^{i}v_{ij}^{i*}\F_{i}
$$
induced by the unit of $(v_{ij}^{i*},v_{ij*}^{i})$.  The resulting left
subdiagram commutes by the naturality of (\ref{eqn.pullback2}),
while the right subdiagram commutes by the commutativity of
(\ref{eqn.bigcompat0}).

We next claim
$$
\gamma_{ij} \circ F\delta_{j}^{ij} = \mathcal{L}
\otimes_{\mathcal{O}_{X}} \theta_{j}^{ij} \circ \gamma_{j}.
$$
To prove the claim, consider the following diagram
$$
\begin{CD}
Fu_{j*}u_{ij*}^{j}u_{ij}^{j*}u_{j}^{*}\mathcal{L} &
\longrightarrow &
Fu_{i*}u_{ij*}^{i}u_{ij}^{i*}u_{i}^{*}\mathcal{L} \\
@VVV @VVV \\
u_{ij*}^{j}u_{ij}^{j*}u_{j}^{*}\mathcal{L}
\otimes_{\mathcal{O}_{U_{j}}} \F_{j} & &
u_{ij*}^{i}u_{ij}^{i*}u_{i}^{*}\mathcal{L}
\otimes_{\mathcal{O}_{U_{i}}} \F_{i} \\
@VVV @VVV \\
u_{ij}^{j*}u_{j}^{*}\mathcal{L} \otimes_{\mathcal{O}_{U_{ij}}}
v_{ij}^{j*}\F_{j} & & u_{ij}^{i*}u_{i}^{*}\mathcal{L}
\otimes_{\mathcal{O}_{U_{ij}}} v_{ij}^{i*}\F_{i} \\
@VVV @VVV \\
u_{j}^{*}\mathcal{L} \otimes_{\mathcal{O}_{U_{j}}}
v_{ij*}^{j}v_{ij}^{j*}\F_{j} & & u_{i}^{*}\mathcal{L}
\otimes_{\mathcal{O}_{U_{i}}} v_{ij*}^{i}v_{ij}^{i*}\F_{i} \\
@VVV @VVV \\
\mathcal{L} \otimes_{\mathcal{O}_{X}}
v_{j*}v_{ij*}^{j}v_{ij}^{j*}\F_{j} & \longrightarrow & \mathcal{L}
\otimes_{\mathcal{O}_{X}} v_{j*}v_{ij*}^{i}v_{ij}^{i*}\F_{i}
\end{CD}
$$
whose top horizontal is induced by the canonical isomorphism
$$
u_{ij}^{j*}u_{j}^{*} \overset{\cong}{\longrightarrow}
(u_{j}u_{ij}^{j})^{*} \overset{=}{\longrightarrow}
(u_{i}u_{ij}^{i})^{*} \overset{\cong}{\longrightarrow}
u_{ij}^{i*}u_{i}^{*}
$$
whose verticals are induced by (\ref{eqn.pullback}) and
(\ref{eqn.pullback2}), and whose bottom horizontal is induced by the map
$$
\psi_{ji}:v_{ij}^{j*}\F_{j} \overset{\cong}{\longrightarrow}
v_{ij}^{i*}\F_{i}
$$
defined after (\ref{eqn.psidef}).  Since (\ref{eqn.firsteq}) holds when $i$ and $j$ are
interchanged, the proof of the claim follows from the commutativity of this diagram.  This follows easily
from the definition of
$\psi_{ji}$.

Next, consider the following diagram
\begin{equation} \label{eqn.kerneldiagram}
\begin{CD}
F(\oplus_i u_{i*}u_{i}^{*}\mathcal{L}) & \overset{F\delta_{\mathcal{L}}}{\longrightarrow} & F(\oplus_{i<j}u_{i*}u_{ij*}^{i}u_{ij}^{i*}u_{i}^{*}\mathcal{L}) \\
@VVV @VVV \\
\oplus_{i} Fu_{i*}u_{i}^{*}\mathcal{L} & \longrightarrow & \oplus_{i<j}Fu_{i*}u_{ij*}^{i}u_{ij}^{i*}u_{i}^{*}\mathcal{L} \\
@V{\oplus_{i}\gamma_i}VV @VV{\oplus_{i<j}\gamma_{ij}}V \\
\oplus_i \mathcal{L} \otimes_{\mathcal{O}_{X}}v_{i*}\mathcal{F}_{i} & \longrightarrow & \oplus_{i<j} \mathcal{L}\otimes_{\mathcal{O}_{X}} v_{i*}v_{ij*}^{i}v_{ij}^{i*}\mathcal{F}_{i} \\
@VVV @VVV \\
\mathcal{L}\otimes_{\mathcal{O}_{X}} ( \oplus_{i} v_{i*}\mathcal{F}_{i}) & \underset{\mathcal{L} \otimes_{\mathcal{O}_{X}} \theta_{F}}{\longrightarrow} & \mathcal{L}\otimes_{\mathcal{O}_{X}} (\oplus_{i<j} v_{i*}v_{ij*}^{i}v_{ij}^{i*}\mathcal{F}_{i})
\end{CD}
\end{equation}
whose second horizontal is induced by the maps $F\delta_i^{jk}$,
whose third horizontal is induced by the maps $\mathcal{L}\otimes_{\mathcal{O}_{X}} \theta_{i}^{jk}$
and whose corner verticals are canonical isomorphisms.
It follows from the claims that the center square in the diagram commutes.
Since the top and bottom square of (\ref{eqn.kerneldiagram}) commute, there is an induced isomorphism
$$
\pi_{2}: \operatorname{ker }F\delta_{\mathcal{L}}
\overset{\cong}{\longrightarrow} \operatorname{ker }(\mathcal{L}
\otimes_{\mathcal{O}_{X}} \theta_{F}).
$$
Finally, since $\mathcal{L}$ is flat and $\operatorname{pr}_{2*}$
is left-exact, there is a canonical isomorphism
$$
\pi_{3}:\operatorname{ker }(\mathcal{L} \otimes_{\mathcal{O}_{X}}
\theta_{F}) \overset{\cong}{\longrightarrow} \mathcal{L}
\otimes_{\mathcal{O}_{X}} \operatorname{ker }\theta_{F}.
$$
We define
$$
\Gamma_{F\mathcal{L}} = \pi_{3} \pi_{2} \pi_{1}.
$$
\newline
{\it Step 3:  We show $\Gamma_{F}$ is natural on flats, i.e. we
show that if
$$
\psi: \mathcal{L} \longrightarrow \mathcal{L}'
$$
is a morphism of flat objects in ${\sf Qcoh }X$ then the diagram
\begin{equation} \label{eqn.naturalonflats}
\begin{CD}
F\mathcal{L} & \overset{F\psi}{\longrightarrow} & F\mathcal{L}' \\
@V{\Gamma_{F\mathcal{L}}}VV @VV{\Gamma_{F\mathcal{L}'}}V \\
\mathcal{L}\otimes_{\mathcal{O}_{X}} W_{\mathfrak{U}}(F) &
\underset{\psi \otimes_{\mathcal{O}_{X}}
W_{\mathfrak{U}}(F)}{\longrightarrow} &
\mathcal{L}'\otimes_{\mathcal{O}_{X}} W_{\mathfrak{U}}(F)
\end{CD}
\end{equation}
commutes.}  We leave it as an easy exercise for the reader to
check that the diagram
$$
\begin{CD}
F(\oplus_{i}u_{i*}u_{i}^{*}\mathcal{L}) &
\overset{F\delta_{\mathcal{L}}}{\longrightarrow} &
F(\oplus_{i<j}u_{i*}u_{ij*}^{i}u_{ij}^{i*}u_{i}^{*}\mathcal{L}) \\
@VVV @VVV \\
F(\oplus_{i}u_{i*}u_{i}^{*}\mathcal{L}') &
\underset{F\delta_{\mathcal{L}'}}{\longrightarrow} &
F(\oplus_{i<j}u_{i*}u_{ij*}^{i}u_{ij}^{i*}u_{i}^{*}\mathcal{L}')
\end{CD}
$$
whose verticals are induced by $\psi$, commutes.  Therefore, by
Step 1, the induced morphism $\psi':\operatorname{ker
}F\delta_{\mathcal{L}} \longrightarrow \operatorname{ker
}F\delta_{\mathcal{L}'}$ makes the diagram
$$
\begin{CD}
F\mathcal{L}=F\operatorname{ker }\delta_{\mathcal{L}} &
\longrightarrow & \operatorname{ker }F\delta_{\mathcal{L}} \\
@V{F\psi}VV @VV{\psi'}V \\
F\mathcal{L}'=F\operatorname{ker }\delta_{\mathcal{L}'} &
\longrightarrow & \operatorname{ker }F\delta_{\mathcal{L}'}
\end{CD}
$$
whose horizontals are canonical, commute.  Thus, the top square in the diagram
$$
\begin{CD}
F \operatorname{ker }\delta_{\mathcal{L}} & \overset{F \psi}{\longrightarrow} & F \operatorname{ker }\delta_{\mathcal{L}'} \\
@V{\pi_{1}}VV @VV{\pi_{1}}V \\
\operatorname{ker }F \delta_{\mathcal{L}} & \underset{\psi'}{\longrightarrow} & \operatorname{ker }F\delta_{\mathcal{L}'} \\
@V{\pi_{2}}VV @VV{\pi_{2}}V \\
\operatorname{ker }(\mathcal{L} \otimes_{\mathcal{O}_{X}} \theta_{F}) & \longrightarrow & \operatorname{ker }(\mathcal{L}' \otimes_{\mathcal{O}_{X}} \theta_{F}) \\
@V{\pi_{3}}VV @VV{\pi_{3}}V \\
\mathcal{L}\otimes_{\mathcal{O}_{X}} \operatorname{ker }\theta_{F} & \longrightarrow & \mathcal{L}'\otimes_{\mathcal{O}_{X}} \operatorname{ker }\theta_{F}
\end{CD}
$$
whose verticals are defined in Step 2 and whose bottom two horizontals are induced by $\psi$, commutes.  The proofs that the middle and bottom squares of this diagram commute are left as straightforward exercises.
\newline
{\it Step 4:  We show that, for each $\M$ in ${\sf Qcoh }X$ and each flat presentation
\begin{equation} \label{eqn.presentation}
\mathcal{L}_{1} \overset{\xi_{1}}{\longrightarrow} \mathcal{L}_{0}
\overset{\xi_{0}}{\longrightarrow} \M,
\end{equation}
there exists a unique morphism
$$
\gamma_{F\M}:F\M \longrightarrow \M
\otimes_{\mathcal{O}_{X}}W_{\mathfrak{U}}(F)
$$
making
\begin{equation} \label{eqn.absolutecom}
\begin{CD}
F\mathcal{L}_{0} & \overset{F\xi_{0}}{\longrightarrow} &
F\mathcal{M} \\
@V{\Gamma_{F\mathcal{L}_{0}}}VV  @VV{\gamma_{F\M}}V  \\
\mathcal{L}_{0} \otimes_{\mathcal{O}_{X}} W_{\mathfrak{U}}(F) &
\underset{\xi_{0}\otimes_{\mathcal{O}_{X}}W_{\mathfrak{U}}(F)}{\longrightarrow}
& \mathcal{M} \otimes_{\mathcal{O}_{X}} W_{\mathfrak{U}}(F)
\end{CD}
\end{equation}
commute.}  Applying $F$ to the flat presentation (\ref{eqn.presentation})
yields the first row in the diagram
\begin{equation} \label{eqn.wattlike}
\begin{CD}
F\mathcal{L}_{1} & \overset{F\xi_{1}}{\longrightarrow} & F\mathcal{L}_{0} & \overset{F\xi_{0}}{\longrightarrow} & F\mathcal{M}  \\
@V{\Gamma_{F\mathcal{L}_{1}}}VV  @VV{\Gamma_{F\mathcal{L}_{0}}}V  \\
\mathcal{L}_{1} \otimes_{\mathcal{O}_{X}}W_{\mathfrak{U}}(F) &
\underset{\xi_{1}\otimes_{\mathcal{O}_{X}}W_{\mathfrak{U}}(F)}{\longrightarrow}
& \mathcal{L}_{0} \otimes_{\mathcal{O}_{X}} W_{\mathfrak{U}}(F) &
\underset{\xi_{0}\otimes_{\mathcal{O}_{X}}W_{\mathfrak{U}}(F)}{\longrightarrow}
& \mathcal{M} \otimes_{\mathcal{O}_{X}} W_{\mathfrak{U}}(F)
\end{CD}
\end{equation}
which commutes by Step 3.  Thus, there exists a unique morphism
$$
\gamma_{F\M}:F \mathcal{M} \longrightarrow \mathcal{M}
\otimes_{\mathcal{O}_{X}} W_{\mathfrak{U}}(F)
$$
making (\ref{eqn.absolutecom}) commute.

We will show, in Step 6, that $\gamma_{F\M}$ is
independent of presentation chosen.
\newline
{\it Step 5: We show that if $\phi:\M \rightarrow \N$ is a
morphism in ${\sf Qcoh }X$, then the diagram
$$
\begin{CD}
F\M & \overset{F\phi}{\longrightarrow} & F\N \\
@V{\gamma_{F\M}}VV @VV{\gamma_{F\N}}V \\
\M\otimes_{\mathcal{O}_{X}}W_{\mathfrak{U}} & \underset{\phi
\otimes_{\mathcal{O}_{X}}W_{\mathfrak{U}}(F)}{\longrightarrow} &
\N\otimes_{\mathcal{O}_{X}}W_{\mathfrak{U}}
\end{CD}
$$
commutes.}  Suppose
$$
\mathcal{L}_{1}' \longrightarrow \mathcal{L}_{0}'
\overset{\pi'}{\longrightarrow} \N
$$
is a flat presentation for $\N$ and let $\gamma_{F\mathcal{N}}:F \mathcal{N} \longrightarrow \mathcal{N}
\otimes_{\mathcal{O}_{X}} W_{\mathfrak{U}}(F)$ denote the corresponding morphism constructed in Step 4.  Then there exists a flat presentation
\begin{equation} \label{eqn.newnewnewflat}
\mathcal{L} \longrightarrow \mathcal{L}_{0} \oplus
\mathcal{L}_{0}' \overset{ \phi \pi \oplus \pi'}{\longrightarrow}
\N
\end{equation}
for $\N$, and the corresponding morphism $\gamma_{F\N}'$ constructed
in Step 4 makes the outer circuit of the diagram
\begin{equation} \label{eqn.outerouter}
\begin{CD}
F(\mathcal{L}_{0} \oplus \mathcal{L}_{0}') & \overset{F(\pi \oplus
\operatorname{id}_{\mathcal{L}_{0}'})}{\longrightarrow} & F(\M
\oplus \mathcal{L}_{0}') & \overset{F(\phi \oplus
\pi')}{\longrightarrow} & F\N \\
@V{\Gamma_{F(\mathcal{L}_{0} \oplus \mathcal{L}_{0}')}}VV
@V{\Gamma_{F(\M \oplus \mathcal{L}_{0}')}}VV @VV{\gamma_{F\N}'}V
\\
(\mathcal{L}_{0} \oplus
\mathcal{L}_{0}')\otimes_{\mathcal{O}_{X}}W_{\mathfrak{U}}(F) &
\longrightarrow & (\M \oplus
\mathcal{L}_{0}')\otimes_{\mathcal{O}_{X}}W_{\mathfrak{U}}(F) &
\longrightarrow & \N \otimes_{\mathcal{O}_{X}}W_{\mathfrak{U}}(F)
\end{CD}
\end{equation}
whose bottom-left horizontal is induced  by $\pi \oplus
\operatorname{id}_{\mathcal{L}_{0}'}$ and whose bottom-right
horizontal is induced by $\phi \oplus \pi'$, commute.

It follows from the commutativity of the outer circuit of (\ref{eqn.outerouter}) and from Step 3 that the outer circuit of the diagram constructed by placing the diagram
\begin{equation} \label{eqn.expand1}
\begin{CD}
F \mathcal{L}_{0} \oplus F\mathcal{L}_{0}' & \overset{F \pi
\oplus {\operatorname{id}}_{F\mathcal{L}_{0}'}}{\longrightarrow} &
F\M \oplus F\mathcal{L}_{0}' \\
@V{\Gamma_{F\mathcal{L}_{0}} \oplus \Gamma_{F\mathcal{L}_{0}'}}VV @V{\gamma_{F\M} \oplus \Gamma_{F\mathcal{L}_{0}'}}VV  \\
(\mathcal{L}_{0} \otimes_{\mathcal{O}_{X}}W_{\mathfrak{U}}(F))
\oplus (\mathcal{L}_{0}'
\otimes_{\mathcal{O}_{X}}W_{\mathfrak{U}}(F)) & \longrightarrow &
(\M \otimes_{\mathcal{O}_{X}}W_{\mathfrak{U}}(F)) \oplus
(\mathcal{L}_{0}' \otimes_{\mathcal{O}_{X}}W_{\mathfrak{U}}(F))
\end{CD}
\end{equation}
whose bottom horizontal is $(\pi
\otimes_{\mathcal{O}_{X}}W_{\mathfrak{U}}(F)) \oplus
(\operatorname{id}_{\mathcal{L}_{0}'}\otimes_{\mathcal{O}_{X}}W_{\mathfrak{U}}(F))$,
to the left of the diagram
\begin{equation} \label{eqn.expand2}
\begin{CD}
F\M \oplus F\mathcal{L}_{0}' & \overset{F \phi \oplus F
\pi'}{\longrightarrow} & F\N \\
@V{\gamma_{F\M} \oplus \Gamma_{F\mathcal{L}_{0}'}}VV @VV{\gamma_{F\N}'}V \\
(\M \otimes_{\mathcal{O}_{X}}W_{\mathfrak{U}}(F)) \oplus
(\mathcal{L}_{0}' \otimes_{\mathcal{O}_{X}}W_{\mathfrak{U}}(F)) &
\longrightarrow & \N \otimes_{\mathcal{O}_{X}}W_{\mathfrak{U}}(F)
\end{CD}
\end{equation}
whose bottom horizontal is induced by $\phi
\otimes_{\mathcal{O}_{X}}W_{\mathfrak{U}}(F)$ and $\pi'
\otimes_{\mathcal{O}_{X}}W_{\mathfrak{U}}(F)$, commutes. We note also that
the diagram (\ref{eqn.expand1}) commutes since Step 4 implies that (\ref{eqn.absolutecom}) commutes.
Since the top horizontal in (\ref{eqn.expand1}) is an epimorphism, it follows that (\ref{eqn.expand2})
commutes as well. By restricting both routes of (\ref{eqn.expand2})
to $F\mathcal{L}_{0}'$ and using the fact, established in Step 4, that $\gamma_{F\N}$ is unique making the diagram
$$
\begin{CD}
F\mathcal{L}_{0}' & \overset{F\pi'}{\longrightarrow} &
F\mathcal{N} \\
@V{\Gamma_{F\mathcal{L}_{0}'}}VV  @VV{\gamma_{F\N}}V  \\
\mathcal{L}_{0}' \otimes_{\mathcal{O}_{X}} W_{\mathfrak{U}}(F) &
\underset{\pi' \otimes_{\mathcal{O}_{X}}W_{\mathfrak{U}}(F)}{\longrightarrow}
& \mathcal{N} \otimes_{\mathcal{O}_{X}} W_{\mathfrak{U}}(F)
\end{CD}
$$
commute, we have $\gamma_{F\N}=\gamma_{F\N}'$.  On the other hand,
restricting both routes of (\ref{eqn.expand2}) to $F\M$ allows us
to conclude that
$$
(\phi \otimes_{\mathcal{O}_{X}} W_{\mathfrak{U}}(F)) \gamma_{F\M} =
\gamma'_{F\N} F\phi.
$$
Step 5 follows.
\newline
{\it Step 6: We show that $\gamma_{F\mathcal{M}}$ is independent
of presentation.}  Let $\gamma_{F\M}':F\M \longrightarrow \M
\otimes_{\mathcal{O}_{X}}W_{\mathfrak{U}}(F)$ denote the morphism
constructed in Step 4 using a flat presentation
$$
\mathcal{L}_{1}' \longrightarrow \mathcal{L}_{0}' \longrightarrow
\M.
$$
Now apply Step 5 to conclude that the diagram
$$
\begin{CD}
F\M & \overset{F\operatorname{id}_{\M}}{\longrightarrow} & F\M \\
@V{\gamma_{F\M}}VV @VV{\gamma_{F\M}'}V \\
\M \otimes_{\mathcal{O}_{X}}W_{\mathfrak{U}}(F) &
\underset{\operatorname{id}_{\M}
\otimes_{\mathcal{O}_{X}}W_{\mathfrak{U}}(F)}{\longrightarrow} &
\M \otimes_{\mathcal{O}_{X}}W_{\mathfrak{U}}(F)
\end{CD}
$$
commutes.  Step 6 follows.

We define
$$
\Gamma_{F\mathcal{M}}:=\gamma_{F\mathcal{M}}.
$$
\newline
{\it Step 7: We show that $\Gamma_{F}$ is natural in $\M$.}  This
follows from Step 5 in light of the definition of $\Gamma_{F\M}$
given in Step 6.
\newline
{\it Step 8: We show $\Gamma_{F}$ is natural in $F$.} It suffices
to check that if $\mathcal{L}$ is a flat object in ${\sf Qcoh }X$
and $\eta:F \rightarrow G$ is a morphism in ${\sf Bimod}_{k}(X-Y)$
then the diagram
\begin{equation} \label{eqn.natonflat}
\begin{CD}
F(\mathcal{L}) & \overset{\eta_{\mathcal{L}}}{\longrightarrow} &
G(\mathcal{L}) \\
@V{\Gamma_{F\mathcal{L}}}VV @VV{\Gamma_{G\mathcal{L}}}V \\
\mathcal{L} \otimes_{\mathcal{O}_{X}} W_{\mathfrak{U}}(F) &
\underset{\mathcal{L}
\otimes_{\mathcal{O}_{X}}W_{\mathfrak{U}}(\eta)}{\longrightarrow}
& \mathcal{L} \otimes_{\mathcal{O}_{X}}W_{\mathfrak{U}}(G)
\end{CD}
\end{equation}
commutes.  Sufficiency follows from the right exactness of $F$.
The proof that (\ref{eqn.natonflat}) commutes is straightforward,
and we omit it.

\subsection{Properties of the Eilenberg-Watts Transformation}
As in the previous subsection, we specialize the notation preceding
Lemma \ref{lemma.descent} to the case that $S=X$ and
$W_{i}=U_{i}$.  Let $\mathcal{M}$ be an object in ${\sf Qcoh }X$.  By Lemma \ref{lemma.descent}, the morphism
$$
\mathcal{M} \longrightarrow \oplus_{i}u_{i*}u_{i}^{*}\mathcal{M}
$$
induced by unit morphisms is a kernel of
$$
\delta_{\mathcal{M}}:\oplus_{i}u_{i*}u_{i}^{*}\mathcal{M}
\longrightarrow
\oplus_{i<j}u_{i*}u_{ij*}^{i}u_{ij}^{i*}u_{i}^{*}\mathcal{M}.
$$
Throughout this subsection, $F$ is assumed to be an object in ${\sf Bimod}_{k}(X-Y)$.

\begin{prop} \label{prop.gammaonflat}
If $\mathcal{L}$ is a flat object in ${\sf Qcoh }X$, then $\Gamma_{F\mathcal{L}}$ is an isomorphism if and only if the canonical map $F\operatorname{ker }\delta_{\mathcal{L}} \rightarrow \operatorname{ker }F \delta_{\mathcal{L}}$ is an isomorphism.
\end{prop}

\begin{proof}
The map $\Gamma_{F\mathcal{L}}$ is a composition of the canonical map $F\operatorname{ker }\delta_{\mathcal{L}} \rightarrow \operatorname{ker }F \delta_{\mathcal{L}}$ and two isomorphisms, by Step 2 of the construction of $\Gamma$.
\end{proof}
The next result follows from Proposition \ref{prop.gammaonflat} and a straightforward diagram chase.
\begin{cor} \label{cor.gammaisom}
If $F \in {\sf Bimod}_{k}(X-Y)$ then $\Gamma_{F}$ is an isomorphism if and only if
\begin{enumerate}

\item{} for all flat objects $\mathcal{L}$ in ${\sf Qcoh }X$, the canonical map $F\operatorname{ker }\delta_{\mathcal{L}} \rightarrow \operatorname{ker }F \delta_{\mathcal{L}}$ is an isomorphism, and

\item{} $-\otimes_{\mathcal{O}_{X}}W_{\mathfrak{U}}(F)$ is right exact.

\end{enumerate}
\end{cor}

\begin{cor} \label{cor.exactvanish}
Let $F$ be a totally global, exact
functor such that $-\otimes_{\mathcal{O}_{X}}W_{\mathfrak{U}}(F)$
is right exact. Then $F=0$.
\end{cor}

\begin{proof}
Since $F$ is exact and $-\otimes_{\mathcal{O}_{X}}W_{\mathfrak{U}}(F)$ is right exact, $F \cong -\otimes_{\mathcal{O}_{X}}W_{\mathfrak{U}}(F)$ by Corollary \ref{cor.gammaisom}.
Thus, since $F$ is totally global, $F=0$ by Proposition
\ref{prop.vanish}.
\end{proof}

\begin{prop} \label{prop.recover}
If $F \cong -\otimes_{\mathcal{O}_{X}}\F$ for some object $\F$ in ${\sf
Qcoh }X \times Y$, then $\Gamma_{F}$ is an isomorphism.
\end{prop}

\begin{proof}
By the naturality of $\Gamma$ (noted in Step 8 of the construction of $\Gamma$) we may assume without loss of generality that $F=-\otimes_{\mathcal{O}_{X}}\F$.  By Proposition \ref{prop.sheafid}, $W_{\mathfrak{U}}(F) \cong \F$.
Since $F$ is right exact, so is $-\otimes_{\mathcal{O}_{X}}
W_{\mathfrak{U}}(F)$.  Hence, by Corollary \ref{cor.gammaisom}, it suffices to show that if $\mathcal{L}$ is a flat object in ${\sf Qcoh }X$, then the canonical map $F(\mathcal{L})=F(\operatorname{ker }\delta_{\mathcal{L}})
\longrightarrow \operatorname{ker }F\delta_{\mathcal{L}}$ is an isomorphism.  To prove this, we note that in Step 2 of the construction of $\Gamma$ we constructed an isomorphism
\begin{equation} \label{eqn.pi2pi3}
\pi_2^{-1} \pi_{3}^{-1}:(\operatorname{ker }\delta_{\mathcal{L}}) \otimes_{\mathcal{O}_{X}} \mathcal{F} \longrightarrow \operatorname{ker }(\delta_{\mathcal{L}} \otimes_{\mathcal{O}_{X}} \mathcal{F}).
\end{equation}
Hence, to complete the proof of the proposition, it suffices to prove that (\ref{eqn.pi2pi3}) is the canonical map induced by the universal property of the kernel.  This fact follows from Lemma \ref{lemma.tensorcomp}, as one can check.
\end{proof}

\begin{cor} \label{cor.approx}
Let $\mathcal{F}'$ be an object of ${\sf Qcoh }X \times Y$ such that $F' := -\otimes_{\mathcal{O}_{X}}\F'$ is an object in
${\sf Bimod}_{k}(X-Y)$.  If $\Phi:F \rightarrow F'$ is a morphism
in ${\sf Bimod}_{k}(X-Y)$, then $\Phi$ factors through
$\Gamma_{F}$.
\end{cor}

\begin{proof}
Since $\Gamma_{G}$ is natural in $G$, the diagram
$$
\begin{CD}
F & \overset{\Phi}{\longrightarrow} & F' \\
@V{\Gamma_{F}}VV @VV{\Gamma_{F'}}V \\
-\otimes_{\mathcal{O}_{X}}W_{\mathfrak{U}}(F) &
\underset{-\otimes_{\mathcal{O}_{X}}W_{\mathfrak{U}}(\Phi)}{\longrightarrow}
& -\otimes_{\mathcal{O}_{X}}W_{\mathfrak{U}}(F')
\end{CD}
$$
commutes.  Since $\Gamma_{F'}$ is an isomorphism by Proposition
\ref{prop.recover}, the assertion follows.
\end{proof}

\begin{prop} \label{prop.compatwithaffine}
Let $\Gamma_{F}
* u_{k*}$ denote the horizontal composition of the natural transformations $\Gamma_{F}$ and $\operatorname{id}_{u_{k*}}$.
Then $\Gamma_{F}$ is compatible with affine localization, i.e.
the diagram
\begin{equation} \label{eqn.affineloc}
\begin{CD}
Fu_{k*} & \overset{\Gamma_{F}*u_{k*}}{\longrightarrow} & u_{k*}(-)
\otimes_{{\mathcal{O}}_{X}}W_{\mathfrak{U}}(F) \\
@V{\Gamma_{Fu_{k*}}}VV  @VVV \\
-\otimes_{{\mathcal{O}_{U_{k}}}}W_{\mathfrak{U} \cap
U_{k}}(Fu_{k*}) & \longrightarrow & -
\otimes_{{\mathcal{O}}_{U_{k}}} v_{k}^{*}W_{\mathfrak{U}}(F)
\end{CD}
\end{equation}
whose bottom horizontal is induced by the isomorphism constructed
in Proposition \ref{prop.loc} and whose right vertical is induced
by the isomorphism (\ref{eqn.pullback}), commutes for all $k$.
\end{prop}

\begin{proof}
We prove the result in several steps.
\newline
{\it Step 1:  We show that it suffices to prove that
(\ref{eqn.affineloc}) commutes when applied to flat objects of
${\sf Qcoh }U_{k}$.}  For, if $\pi:\mathcal{L} \rightarrow \M$ is
an epimorphism in ${\sf Qcoh }U_{k}$ where $\mathcal{L}$ flat,
then, since the arrows in (\ref{eqn.affineloc}) are natural, and
since $Fu_{k*}$ is right exact, Step 1 follows from a standard
diagram chase.
\newline
{\it Step 2:  Consider the following diagram
\begin{equation} \label{eqn.finaldig1}
\begin{CD}
F & \longrightarrow & \oplus_{i}Fu_{i*}u_{i}^{*} \\
& & @VVV \\
@V{\Gamma_{F}}VV
\oplus_{i}u_{i}^{*}(-)\otimes_{\mathcal{O}_{U_{i}}}\F_{i}
\\
& & @VVV \\
-\otimes_{\mathcal{O}_{X}}W_{\mathfrak{U}}(F) & \longrightarrow &
\oplus_{i} - \otimes_{\mathcal{O}_{X}}v_{i*}\F_{i}
\end{CD}
\end{equation}
whose top horizontal is induced by a unit, whose top vertical is
induced by the canonical isomorphism from Proposition
\ref{prop.wattone}, whose bottom vertical is induced by
(\ref{eqn.pullback2}), and whose bottom horizontal comes from the
definition of $W_{\mathfrak{U}}(F)$ as a kernel.  We note that this diagram commutes.}  We
first note that (\ref{eqn.finaldig1}) commutes on flats by the
definition of $\Gamma_{F}$.  Now, if $\M$ is an object in ${\sf
Qcoh }X$, there exists an epimorphism from a flat object
$\mathcal{L}$ in ${\sf Qcoh }X$ to $\M$.  This epimorphism induces
a map from (\ref{eqn.finaldig1}) applied to $\mathcal{L}$ to
(\ref{eqn.finaldig1}) applied to $\M$.  Since all the arrows in
(\ref{eqn.finaldig1}) are natural and the induced map
$F\mathcal{L} \rightarrow F\M$ is an epimorphism, the commutativity
of (\ref{eqn.finaldig1}) applied to $\M$ follows from a routine
diagram chase.
\newline
{\it Step 3:  Consider the following diagram
\begin{equation} \label{eqn.finaldig2}
\begin{CD}
Fu_{k*} & \longrightarrow & Fu_{k*}u_{ik*}^{k}u_{ik}^{k*} \\
@VVV @VV{=}V \\
Fu_{i*}u_{i}^{*}u_{k*} & \longrightarrow &
Fu_{i*}u_{ik*}^{i}u_{ik}^{k*} \\
@VVV @VVV \\
u_{i}^{*}u_{k*}(-)\otimes_{\mathcal{O}_{U_{i}}}\F_{i} &
\longrightarrow &
u_{ik*}^{i}u_{ik}^{k*}(-)\otimes_{\mathcal{O}_{U_{i}}} \F_{i} \\
@VVV @VVV \\
u_{k*}(-)\otimes_{\mathcal{O}_{X}}v_{i*}\F_{i} & &
u_{ik}^{k*}(-)\otimes_{\mathcal{O}_{U_{ik}}}v_{ik}^{i*}\F_{i} \\
@VVV @VVV \\
-\otimes_{\mathcal{O}_{U_{k}}}v_{k}^{*}v_{i*}\F_{i} &
\longrightarrow &
-\otimes_{\mathcal{O}_{U_{k}}}v_{ik*}^{k}v_{ik}^{i*}\F_{i}
\end{CD}
\end{equation}
whose top horizontal and top-left vertical are unit morphisms,
whose second verticals are from Propostion \ref{prop.wattone},
whose second, third and fourth horizontal are induced by
basechange, whose third left-vertical and bottom right-vertical
are induced by (\ref{eqn.pullback2}), and whose bottom-left
vertical and third right-vertical are induced by
(\ref{eqn.pullback}).  Then this diagram commutes.} The proof of the commutativity of
the top square of (\ref{eqn.finaldig2}) is routine and left to the
reader.  The commutativity of the middle square of
(\ref{eqn.finaldig2}) follows from the fact that the second
verticals are induced by the same natural transformations.  The
fact that the bottom rectangle in (\ref{eqn.finaldig2}) commutes
follows from Lemma \ref{lemma.canonicalcompatwithbasechange}.
\newline
{\it Step 4:  We complete the proof of the proposition.}
Recall that $\mathcal{E}_{i} \in {\sf Qcoh }U_{ik} \times Y$ denotes the object
corresponding to the functor $Fu_{ik*} \in {\sf Bimod}_{k}(U_{ik}-Y)$ in the proof
of Proposition \ref{prop.wattone}.  Consider the following commutative diagram
\begin{equation} \label{eqn.finaldig4}
\begin{CD}
-\otimes_{\mathcal{O}_{U_{k}}}W_{\mathfrak{U}\cap U_{k}}(Fu_{k*})
& \longrightarrow &
\oplus_{i}-\otimes_{\mathcal{O}_{U_{k}}}v_{ik*}^{k}\E_{i} &
\longrightarrow &
\oplus_{i}u_{ik}^{k*}(-)\otimes_{\mathcal{O}_{U_{ik}}}\E_{i} \\
@A{\Gamma_{Fu_{k*}}}AA & & @VVV \\
Fu_{k*} & \underset{=}{\longrightarrow} & Fu_{k*} &
\longrightarrow & \oplus_{i}Fu_{k*}u_{ik*}^{k}u_{ik}^{k*} \\
& & @VVV  @VV{=}V \\
& & \oplus_{i}Fu_{i*}u_{i}^{*}u_{k*} & &
\oplus_{i}Fu_{i*}u_{ik*}^{i}u_{ik}^{k*} \\
@V{\Gamma_{F}*u_{k*}}VV @VVV @VVV \\
& &
\oplus_{i}u_{i}^{*}u_{k*}(-)\otimes_{\mathcal{O}_{U_{i}}}\F_{i} &
 & \oplus_{i}
u_{ik*}^{i}u_{ik}^{k*}(-)\otimes_{\mathcal{O}_{U_{i}}}\F_{i} \\
& & @VVV @VVV \\
u_{k*}(-)\otimes_{\mathcal{O}_{X}}W_{\mathfrak{U}}(F) &
\longrightarrow &
\oplus_{i}u_{k*}(-)\otimes_{\mathcal{O}_{X}}v_{i*}\F_{i} & &
\oplus_{i}u_{ik}^{k*}(-)
\otimes_{\mathcal{O}_{U_{ik}}}v_{ik}^{i*}\F_{i} \\
@VVV @VVV  @VVV \\
-\otimes_{\mathcal{O}_{U_{k}}}v_{k}^{*}W_{\mathfrak{U}}(F) &
\longrightarrow & \oplus_{i}
-\otimes_{\mathcal{O}_{U_{k}}}v_{k}^{*}v_{i*}\F_{i} &
\longrightarrow &
\oplus_{i}-\otimes_{\mathcal{O}_{U_{k}}}v_{ik*}^{k}v_{ik}^{i*}\F_{i}
\end{CD}
\end{equation}
whose upper and middle-left rectangle are (\ref{eqn.finaldig1}),
whose lower-right rectangle is (\ref{eqn.finaldig2}) and whose
lower-left square has verticals induced by (\ref{eqn.pullback})
and horizontals induced by the inclusion
\begin{equation} \label{eqn.inclusion}
W_{\mathfrak{U}}(F) \longrightarrow \oplus_{i} v_{i*}\F_{i}.
\end{equation}
It follows from Step 2, Step 3, and the naturality of
(\ref{eqn.pullback}) that all squares in this diagram commute.

Next, we consider the following commutative diagram
\begin{equation} \label{eqn.finaldig5}
\begin{CD}
-\otimes_{\mathcal{O}_{U_{k}}}W_{\mathfrak{U}\cap U_{k}}(Fu_{k*})
& \overset{=}{\longrightarrow} &
-\otimes_{\mathcal{O}_{U_{k}}}W_{\mathfrak{U}\cap U_{k}}(Fu_{k*})
\\
& & @AA{\Gamma_{Fu_{k*}}}A \\
& & Fu_{k*} \\
@V{-\otimes_{\mathcal{O}_{U_{k}}}\rho}VV @VV{\Gamma_{F}*u_{k*}}V \\
& & u_{k*}(-) \otimes_{\mathcal{O}_{X}} W_{\mathfrak{U}}(F) \\
& & @VVV \\
-\otimes_{\mathfrak{O}_{U_{k}}}v_{k}^{*}W_{\mathfrak{U}}(F) &
\underset{=}{\longrightarrow} &
-\otimes_{\mathfrak{O}_{U_{k}}}v_{k}^{*}W_{\mathfrak{U}}(F)
\end{CD}
\end{equation}
whose bottom-right vertical is induced by (\ref{eqn.pullback}).
The outside of the diagram formed by placing this diagram to
the left of (\ref{eqn.finaldig4}) commutes by Step 6 of Proposition \ref{prop.loc}. Since
(\ref{eqn.finaldig5}) equals (\ref{eqn.affineloc}), and since the
map
$$
-\otimes_{\mathcal{O}_{U_{k}}}v_{k}^{*}W_{\mathfrak{U}}(F)
\longrightarrow
\oplus_{i}(-)\otimes_{\mathcal{O}_{U_{k}}}v_{k}^{*}v_{i*}\F_{i}
$$
induced by (\ref{eqn.inclusion}) is monic on flat objects, we
conclude, by a straightforward diagram chase on the diagram
constructed by placing (\ref{eqn.finaldig5}) to the left of
(\ref{eqn.finaldig4}), that (\ref{eqn.affineloc}) commutes on flat
objects. The proposition follows from Step 1.
\end{proof}

\begin{cor} \label{cor.totallyglobal}
If $F$ is an object of ${\sf Bimod}_{k}(X-Y)$ then
$\operatorname{ker }\Gamma_{F}$ and $\operatorname{cok
}\Gamma_{F}$ are totally global.  In particular, if $X$ is affine,
then $\Gamma_{F}$ is an isomorphism.
\end{cor}

\begin{proof}
By Proposition \ref{prop.single}, it suffices to show that
$(\operatorname{ker }\Gamma_{F})u_{i*}$ and $(\operatorname{cok
}\Gamma_{F})u_{i*}$ equal $0$ for all $i$.  To this end, we
compute

\begin{eqnarray*}
(\operatorname{ker }\Gamma_{F})u_{i*} & = & \operatorname{ker
}(\Gamma_{F}*u_{i*}) \\
& \cong & \operatorname{ker }\Gamma_{Fu_{i*}} \\
& = & 0
\end{eqnarray*}
where the second line follows from Proposition
\ref{prop.compatwithaffine}, and the third follows from the fact
that since $Fu_{i*} \cong -\otimes_{\mathcal{O}_{U_{i}}}\F_{i}$ by
Proposition \ref{prop.wattone}, $\Gamma_{Fu_{i*}}$ is an
isomorphism by Proposition \ref{prop.recover}.

A similar proof establishes the fact that $\operatorname{cok
}\Gamma_{F}$ is totally global.

The last statement follows from the fact that if $X$ is affine, every totally global functor from ${\sf Qcoh }X$ is $0$.
\end{proof}

From now on, we fix a finite affine open cover $\mathfrak{U}$ of
$X$ and write $W$ for $W_{\mathfrak{U}}$.

%By Lemma \ref{lemma.localizing}, we may form the quotient category
%$$
%\Funct/{\sf TG}.
%$$
%Let
%$$
%\pi: \Funct \rightarrow \Funct/{\sf TG}
%$$
%denote the quotient functor.  The exactness of $\pi$ together with
%Corollary \ref{cor.totallyglobal} implies the following na\"{\i}ve
%version of the Eilenberg-Watts Theorem:

%\begin{cor}
%If $F \in {\sf Bimod}_{k}(X-Y)$, then $\pi(\Gamma_{F})$ is an
%isomorphism, i.e.
%$$
%\pi F \cong \pi(-\otimes_{\mathcal{O}_{X}}W(F))
%$$ in $\Funct/{\sf TG}$.
%\end{cor}

%The next fact will be used in the following section.

%\begin{lemma} \label{lemma.functorpres}
%The functor $\operatorname{cok }\Gamma_{F}$ commutes with direct
%limits.
%\end{lemma}

%\begin{proof}
%There is an exact sequence
%$$
%F \overset{\Gamma_{F}}{\longrightarrow}
%-\otimes_{\mathcal{O}_{X}}W(F) \longrightarrow {\operatorname{cok
%}}\Gamma_{F} \longrightarrow 0.
%$$
%in ${\sf Funct}_{k}({\sf Qcoh }X,{\sf Qcoh }Y)$.  The proof
%follows from the fact that both $F$ and
%$-\otimes_{\mathcal{O}_{X}}W(F)$ commute with direct limits.
%\end{proof}

\begin{cor} \label{thm.cokvanish}
If $F$ is an exact functor in ${\sf Bimod}_{k}(X-Y)$, then $\Gamma_{F}$ is an isomorphism.
\end{cor}

\begin{proof}
Let $\M$ be a quasi-coherent $\mathcal{O}_{X}$-module and let
$$
\delta_{\M}:\oplus_{i}u_{i*}u_{i}^{*}\M \longrightarrow
\oplus_{i<j}u_{i*}u_{ij*}^{i}u_{ij}^{i*}u_{i}^{*}\M
$$
denote the morphism defined by (\ref{eqn.delta1}).  By Proposition \ref{prop.compatwithaffine}, the natural transformation $\Gamma_{F}$
applied to each term of $\delta_{\M}$ is an isomorphism.  Thus, the canonical morphism
$\operatorname{ker }F(\delta_{\M}) \longrightarrow \operatorname{ker }(\delta_{\M}\otimes_{\mathcal{O}_{X}}W(F))$ is an isomorphism.  Since $F$ is
exact, the canonical morphism $F(\operatorname{ker }\delta_{\M}) \longrightarrow \operatorname{ker }F(\delta_{\M})$ is an isomorphism.  On the other hand,
by Corollary \ref{cor.fexact}, $F$ exact implies that $-\otimes_{\mathcal{O}_{X}}W(F)$ is left exact.  Therefore, the canonical morphism
$(\operatorname{ker }\delta_{\M}) \otimes_{\mathcal{O}_{X}}W(F) \longrightarrow \operatorname{ker }(\delta_{\M} \otimes_{\mathcal{O}_{X}} W(F))$ is an isomorphism.
The result now follows from Lemma \ref{lemma.descent}.
\end{proof}

\section{A Structure Theorem for Totally Global Functors in ${\sf bimod}_{k}(\mathbb{P}^{1}-\mathbb{P}^{0})$}
\label{section.structure}

The purpose of this section is to compute the structure of totally global functors in ${\sf
bimod}_{k}(\mathbb{P}^{1}-\mathbb{P}^{0})$ when $k$ is
algebraically closed.

Throughout this section, we let $k$ be an algebraically closed
field, we assume $X$ and $Y$ are noetherian, and we let
$$
{\sf funct}_{k}({\sf Qcoh}X,{\sf Qcoh}Y)
$$
denote the category of $k$-linear functors from ${\sf Qcoh }X$ to
${\sf Qcoh}Y$ which take coherent objects to coherent objects.  If
$F$ is an object of ${\sf funct}_{k}({\sf Qcoh}X,{\sf Qcoh}Y)$, we
let $F|_{{\sf coh} X}$ denote the restriction of $F$ to the full
subcategory of ${\sf Qcoh }X$ consisting of coherent objects.

In order to simplify the exposition, we introduce the concept of
an {\it admissible} functor.

\begin{defn} \label{defn.admissible}
Suppose $X$ is a projective variety with very ample invertible
sheaf $\mathcal{O}(1)$.  A nonzero object $F$ in ${\sf
funct}_{k}({\sf Qcoh }X, {\sf Qcoh }Y)$ is called an {\it
admissible functor} if it
\begin{enumerate}
\item is totally global \item is half-exact on vector-bundles,
\item commutes with direct limits, and \item has the property that $F\alpha$ is epic for
all nonzero $\alpha \in \Hom(\mathcal{O}(m),\mathcal{O}(n))$.
\end{enumerate}
\end{defn}
For $i \in \mathbb{Z}$, the functor $H^{1}(\mathbb{P}^{1}, (-)(i))$ is admissible.

Our main result in this section (Proposition
\ref{propo.split}, Corollary \ref{cor.extensionsplit}) is that an admissible functor $F \in {\sf
funct}_{k}({\sf Qcoh }\mathbb{P}^{1}, {\sf Qcoh }\mathbb{P}^{0})$
admits a split monic
$$
\Delta: H^{1}(\mathbb{P}^{1}, (-)(i)) \longrightarrow F
$$
for some $i \in \mathbb{Z}$.  This
allows us to prove (Theorem \ref{thm.structureoftg}) that every
admissible functor in ${\sf funct}_{k}({\sf Qcoh }\mathbb{P}^{1},
{\sf Qcoh }\mathbb{P}^{0})$ is a direct sum of cohomologies.  Since a non-zero,
totally global functor $F \in {\sf bimod}_{k}(\mathbb{P}^{1}-\mathbb{P}^{0})$ is admissible (Corollary
\ref{cor.bimodadmis}), the same holds for such functors.

\begin{lemma} \label{lemma.babydim}
Let $X$ be a projective variety with very ample invertible sheaf
$\mathcal{O}(1)$, and suppose $F \in {\sf funct}_{k}({\sf Qcoh}X, {\sf
Qcoh}\mathbb{P}^{0})$ is right exact and vanishes on coherent torsion modules.
Then $F$ satisfies (4) in Definition \ref{defn.admissible}.
\end{lemma}

\begin{proof}
If $\alpha \in \Hom(\mathcal{O}(m),\mathcal{O}(n))$ and either $F\mathcal{O}(n)=0$ or $m > n$, then
$\alpha=0$.  If $m=n$, and $\alpha$ is not zero, them $\alpha$ is an isomorphism so that $F\alpha$ is epic.
Thus, suppose $F\mathcal{O}(n) \neq 0$, let $m < n$ and let $\alpha \in
\Hom(\mathcal{O}(m),\mathcal{O}(n))$ be nonzero.  We first show
that the  kernel of $\alpha$ must be zero.  If not, pick an affine
open cover over which both $\mathcal{O}(m)$ and $\mathcal{O}(n)$
are free.  Over one of these sets, $U$, $\operatorname{ker
}\alpha$ is nonzero.  Since $\alpha(U)$ is just multiplication by
some element of $\mathcal{O}(U)$, and since $X$ is integral,
$\alpha(U)$ must be the zero map.  Therefore, $U \subset
\operatorname{Supp }\operatorname{ker}\alpha$. On the other hand,
since $\operatorname{ker }\alpha$ is coherent, its support is
closed in $X$.  Since $X$ is integral, the support of
$\operatorname{ker }\alpha$ must equal $X$.  But the support of
$\operatorname{ker }\alpha$ is disjoint from the set of points $p
\in X$ such that $\alpha_{p} \neq 0$, since this map is just
multiplication by a nonzero element of a domain.  We conclude that
the kernel of $\alpha$ equals $0$.

The cokernel of $\alpha_{p}$ is a torsion
$\mathcal{O}_{X,p}$-module for all $p$.   We conclude that the
cokernel of $\alpha$ is torsion.  Therefore, there is an exact
sequence
$$
0 \rightarrow \mathcal{O}(m) \overset{\alpha}{\rightarrow}
\mathcal{O}(n) \rightarrow \mathcal{T} \rightarrow 0
$$
with $\mathcal{T}$ torsion.  Hence $\operatorname{dim
}F\mathcal{O}(m)  \geq \operatorname{dim }F\mathcal{O}(n)$ by the
right exactness of $F$ and by the fact that $F \mathcal{T}=0$.
\end{proof}

\begin{cor} \label{cor.bimodadmis}
If $F \in {\sf bimod}_{k}(\mathbb{P}^{1}-\mathbb{P}^{0})$ is
non-zero and totally global, then $F$ is admissible.
\end{cor}

\begin{proof} Since $F \in {\sf funct}_{k}({\sf Qcoh
}\mathbb{P}^{1},{\sf Qcoh }\mathbb{P}^{0})$ is totally global, $F$
vanishes on coherent torsion modules by Lemma \ref{lemma.totglob1}.  Therefore, $F$ is admissible by Lemma
\ref{lemma.babydim}.
\end{proof}

\subsection{Subfunctors of Admissible Functors}
In this subsection we prove that if $F \in {\sf funct}_{k}({\sf
Qcoh }\mathbb{P}^{1}, {\sf Qcoh }\mathbb{P}^{0})$ is admissible,
it has a subfunctor isomorphic to $H^{1}(\mathbb{P}^{1},(-)(i))$
for some integer $i$.  We begin with some preliminary results.

\begin{lemma} \label{lemma.dim}
Let $X$ be a projective variety with very ample invertible sheaf
$\mathcal{O}(1)$ such that for all $i>0$, we have
$$
\operatorname{dim }_{k} \Gamma(X,\mathcal{O}(i)) > 1.
$$
If $F \in {\sf funct}_{k}({\sf Qcoh}X, {\sf
Qcoh}\mathbb{P}^{0})$ satisfies (4) in Definition \ref{defn.admissible} and
$F\mathcal{O}(n) \neq 0$ for some $n \in \mathbb{Z}$, then
$$
\operatorname{dim}_{k} F\mathcal{O}(m) > \operatorname{dim }_{k}
F\mathcal{O}(n)
$$
for all $m < n$.
\end{lemma}

\begin{proof}
Let $n$ be such that $F\mathcal{O}(n) \neq 0$ and suppose
that for all nonzero $\alpha \in
\Hom(\mathcal{O}(m),\mathcal{O}(n))$ with $m < n$ we have $F
\alpha$ epic.  To prove the assertion, we must exclude the
possibility that there exists some $m < n$ such that
$\operatorname{dim}_{k} F\mathcal{O}(m) = \operatorname{dim }_{k}
F\mathcal{O}(n)$. Suppose to the contrary that for some $m < n$,
$\operatorname{dim}_{k} F\mathcal{O}(m) = d =\operatorname{dim
}_{k} F\mathcal{O}(n) \neq 0$.  Then, for all nonzero $\alpha \in
\Hom(\mathcal{O}(m),\mathcal{O}(n))$, $F \alpha$ is an
isomorphism.  Pick a basis $\alpha_{0}, \ldots, \alpha_{r}$ for
$\Hom(\mathcal{O}(m),\mathcal{O}(n))$ and let $x_{0},\ldots,x_{r}$
denote indeterminates.  Note that by hypothesis, $r>0$.  Since
$$
\operatorname{det }(x_{0}F\alpha_{0}+ \cdots + x_{r}F\alpha_{r})
$$
is a homogeneous polynomial of degree $d > 0$ in
$k[x_{0},\ldots,x_{r}]$, it has a non-trivial zero which then
gives a non-zero $\alpha$ such that $F\alpha$ is not invertible.
This is a contradiction.
\end{proof}

The following lemma will be invoked in the proof of Proposition
\ref{propo.split}.  Its straightforward proof is omitted.

\begin{lemma} \label{lemma.extension}
Suppose $F_{1}, F_{2} \in {\sf funct}_{k}({\sf Qcoh }X, {\sf Qcoh
}Y)$ preserve direct limits.

If $\underline{\Delta}:F_{1}|_{{\sf coh }X} \longrightarrow
F_{2}|_{{\sf coh }X}$ is a natural transformation, then
$\underline{\Delta}$ extends uniquely to a natural transformation
$\Delta:F_{1} \longrightarrow F_{2}$.  If $\underline{\Delta}$ is
monic, i.e. if $\underline{\Delta}_{\M}$ is monic for all coherent
objects $\M$ in ${\sf Qcoh }X$, then $\Delta$ is monic in ${\sf
Funct}_{k}({\sf Qcoh }X,{\sf Qcoh }Y)$.  If $\underline{\Delta}$
is epic, then $\Delta$ is epic in ${\sf Funct}_{k}({\sf Qcoh
}X,{\sf Qcoh }Y)$.
\end{lemma}

We introduce notation which will be used in the proof of
Proposition \ref{propo.split}: let $A=k[x_{0},x_{1}]$ denote the
polynomial ring in $2$ variables with its usual grading, let $[-]$ denote the shift functor, and let
$f_{i}:A[-(n+1)] \rightarrow A[-n]$ and $g_{i}:A[-(n+2)]
\rightarrow A[-(n+1)]$ denote multiplication by $x_{i}$.  Then we
have a short exact sequence in ${\sf Gr }A$:
$$
0 \longrightarrow A[-(n+2)]
\overset{(g_{1},-g_{0})}{\longrightarrow} A[-(n+1)]^{\oplus 2}
\overset{f_{0}+f_{1}}{\longrightarrow} A[-n] \longrightarrow k[-n]
\longrightarrow 0
$$
where $k$ denotes the trivial module.  This induces the short exact sequence
\begin{equation} \label{eqn.tangent}
0 \longrightarrow \mathcal{O}(-(n+2))
\overset{(\phi_{1},-\phi_{0})}{\longrightarrow}
\mathcal{O}(-(n+1))^{\oplus 2}
\overset{\psi_{0}+\psi_{1}}{\longrightarrow} \mathcal{O}(-n)
\longrightarrow 0.
\end{equation}

\begin{prop} \label{propo.split}
Suppose $F \in {\sf funct}_{k}({\sf Qcoh }\mathbb{P}^{1},{\sf Qcoh
}\mathbb{P}^{0})$ is admissible.  Then the set
$$
\{i \in \mathbb{Z} | F\mathcal{O}(i) \neq 0 \}
$$
has a maximum, $r$, and there is a monic
morphism
$$
\Delta:H^{1}(\mathbb{P}^{1},(-)(-2-r)) \rightarrow F
$$
in ${\sf Funct}_{k}({\sf Qcoh }{\mathbb{P}}^{1}, {\sf Qcoh
}\mathbb{P}^{0})$.
\end{prop}

\begin{proof}
We first show that $r$ is well defined.  Since $F$ is non-zero and
totally global,  $F\mathcal{O}(n) \neq 0$ for some $n$.  Then
$\operatorname{dim }F\mathcal{O}(n) > \operatorname{dim
}F\mathcal{O}(n+1)$ by Lemma \ref{lemma.dim}, so $F \mathcal{O}(i)
= 0 $ for all $i>>0$.  Hence, the set $\{i \in \mathbb{Z} |
F\mathcal{O}(i) \neq 0 \}$ indeed has a maximum.

We let $H := H^{1}(\mathbb{P}^{1},(-)(-2-r))$ and note that
$H\mathcal{O}(r) = H^{1}(\mathbb{P}^{1},\mathcal{O}(-2)) \cong k$.
We first define a natural transformation
$\underline{\Delta}:H|_{{\sf coh }\mathbb{P}^{1}} \rightarrow
F|_{{\sf coh}\mathbb{P}^{1}}$ by defining $\underline{\Delta}_{\mathcal{F}}$ for each
indecomposable coherent sheaf $\mathcal{F}$.  If $\mathcal{F}$ is
torsion or isomorphic to $\mathcal{O}(i)$ with $i
> r$ we define $\underline{\Delta}_{\mathcal{F}}=0$, and we define
$\underline{\Delta}_{\mathcal{O}(r)}: H\mathcal{O}(r) \rightarrow
F\mathcal{O}(r)$ to be any nonzero map.  Now suppose we have
defined $\underline{\Delta}_{\mathcal{O}(i)}$ for all $i > m$ such
that each such $\underline{\Delta}_{\mathcal{O}(i)}$ is injective
and such that
$$
\begin{CD}
H\mathcal{O}(j) & \overset{H \psi}{\longrightarrow} & H\mathcal{O}(j+1) \\
@V{\underline{\Delta}_{\mathcal{O}(j)}}VV @VV{\underline{\Delta}_{\mathcal{O}(j+1)}}V \\
F\mathcal{O}(j) & \overset{F \psi}{\longrightarrow} & F\mathcal{O}(j+1)
\end{CD}
$$
commutes for $j \geq i$ and $\psi \in
\operatorname{Hom}_{\mathbb{P}^{1}}(\mathcal{O}(j),\mathcal{O}(j+1))$.
 We construct an injective homomorphism $\theta:H\mathcal{O}(m)
\rightarrow F\mathcal{O}(m)$ such that
\begin{equation} \label{eqn.theta1}
\begin{CD}
H\mathcal{O}(m) & \overset{H \phi}{\longrightarrow} & H\mathcal{O}(m+1) \\
@V{\theta}VV @VV{\underline{\Delta}_{\mathcal{O}(m+1)}}V \\
F\mathcal{O}(m) & \overset{F \phi}{\longrightarrow} &
F\mathcal{O}(m+1)
\end{CD}
\end{equation}
commutes for $\phi=\phi_{0},\phi_{1}$ (see (\ref{eqn.tangent}) for a
definition of these maps). To this end, we apply both $H$ and $F$
to the exact sequence (\ref{eqn.tangent}) with $n:= -m-2$ to get a
diagram
\begin{equation} \label{eqn.kerwork}
\begin{CD}
H\mathcal{O}(m) &
\overset{(H\phi_{1},-H\phi_{0})}{\longrightarrow} &
H\mathcal{O}(m+1)^{\oplus 2} &
\overset{H\psi_{0}+H\psi_{1}}{\longrightarrow} & H\mathcal{O}(m+2)
\\
&  & @VV{\underline{\Delta}_{\mathcal{O}(m+1)}^{\oplus 2}}V
@VV{\underline{\Delta}_{\mathcal{O}(m+2)}}V \\
F\mathcal{O}(m) &
\overset{(F\phi_{1},-F\phi_{0})}{\longrightarrow} &
F\mathcal{O}(m+1)^{\oplus 2} &
\overset{F\psi_{0}+F\psi_{1}}{\longrightarrow} & F\mathcal{O}(m+2)
\end{CD}
\end{equation}
with exact rows whose right square commutes.

To construct $\theta$, choose a basis $u_{1}, \ldots, u_{r-m+1}$
for $H \mathcal{O}(m)$.  Now,
$$
(H\phi_{1},-H\phi_{0})(u_{i}) \in \operatorname{ker
}(H\psi_{0}+H\psi_{1}).
$$
Thus, by the commutativity of the right-hand square of
(\ref{eqn.kerwork}),
$$
(\underline{\Delta}_{\mathcal{O}(m+1)},\underline{\Delta}_{\mathcal{O}(m+1)})(H\phi_{1}(u_{i}),-H\phi_{0}(u_{i}))
$$
is in the image of $(F\phi_{1},-F\phi_{0})$.  Hence, there exists
a $v_{i} \in F\mathcal{O}(m)$ such that
$F\phi_{j}(v_{i})=\underline{\Delta}_{\mathcal{O}(m+1)}H\phi_{j}(u_{i})$
for $i=1, \ldots, r-m+1$ and $j=0,1$.  We define
$\theta(u_{i})=v_{i}$. Since $F$ is $k$-linear, we conclude that
(\ref{eqn.theta1}) commutes for all $\phi \in
\operatorname{Hom}_{\mathbb{P}^{1}}(\mathcal{O}(m),\mathcal{O}(m+1))$.
We define $\underline{\Delta}_{\mathcal{O}(m)} := \theta$, and we note that
$\underline{\Delta}_{\mathcal{O}(m)}$ is monic since $(H\phi_{1},-H\phi_{0})$ is monic.

Next, we define $\underline{\Delta}_{\mathcal{F}}$ when
$\mathcal{F}$ is isomorphic to $\mathcal{O}(n)$.  Let $\alpha:
\mathcal{F} \rightarrow \mathcal{O}(n)$ be an isomorphism.  Define
$$
\underline{\Delta}_{\mathcal{F}}:=(F \alpha)^{-1} \circ
\underline{\Delta}_{\mathcal{O}(n)} \circ H \alpha.
$$
If $\beta: \mathcal{F} \rightarrow \mathcal{O}(n)$ is another
isomorphism, then $\beta = \lambda \alpha$ for some $0 \neq
\lambda \in k$, whence $(F \beta)^{-1}=\lambda^{-1}(F
\alpha)^{-1}$ and $H \beta = \lambda H \alpha$; thus the
definition of $\delta_{\mathcal{F}}$ does not depend on the choice
of $\alpha$.

We now define $\underline{\Delta}_{\mathcal{F}}$ for arbitrary
$\mathcal{F}$ by writing $\mathcal{F}$ as a direct sum of
indecomposables, say $\mathcal{F}= \oplus \mathcal{F}_{i}$, and
defining $\underline{\Delta}_{\mathcal{F}}:=\oplus
\underline{\Delta}_{{\mathcal{F}}_{i}}$.

To show that $\underline{\Delta}$ is a natural transformation we
must show that
\begin{equation} \label{eqn.welldefined}
\begin{CD}
H \mathcal{F} & \overset{Hf}{\longrightarrow} & H \mathcal{G} \\
@V{\underline{\Delta}_{\mathcal{F}}}VV @VV{\underline{\Delta}_{\mathcal{G}}}V \\
F \mathcal{F} & \underset{Ff}{\longrightarrow} & F \mathcal{G}
\end{CD}
\end{equation}
commutes for all $\mathcal{F}$ and $\mathcal{G}$ and all maps
$f:\mathcal{F} \rightarrow \mathcal{G}$.  It suffices to check
this when $\mathcal{F}$ and $\mathcal{G}$ are indecomposable.  The
diagram commutes when $\mathcal{G}$ is torsion because $F
\mathcal{G}=0$ then.  If $\mathcal{G}$ is torsion-free and
$\mathcal{F}$ torsion, then $f=0$ so the diagram commutes.  Thus,
the only remaining case is that when $\mathcal{F} \cong
\mathcal{O}(i)$ and $\mathcal{G} \cong \mathcal{O}(j)$ with $i
\leq j$.  The case $i=j$ is straightforward and we omit the verification in this case.
Thus, we may suppose $i>j$.

Write $f = \beta^{-1} g \alpha$ were $\alpha: \mathcal{F}
\rightarrow \mathcal{O}(i)$ and $\beta: \mathcal{G} \rightarrow
\mathcal{O}(j)$ are isomorphisms and $0 \neq g: \mathcal{O}(i)
\rightarrow \mathcal{O}(j)$.  We can write $g$ as a sum of terms of the form $\psi_{j}\psi_{j-1}
\cdots \psi_{i+1}$ where each $\psi_{l}:\mathcal{O}(l-1)
\rightarrow \mathcal{O}(l)$ is monic.  Now
$$
\underline{\Delta}_{\mathcal{O}(j)} \circ H\psi_{j} \circ
\cdots \circ H \psi_{i+1} = F \psi_{j} \circ \cdots \circ F
\psi_{i+1} \circ \underline{\Delta}_{\mathcal{O}(i)}
$$
and this implies
$$
\underline{\Delta}_{\mathcal{O}(j)} \circ
Hg = Fg \circ
\underline{\Delta}_{\mathcal{O}(i)}.
$$

Therefore,
\begin{eqnarray*}
\underline{\Delta}_{\mathcal{G}} \circ Hf & = & F \beta^{-1} \circ
\underline{\Delta}_{\mathcal{O}(j)} \circ H \beta \circ Hf \\
& = & F \beta^{-1} \circ
\underline{\Delta}_{\mathcal{O}(j)} \circ H g \circ H \alpha \\
& = & F \beta^{-1} \circ F g \circ
\underline{\Delta}_{\mathcal{O}(i)} \circ H \alpha  \\
& = & F f \circ F \alpha^{-1} \circ
\underline{\Delta}_{\mathcal{O}(i)} \circ H \alpha  \\
& = & Ff \circ \underline{\Delta}_{\mathcal{F}}.
\end{eqnarray*}

This shows that (\ref{eqn.welldefined}) commutes and so completes
the proof that $\Delta$ is natural.

Finally, $\underline{\Delta}_{\mathcal{F}}$ is monic for all
indecomposable coherent $\mathcal{F}$ and hence for all coherent
$\mathcal{F}$. It follows from Lemma \ref{lemma.extension} that
$\underline{\Delta}$ extends to a monic natural transformation
$$
\Delta: H \longrightarrow F.
$$
\end{proof}

\subsection{The Structure of Admissible Functors in ${\sf funct}_{k}({\sf Qcoh }\mathbb{P}^{1},{\sf Qcoh }\mathbb{P}^{0})$}
In this subsection, we work towards a proof, realized in Corollary
\ref{cor.extensionsplit}, that the monic $\Delta$ constructed in
Proposition \ref{propo.split} is split.  It follows (Theorem \ref{thm.structureoftg}) that an admissible functor in
${\sf funct}_{k}({\sf Qcoh }\mathbb{P}^{1},{\sf Qcoh }\mathbb{P}^{0})$ is a direct sum of cohomologies.  We assume, throughout the
subsection, that $X$ and $Y$ are projective schemes, $\F, \G, \M
\in {\sf Qcoh }X$ are coherent, and $F$ is an object of ${\sf
funct}_{k}({\sf Qcoh }X,{\sf Qcoh }Y)$.

We first define a natural transformation
$$
\Phi_{F}:  F|_{{\sf coh}X} \longrightarrow
\Hom(-,\mathcal{G})^{*}|_{{\sf coh }X}\otimes_{k} F \mathcal{G}
$$
which will be used to split the monic $\Delta$ constructed in
Proposition \ref{propo.split}.  To this end, we let
$$
\eta_{\F,\G}: k \longrightarrow \Hom_{X}(\F,\G)^{*} \otimes_{k}
\Hom_{X}(\F,\G)
$$
be defined as follows: $\eta_{\F,\G}(a):= a(\sum_{i}f_{i}^{*}
\otimes f_{i})$ where $\{f_{1},\ldots,f_{m}\}$ is a basis for
$\Hom_{X}(\F,\G)$.  We next note that the functor $F$ induces a
map
\begin{equation} \label{eqn.natural}
\phi_{\F,\G}: \Hom_{X}(\F,\G) \otimes_{k} F\F \longrightarrow F\G
\end{equation}
as follows:  if $U$ is an open set in $Y$, and $s \in F\F(U)$, we
define (\ref{eqn.natural}) over $U$ to be the map
$$
f \otimes s \mapsto F(f)(U)(s).
$$

We define the natural transformation
\begin{equation} \label{wattmap}
\Phi_{F}:  F|_{{\sf coh}X} \longrightarrow
\Hom(-,\mathcal{G})^{*}|_{{\sf coh }X}\otimes_{k} F \mathcal{G}.
\end{equation}
as follows:
$$
\Phi_{F\F}:F\F \longrightarrow
\Hom_{X}(\F,\mathcal{G})^{*}\otimes_{k} F \mathcal{G}
$$
is defined to be the composition of
$$
\eta_{\F,\mathcal{G}} \otimes_{k} F\F:F\F \longrightarrow
\Hom_{X}(\F,\mathcal{G})^{*} \otimes_{k} \Hom(\F,\G) \otimes_{k}
F\F
$$
with
$$
\Hom_{X}(\F,\G)^{*} \otimes_{k} \phi_{\F,\G}:\Hom_{X}(\F,\G)^{*}
\otimes_{k} \Hom_{X}(\F,\G) \otimes_{k} F\F \longrightarrow
$$
$$
\Hom_{X}(\F,\G)^{*} \otimes_{k}F\G.
$$
The proof that $\Phi_{F}$ is natural is straightforward and we
omit it.

\begin{lemma} \label{lemma.h1}
If $\N$ is a coherent object of ${\sf Qcoh }Y$, $\G$ is an
invertible $\mathcal{O}_{X}$-module and
$$
F=\Hom_{X}(-,\G)^{*} \otimes_{k} \N,
$$
then the morphism $\Phi_{F}$ is an isomorphism.
\end{lemma}

\begin{proof}
Let $\{f_{1},\ldots,f_{m}\}$ be a basis for $\Hom_{X}(\F,\G)$ and
let $U$ be open in $Y$.  Then
$$
\Phi_{F\F}(U):F\F(U) \longrightarrow
\Hom_{X}(\F,\G)^{*}\otimes_{k} F \G(U)
$$
sends $s \in F\F(U)$ to $\sum_{i=1}^{m}f_{i}^{*} \otimes
F(f_{i})(U)(s)$.

Suppose $s$ is a simple tensor, so
$$
s=\delta \otimes t \in \Hom_{X}(\F,\G)^{*} \otimes_{k} \N(U).
$$
We describe $F(f_{i})(U)(\delta \otimes t)$.  The map $f_{i}:\F
\longrightarrow \G$ induces the map
$$
- \circ f_{i}:\Hom_{X}(\G,\G) \longrightarrow \Hom_{X}(\F,\G).
$$
Dualizing gives a map
$$
\Hom_{X}(\F,\G)^{*} \longrightarrow \Hom_{X}(\G,\G)^{*}
$$
which sends $\delta$ to $\delta \circ (- \circ f_{i})$.  Therefore, $F(f_{i})(U)(\delta \otimes t)=\delta \circ (- \circ f_{i}) \otimes t$
and so the morphism
$$
\Phi_{F\F}(U):\Hom_{X}(\F,\G)^{*} \otimes_{k} \N(U) \longrightarrow \Hom_{X}(\F,\G)^{*} \otimes_{k}\Hom_{X}(\G,\G)^{*}\otimes_{k}\N(U)
$$
sends $\delta \otimes t$ to $\sum_{i}f_{i}^{*}\otimes (\delta \circ (- \circ f_{i})) \otimes t$.  Since the map
$\delta \circ (- \circ f_{i}) \in
\Hom_{X}(\G, \G)^{*} \cong k$ sends multiplication by $\alpha$ to
multiplication by $\alpha \delta(f_{i})$, the function
$$
\Hom_{X}(\F,\G)^{*} \longrightarrow \Hom_{X}(\F,\G)^{*}
\otimes_{k} \Hom_{X}(\G,\G)^{*}
$$
defined by sending $\delta$ to $\sum_{i}f_{i}^{*}
\otimes_{k}(\delta \circ (- \circ f_{i}))$ is injective and $k$-linear, hence
an isomorphism of vector spaces.  It follows that $\Phi_{F\F}(U)$ is a tensor product of two isomorphisms, and the assertion follows.
\end{proof}

\begin{lemma} \label{lemma.commutivity}
Let $\Theta:F' \longrightarrow F$ be a natural transformation
between elements of ${\sf funct}_{k}({\sf Qcoh }X,{\sf Qcoh }Y)$.
Then the diagram
$$
\begin{CD}
F & & & \overset{\Phi_{F}}{\longrightarrow} & \operatorname{Hom
}_{X}(-,\G)^{*} \otimes_{k} F\G \\
@A{\Theta}AA & & @AAA \\
F' & & & \underset{\Phi_{F'}}{\longrightarrow} & \operatorname{Hom
}_{X}(-,\G)^{*} \otimes_{k} F'\G
\end{CD}
$$
whose right vertical is induced by $\Theta$, commutes on coherent
objects.
\end{lemma}

\begin{proof}
From the definition of $\Phi$, it suffices to show that the
diagram
\begin{equation} \label{eqn.unit0}
\begin{CD}
F\M & & & {\longrightarrow} & \operatorname{Hom
}_{X}(\M,\G)^{*} \otimes_{k} \Hom_{X}(\M,\G) \otimes_{k} F \\
@A{\Theta_{\M}}AA & & @AAA \\
F'\M & & & {\longrightarrow} & \operatorname{Hom }_{X}(\M,\G)^{*}
\otimes_{k} \Hom_{X}(\M,\G) \otimes_{k} F'
\end{CD}
\end{equation}
whose right vertical is induced by $\Theta$ and whose horizontals
are induced by the unit map $k \longrightarrow \Hom_{X}(-,\G)^{*}
\otimes_{k} \Hom_{X}(-,\G)$ commutes, and that the diagram
\begin{equation} \label{eqn.unit}
\begin{CD}
\Hom_{X}(\mathcal{M},\G) \otimes_{k} F\mathcal{M}
& & & \longrightarrow & & F\G \\
@AAA & & & @AA{\Theta_{\G}}A \\
\Hom_{X}(\mathcal{M},\G) \otimes_{k} F'\mathcal{M} & & & \longrightarrow & & F'\G \\
\end{CD}
\end{equation}
whose left vertical is induced by $\Theta$ and whose horizontals
are induced by evaluation, commutes.  The fact that (\ref{eqn.unit0}) commutes is trivial.
  We check commutativity of
(\ref{eqn.unit}).  The top route of (\ref{eqn.unit}) evaluated on
the open set $U \subset Y$ sends $f \otimes x$ to
$F(f)(U)(\Theta_{\M}(U)(x))$ while the bottom route of
(\ref{eqn.unit}) sends $f \otimes x$ to $\Theta_{\G}(U)
F'(f)(U)(x)$. These values are equal by the naturality of
$\Theta$.
\end{proof}

\begin{lemma} \label{lemma.split}
If $F \in {\sf funct}_{k}({\sf Qcoh}X, {\sf Qcoh}\mathbb{P}^{0})$
is such that there exists an invertible $\G \in {\sf Qcoh }X$ and
a monomorphism
$$
\Psi:\operatorname{Hom }_{X}(-,\G)^{*} \longrightarrow F
$$
in ${\sf funct}_{k}({\sf Qcoh }X,{\sf Qcoh}\mathbb{P}^{0})$, then
the restriction of $\Psi$ to coherents,
$$
\underline{\Psi}:\operatorname{Hom}_{X}(-,\G)^{*}|_{{\sf coh }X}
\longrightarrow F|_{{\sf coh }X},
$$
splits.
\end{lemma}

\begin{proof}
Let $\psi: F\G \longrightarrow \operatorname{Hom }_{X}(\G,\G)^{*}$
be a splitting of $\Psi_{\G}$. Consider the diagram
$$
\begin{CD}
F\M  & \overset{\Phi_{F\M}}{\longrightarrow} & & & \operatorname{Hom
}_{X}(\M,\G)^{*} \otimes_{k}
F\G \\
@A{\Psi_{\M}}AA & & @AAA
\\
\operatorname{Hom }_{X}(\M,\G)^{*} &
\underset{\Phi_{\operatorname{Hom
}_{X}(-,\G)^{*}\M}}{\longrightarrow} & & & \operatorname{Hom
}_{X}(\M,\G)^{*} \otimes_{k} \operatorname{Hom }_{X}(\G,\G)^{*}
\end{CD}
$$
whose right vertical is induced by $\Psi_{\G}$.  The bottom
horizontal is an isomorphism by Lemma \ref{lemma.h1}, and the
diagram commutes by Lemma \ref{lemma.commutivity}.  It follows
that the diagram
$$
\begin{CD}
F\M  & \overset{\Phi_{F\M}}{\longrightarrow} & & & \operatorname{Hom
}_{X}(\M,\G)^{*} \otimes_{k}
F\G \\
@A{\Psi_{\M}}AA & & @VVV
\\
\operatorname{Hom }_{X}(\M,\G)^{*} &
\underset{\Phi_{\operatorname{Hom
}_{X}(-,\G)^{*}\M}}{\longrightarrow} & & & \operatorname{Hom
}_{X}(\M,\G)^{*} \otimes_{k} \operatorname{Hom }_{X}(\G,\G)^{*}
\end{CD}
$$
whose right vertical is induced by $\psi$, commutes.  The lemma follows.
\end{proof}

\begin{cor} \label{cor.extensionsplit}
Let $F \in {\sf funct}_{k}({\sf Qcoh }\mathbb{P}^{1}, {\sf Qcoh
}\mathbb{P}^{0})$ be admissible.  The monic
$$
\Delta:H^{1}(\mathbb{P}^{1},(-)(-2-r)) \longrightarrow F
$$
constructed in Proposition \ref{propo.split} splits.
\end{cor}

\begin{proof}
The monic $\Delta$ restricts to a monic
$$
\underline{\Delta}:H^{1}(\mathbb{P}^{1},(-)(-2-r))|_{{\sf
coh}\mathbb{P}^{1}} \longrightarrow F|_{{\sf coh}\mathbb{P}^{1}}.
$$
By Serre duality, $\underline{\Delta}$ induces a monic
$$
\underline{\Delta}':\operatorname{Hom}_{\mathbb{P}^{1}}(-,\mathcal{O}(r))^{*}|_{{\sf
coh}\mathbb{P}^{1}} \longrightarrow F|_{{\sf coh}\mathbb{P}^{1}}
$$
which by Lemma \ref{lemma.split}, admits a splitting
$$
\underline{\Psi}':F|_{{\sf coh}\mathbb{P}^{1}} \longrightarrow
\operatorname{Hom}_{\mathbb{P}^{1}}(-,\mathcal{O}(r))^{*}|_{{\sf
coh}\mathbb{P}^{1}}.
$$
The map $\underline{\Psi}'$ induces, by Serre duality again, a
splitting
$$
\underline{\Psi}:F|_{{\sf coh}\mathbb{P}^{1}} \longrightarrow
H^{1}(\mathbb{P}^{1},(-)(-2-r))|_{{\sf coh}\mathbb{P}^{1}}
$$
of $\underline{\Delta}$.  We claim that $\underline{\Psi}$ extends
to a splitting
$$
\Psi: F \longrightarrow H^{1}(\mathbb{P}^{1},(-)(-2-r))
$$
of $\Delta$.  To this end, Lemma \ref{lemma.extension} implies
that $\underline{\Psi}$ has a unique extension
$\Psi$. We also know that $\Psi \Delta$
restricts on coherent objects to the map $\underline{\Psi}\mbox{
}\underline{\Delta} =
\operatorname{id}_{H^{1}(\mathbb{P}^{1},(-)(-2-r))|_{{\sf
coh}\mathbb{P}^{1}}}$.  But by Lemma \ref{lemma.extension},
$\underline{\Psi}\mbox{ }\underline{\Delta}$ extends uniquely to a
natural transformation
$$
H^{1}(\mathbb{P}^{1},(-)(-2-r)) \longrightarrow
H^{1}(\mathbb{P}^{1},(-)(-2-r)).
$$
Thus, $\Psi
\Delta=\operatorname{id}_{H^{1}(\mathbb{P}^{1},(-)(-2-r))}$,
whence the Corollary.
\end{proof}

We omit the straightforward proof of the following
\begin{lemma} \label{lemma.condition}
Suppose $F \in {\sf funct}_{k}({\sf Qcoh }\mathbb{P}^{1},{\sf Qcoh
}\mathbb{P}^{0})$ is admissible.

If $F \cong A \oplus B$ in ${\sf funct}_{k}({\sf Qcoh
}\mathbb{P}^{1}, {\sf Qcoh }\mathbb{P}^{0})$, and if $A$ is
non-zero, then $A$ is admissible as well.
\end{lemma}

\begin{thm} \label{thm.structureoftg}
If $F \in {\sf funct}_{k}({\sf Qcoh }\mathbb{P}^{1},{\sf Qcoh
}\mathbb{P}^{0})$ is admissible, then there exist integers $m,
n_{i} \geq 0$ such that
$$
F \cong \oplus_{i=-m}^\infty {H}^{1}(\mathbb{P}^{1},(-)(i))^{\oplus n_{i}}.
$$
\end{thm}

\begin{proof}
Since $F$ is admissible, Proposition \ref{propo.split} implies that the set $\{i | F\mathcal{O}(i) \neq 0 \}$ has a maximum, $r$.  We let $m=r+2$.  Since $F$ preserves coherence, the set
$$
\{ n | \mbox{there exists a split monomorphism }H^{1}(\mathbb{P}^{1},(-)(-m))^{\oplus n} \rightarrow F \}
$$
has a maximum, which we call $n_{0}$.  If we let $F_{0}= H^{1}(\mathbb{P}^{1},(-)(-m))^{\oplus n_{0}}$, and we let $\delta_{0}:F_{0} \rightarrow F$ be a split monomorphism, then there is a sub-functor $F^{(1)}$ of $F$ such that $F \cong F_{0} \oplus F^{(1)}$.  By Lemma \ref{lemma.condition}, either $F^{(1)}$ is $0$ or $F^{(1)}$ is admissible.

Now, given a sub-functor $F^{(i)}$ of $F$ which is either $0$ or admissible, we construct an object $F_{i}$ in the category ${\sf funct}_{k}({\sf Qcoh }\mathbb{P}^{1},{\sf Qcoh
}\mathbb{P}^{0})$, a split monomorphism $\delta_{i}:F_{i} \rightarrow F^{(i)}$, and a sub-functor $F^{(i+1)}$ of $F^{(i)}$ which is either $0$ or admissible, as follows.  We let
$$
n_{i}=\mbox{max }\{ n | \mbox{there exists a split monomorphism }H^{1}(\mathbb{P}^{1},(-)(-m+i))^{\oplus n} \rightarrow F^{(i)} \},
$$
we let $F_{i}= H^{1}(\mathbb{P}^{1},(-)(-m+i))^{\oplus n_{i}}$, and we let $\delta_{i}:F_{i} \rightarrow F^{(i)}$ be a split monomorphism.  Then there is a sub-functor $F^{(i+1)}$ of $F^{(i)}$ such that $F^{(i)} \cong F_{i} \oplus F^{(i+1)}$.  By Lemma \ref{lemma.condition}, either $F^{(i+1)}$ is $0$, or $F^{(i+1)}$ is admissible.

In this way we get a morphism
$$
\Delta :\oplus_{i=0}^{\infty} {H}^{1}(\mathbb{P}^{1},(-)(-m+i))^{\oplus n_{i}} \rightarrow F.
$$
defined by $\Delta:=\oplus_{i=0}^{\infty}\delta_{i}$.  We claim that $\Delta$ is an isomorphism.  By Lemma \ref{lemma.extension}, it suffices to show that $\Delta|_{\sf{coh }\mathbb{P}^{1}}$ is an isomorphism.  To this end, let
$\mathcal{M}$ be a coherent $\mathcal{O}_{\mathbb{P}^{1}}$-module.  Then $\mathcal{M} \cong \mathcal{O}(i_{1}) \oplus \cdots \oplus \mathcal{O}(i_{n}) \oplus \mathcal{T}$ where $\mathcal{T}$ is coherent torsion and $i_{1}=\operatorname{min }\{i_{1},\ldots,i_{n}\}$.  It follows that
$$
\oplus_{i=0}^{\infty} {H}^{1}(\mathbb{P}^{1},\mathcal{M}(-m+i))^{\oplus n_{i}}=\oplus_{i=0}^{-2-i_{1}+m} {H}^{1}(\mathbb{P}^{1},\mathcal{M}(-m+i))^{\oplus n_{i}}.
$$
By the construction of $\Delta$, in order to show that $\Delta_{\mathcal{M}}$ is an isomorphism, it suffices to show that $F^{(-1-i_{1}+m)}(\mathcal{M})=0$.  If not, then $F^{(-1-i_{1}+m)}$ is an admissible direct summand of $F$.  By Proposition \ref{propo.split}, the set $\{i | F^{(-1-i_{1}+m)}\mathcal{O}(i) \neq 0 \}$ has a maximum, $s$, and there exists a split monomorphism ${H}^{1}(\mathbb{P}^{1},(-)(-2-s))\rightarrow F^{(-1-i_{1}+m)}$.  Since $F^{(-1-i_{1}+m)}$ is totally global and $F^{(-1-i_{1}+m)}(\mathcal{M})$ is nonzero, it follows that one of
$$
F^{(-1-i_{1}+m)}(\mathcal{O}(i_{1})),\ldots,F^{(-1-i_{1}+m)}(\mathcal{O}(i_{n}))
$$
is nonzero.  Hence $i_{1} \leq s$.  Since $m=r+2$ and $s \leq r$, it follows that $-m \leq -s-2$.  Thus, we have
$$
-m \leq -2-s \leq -2-i_{1}.
$$
This contradicts the maximality of $n_{-2-s+m}$.
\end{proof}

By Corollary \ref{cor.bimodadmis}, Theorem \ref{thm.structureoftg} immediately implies the following
\begin{cor} \label{cor.structureoftg}
If $F \in {\sf bimod}_{k}(\mathbb{P}^{1}-\mathbb{P}^{0})$ is totally global, then $F$ is a direct sum of
cohomologies, i.e. there exist integers $m,
n_{i} \geq 0$ such that
$$
F \cong \oplus_{i=-m}^\infty {H}^{1}(\mathbb{P}^{1},(-)(i))^{\oplus n_{i}}.
$$
\end{cor}

\end{document}